\documentclass[a4paper,9pt]{amsart}
 \usepackage[margin=3cm]{geometry} %layout

\usepackage{a4wide,enumerate,xcolor, appendix, lscape}
\usepackage[colorlinks, linkcolor=blue]{hyperref}
\usepackage{tikz}
\usepackage{tikz-cd}
  \usepackage{mathtools}
\usetikzlibrary{matrix}
\usetikzlibrary{positioning}

\usepackage{amssymb}
\allowdisplaybreaks
\numberwithin{equation}{section}
\usepackage{enumitem}
\usepackage{graphicx}
\usepackage{multicol}

\usepackage{caption}
\usepackage{float}
\usepackage{algpseudocode}
\definecolor{change}{rgb}{0,.55,.55}
\definecolor{ao(english)}{rgb}{0.0, 0.5, 0.0}

\let\eps\varepsilon  
\newcommand{\N}{{\mathbb N}}  
\newcommand{\R}{{\mathbb R}}

\newcommand{\dom}{{\mathcal O}}

\newcommand{\Law}{{\mathcal L}}
\newcommand{\Prob}{\mathbb{P}}
\newcommand{\E}{\mathbb{E}}

\newcommand{\Wd}{\mathcal{P}^{1}}

\newcommand{\indicator}{\mathbf{1}}
\newcommand{\idx}{\operatorname{id}_{x}}
\newcommand{\idxo}{x}

\renewcommand{\d}{\,\textnormal{d}}

\newcommand{\dt}{\,\textnormal{d}t}
\newcommand{\ds}{\,\textnormal{d}s}
\newcommand{\dr}{\,\textnormal{d}r}

\newcommand{\dx}{\,\textnormal{d}x}
\newcommand{\dy}{\,\textnormal{d}y}

\newcommand{\dW}{\,\textnormal{d}W}
\newcommand{\dB}{\,\textnormal{d}B}

\newcommand{\V}{V}

\newtheorem{theorem}{Theorem}  
\newtheorem{lemma}[theorem]{Lemma}   
\newtheorem{proposition}[theorem]{Proposition}   
\newtheorem{remark}[theorem]{Remark}   
  
\newtheorem{definition}[theorem]{Definition}  
 
\newtheorem{notation}[theorem]{Notation}  
\newtheorem{assumption}[theorem]{Assumption} 

\makeatletter
\@namedef{subjclassname@2020}{\textup{2020} Mathematics Subject Classification}
\makeatother

\title[Polynomial interacting particle systems and non-linear SPDEs]{Polynomial interacting particle systems and non-linear SPDEs for market capitalization curves}

\author{Christa Cuchiero \and Florian Huber}

\address{University of Vienna, Department of Statistics and Operations Research, Data Science @ Uni Vienna, Kolingasse 14-16, 1090 Wien, Austria}
\email{christa.cuchiero@univie.ac.at}
\thanks{The authors gratefully acknowledge financial support 
through grant Y 1235 of the START-program.}
\address{École Polytechnique Fédérale de Lausanne, Switzerland}
\email{florian.huber@epfl.ch}
\date{}


\begin{document}

\begin{abstract}
Motivated by the robustness of the capital distribution curves, we study the behavior of a certain polynomial equity market model as the number of companies goes to infinity. 
More precisely, we extend volatility-stabilized market models introduced by Fernholz et al.~ \cite{FK05_relative_arbitrage_volatility_stabilized}
by allowing for a common noise term such that the models remain polynomial. As the number of companies approaches infinity, we show that the limit of the empirical measure of the $N$-company system converges to the unique solution of a degenerate, non-linear SPDE. The obtained limit also has a representation as the conditional probability of the solution to a certain McKean-Vlasov SDE. Together with its conditional, this is again a polynomial process for which we can prove pathwise uniqueness as well as regularity properties for the marginal densities.
We also provide conditional propagation of chaos results and numerical implementations of the particle system as well as its limiting equations.
\end{abstract}

\maketitle
\tableofcontents

\section{Introduction}

The stability of the \emph{capital distribution curves} over time, as shown in Figure~\ref{fig:1},
\begin{figure}[ht]
\centering
\includegraphics[width=0.6\textwidth]{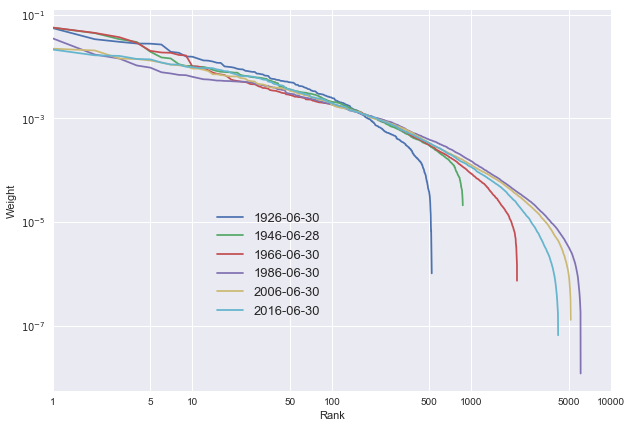}
\caption{Capital distribution curves: 1926 - 2016, source \cite{R:18}} \label{fig:1}
\end{figure}
can be seen as a universal phenomenon in finance.
By this, we here mean a robust empirical feature that holds universally
across different markets, asset classes, and in particular over time.
Each of the curves in Figure~\ref{fig:1} depicts the relative market capitalization in ranked order of the major US markets' stocks on a log-log scale from 1926 to 2016. The \emph{relative market capitalization} or \emph{market weight} is defined as the percentage of the market capitalization of a fixed company, i.e., the number of outstanding shares times the current price of one share, divided by the capitalization of the whole market.
The striking feature of these curves is their remarkably stable shape over the last century. Although the market weights of each company fluctuate stochastically, 
the shape of the capital distribution curves differs (in first order) over the years only by the number of stocks present
in the considered market. This fundamental observation was the starting point for R.~Fernholz to develop \emph{stochastic portfolio theory (SPT)} about 20 years ago, see \cite{fernholz_02_stochastic_portfolio_theory} and in particular Figure 5.1 therein. 

SPT does not only 
analyze these capital distribution curves but also 
 seeks to explain the superior performance relative to the market portfolio of specific portfolios that exploit the described stability features.  This analysis leads to the concept of \emph{relative arbitrages}. According to ~\cite{vervuurt_15_topics_in_stochastic_portfolio_theory}, models which allow for such relative arbitrages can be categorized into three main classes, namely rank-based models (see \cite{banner_05_atlas_model_in_equity_markets,ichiba_11_hybrid_atlas_model, Jourdain_Reygner_15_mean_field_atlas}), diverse models (see \cite{fernholz_02_stochastic_portfolio_theory}) and sufficiently volatile models (see, e.g., \cite{fernholz_18_volatility_and_arbitrage}).
 
One example of the latter are so-called
\emph{volatility-stabilized market models} introduced in \cite{FK05_relative_arbitrage_volatility_stabilized} which fall into the rich class of
of \emph{polynomial diffusion processes} (see \cite{cuchiero2012polynomial,filipovic_larsson_16_polynomial_diffusion, Cuchiero_19_polynomial_processes_SPT}). Together with the subset of affine diffusion processes (see \cite{DFS:03}) and infinite dimensional extensions (see e.g.~\cite{schmidt2020infinite,  Cuchiero2019InfinitedimensionalPP}),
they appear throughout in mathematical finance, but also in several other fields like population genetics, chemistry, and physics. This, together with certain universal approximation properties of so-called signature SDEs \cite{arribas2020sig, cuchiero2023signature}, justifies to classify them
as \emph{universal stochastic modeling class}.

Hence, one goal of the present work is to combine some aspects of this mathematical universality with the financial one and to model in the spirit of  \cite{Cuchiero_19_polynomial_processes_SPT} the capital distribution curves via polynomial processes, which have empirically proved to provide a very good fit to these curves. 
More precisely, we extend volatility-stabilized market models by allowing for a common noise term such that the models remain polynomial. Indeed, for $t \in [0,T]$, where $T$ denotes some finite time horizon,  we consider the following model for the $N$ individual market capitalizations $S(t)=(S_1(t), \ldots, S_N(t))$
\begin{align*}
 d S_{i}(t)=\beta \sum_{j=1}^{N}S_{j}(t)dt+\sqrt{\alpha}\sqrt{S_{i}(t)\sum_{j=1}^{N}S_{j}(t)}dW^{i}_{t} + \sqrt{(N-\alpha)} S_{i}(t) dW_t^0, \quad i=1, \ldots, N,
\end{align*}
where $\alpha \geq 0$, $\beta \geq \frac{\alpha}{2}$ and $W^{(N)}\coloneqq \big(W^{1},\dots,W^{N})$  are the idiosyncratic Brownian motions and $W^0$ the common one independent of $W^{(N)}$.
The introduction of this common noise term permits to overcome the absence of correlation between the individual stocks in the original model of \cite{FK05_relative_arbitrage_volatility_stabilized}, while maintaining the noteworthy empirical feature that smaller stocks tend to exhibit greater volatilities compared to the ones of larger stocks.
Note that the correlation structure with a single common noise term is very simple.
Within the polynomial framework, one could consider various more sophisticated extensions, for example block-diagonal matrices, corresponding to common noises only for stocks in the same economy sector. The application of our methods to such more refined correlation structures is, however, left for future research. 
Let us remark that the market weight process defined by $S_i /\sum_{i=1}^N S_i$ is the same as in volatility-stabilized markets and given by a multivariate Jacobi process which is again a polynomial process whose state space is the unit simplex (see,  e.g.,~\cite{Cuchiero_19_polynomial_processes_SPT}).
 In other words, the common noise term enters only via the total market capitalization process $\Sigma:= \sum_{i=1}^N S_i$ which is given by
\[
d \Sigma(t)= \beta N \Sigma(t) dt + \sqrt{\alpha}\sqrt{\Sigma(t)}\sum_{i=1}^N \sqrt{S_i(t)} dW^i_t+ \sqrt{N-\alpha} \Sigma(t)dW_t^0
\]
and whose law is (similarly to volatility stabilized markets) equal to the one of a Black Scholes model.

The main aim of this article is to analyze the limit of the individual market capitalizations as $N\rightarrow \infty$.
This is inspired by several works that studied mean field limits and sometimes also fluctuations around them, on the one hand for volatility-stabilized markets and on the other hand for ranked-based models. Indeed, the closest connection is to the article by M.~Shkolnikov \cite{Shkolnikov_13_large_volatility-stabilized} who considered large volatility stabilized markets \emph{without} common noise.
Rank-based models without common noise have been treated in \cite{Shkolnikov_12_interacting_through_ranks,Kolli_Shkolnikov_18_fluctuations_rank-based, jourdain_13_propagation_chaos_rank_based,Jourdain_Reygner_15_mean_field_atlas,reygner_17_long_term_behaviour}, while the inclusion of a common noise has been studied in \cite{kolli_19_large_rank_based_common_noise}.

To analyze the limit as $N\rightarrow \infty$, we have to  rescale time, i.e.~let time go slower as we add particles, and consider $X(t):= S(t/N)$
\begin{align}\label{eqn:X_Ni_SDE}
 d X_{i}(t)=\frac{\beta}{N}\sum_{j=1}^{N}X_{j}(t)dt+\sqrt{\frac{\alpha}{N}}\sqrt{X_{i}(t)}\sqrt{\sum_{j=1}^{N}X_{j}(t)}dW_{t}^{i}+\sqrt{1-\frac{\alpha}{N}}X_{i}(t)dW^{0}_{t}.
\end{align}
Note that the difficulty in analyzing the mean-field limit comes from the combination of the non-Lipschitz coefficients in front of the idiosyncratic Brownian motions and the common noise term, which neither allows to apply standard theory nor the methods used in \cite{Shkolnikov_13_large_volatility-stabilized}.
Nevertheless, 
taking formal limits in \eqref{eqn:X_Ni_SDE} and denoting the   typical particle in the limit by $Y$, then yields
   \begin{align}\label{eqn:MKV_SDE_introduction}
      d Y(t) &=\beta\mathbb{E}[Y(t)|\mathcal{W}^{0}]dt+\sqrt{\alpha Y(t)\mathbb{E}[Y(t)|\mathcal{W}^{0}]}dB_{t}+Y(t)dW_{t}^{0}, 
    \end{align}
    where 
$B$ denotes some  Brownian motion  independent of $W^0$ and  $\mathcal{W}^{0}$  the sigma-algebra generated by $W^0$. 
To make this rigorous we consider, as usual for McKean-Vlasov equations, the 
 particles' empirical probability measure on path space, i.e.~
    \begin{align*}
       \rho^{N}:=\frac{1}{N}\sum_{i=1}^{N}\delta_{X_{i}}
    \end{align*}
 and its  mean-field limit $(\rho^{N}_{t})_{t\in[0,T]}\rightarrow (\rho_{t})_{t\in[0,T]}$ in  $C([0,T];M_1(\mathbb{R}_{+}))$,
where 
 $M_1(\mathbb{R}_{+})$ denotes probability measures over $\mathbb{R}_+$ with finite first moment, i.e.
\begin{align*}
    M_{1}(\mathbb{R}_+)=\left\{\mu\in M(\mathbb{R}_+) \mid \int_{\mathbb{R}_+}x\mu(dx)=:\langle x,\mu\rangle<\infty \right\},
\end{align*}
equipped with the Wasserstein-$1$ distance. 

More precisely,  we first prove almost sure convergence of $\rho^N$ in $C([0,T];M_1(\mathbb{R}_{+}))$ to
 the unique probabilistically strong, analytically weak and $M_{1}(\mathbb{R}_{+})$-valued solution $\rho$ of the following  \emph{degenerate, non-linear and non-local} stochastic partial differential equation (SPDE)
\begin{align}\label{eqn:intro_spde}
d \rho_{t}= (\frac{\alpha}{2}\langle \rho_{t},\operatorname{id}_x\rangle\partial_{x}^{2}(x  \rho_{t})+\frac{1}{2}\partial_{x}^{2}(x^2 \rho_{t})-\beta\langle \rho_{t},\operatorname{id}_x\rangle \partial_{x} \rho_{t})dt-\partial_{x}(x \rho_{t})dW^{0}_{t},
\end{align}
driven by the one-dimensional Brownian motion $W^0$ (see Theorem~\ref{thm:introduction_convergence_to_SPDE}). As already mentioned, for this result the methods used in \cite{Shkolnikov_13_large_volatility-stabilized}, in particular the elegant connection with a time-changed squared Bessel-process, are no longer applicable due to the common noise term.
Therefore we have to resort to different more analytic approaches, in particular in view of (pathwise) uniqueness of \eqref{eqn:intro_spde}.
Indeed, one
of the mathematical subtleties  lies in this proof
which involves fine estimates of weighted Sobolev norms where the weights have to be chosen in dependence on the differentiation order. 

In the second step we
 associate the solution $\rho$ of \eqref{eqn:intro_spde} with the conditional law of the solution of the  McKean-Vlasov SDE \eqref{eqn:MKV_SDE_introduction} 
and show that 
$\rho=\mathcal{L}(Y(\cdot)|\mathcal{W}^{0})$.
and $\langle \rho_t, \operatorname{id}_x \rangle=\int_{\mathbb{R}_+} x \rho_t(dx)= \mathbb{E}[Y(t)| \mathcal{W}^{0}]$ hold. Intriguingly,  the two-dimensional process $(Y,\mathbb{E}[Y|\mathcal{W}^{0}])$ is a polynomial diffusion on $\mathbb{R}_{++}^2$ whose dynamics are given by
\begin{align*}
        d Y(t)&=\beta \mathbb{E}[Y(t)|\mathcal{W}^{0}]dt+\sqrt{\alpha}\sqrt{Y(t)\mathbb{E}[Y(t)|\mathcal{W}^{0}]}dB_{t}+Y(t)dW^{0}_{t}
\\d \mathbb{E}[Y(t)|\mathcal{W}^{0}]&=\beta \mathbb{E}[Y(t)|\mathcal{W}^{0}]dt+\mathbb{E}[Y(t)|\mathcal{W}^{0}]dW^{0}_{t}.
    \end{align*}
Note that the equation for  $\mathbb{E}[Y(t)|\mathcal{W}^{0}]$ corresponds to a Black-Scholes model with drift coefficient $\beta$ and volatilty $1$.
Moreover, the pathwise uniqueness result for the SPDE allows us to conclude also uniqueness in law of the polynomial diffusion $(Y,\mathbb{E}[Y(t)|\mathcal{W}^{0}])$ (which was open so far see e.g.~\cite[Section 4]{filipovic_larsson_16_polynomial_diffusion}). Additionally, we also obtain a Besov regularity estimate for certain modifications of the marginal densities of $Y$ with respect to the Lebesgue measure. All these results are gathered in Theorem \ref{thm:collected_results_McKean-Vlasov_SDE_introduction}.

To consider 
the limiting behavior of the particle systems on the level of the individual particles we also establish a 
conditional propagation-of-chaos result (see Proposition ~\ref{prop:conditional_propagation_of_chaos_introduction}). Also in this case we cannot apply standard arguments 
since the usual Lipschitz conditions (\cite[Assumption 3.5]{Gess_Coghi_19_stochastic_fokker_planck}) are not satisfied.

Some numerical implementations of the particle system, the McKean Vlasov SDE as well as the SPDE are presented in Section \ref{sec:numerics}. We investigate there in particular the impact of the common noise by comparing the dynamics with the particle system and the PDE without common noise as studied in \cite{Shkolnikov_13_large_volatility-stabilized}.

In the remainder of this section, we shall introduce the notation used throughout the paper and present the main results.

\subsection{Notation}
\begin{itemize}

   \item We usually write $\R_{+}\coloneqq [0,\infty)$ and $\R_{++}\coloneqq (0,\infty)$.

    \item We denote by 
    $C_{b}^{k}(\R)$ the space of bounded $k$-times continuously differentiable functions from $\R $ to $\R$. If they additionally have compact support or vanish at $\pm \infty$, we write 
    $C_{c}^{k}(\R)$ and $C_{0}^{k}(\R)$ respectively.
     \item To indicate a derivative, with respect to the $x$ variable, we will use $\partial_{x}$, $\frac{\partial}{\partial x}$ and $\nabla$ interchangeably. Higher order derivatives will be denoted in the usual way by  $\partial_{x}^{k}=\underbrace{\partial_{x}\circ\dots\circ \partial_{x}}_{k\operatorname{-times}}$ for $k\in \N$.

    \item For a metric space we shall generically write $d(\cdot,\cdot)$ to  denote its metric.
    \item $\mathcal{S}$ denotes the space of Schwartz functions on $\mathbb{R}$, 
  while $\mathcal{S}_{0}\coloneqq \{\varphi\in \mathcal{S}\colon \varphi(0)=0\}$ is the set of Schwartz functions vanishing at the origin.
   The dual of  $\mathcal{S}$ is the space of Schwartz/tempered distributions, denoted by $\mathcal{S}'$.
\item The identity function $x\mapsto x$ is denoted by $\idx$ or $(\cdot_{x})$.  If it appears as part of an operator, we slightly abuse the notation and simply write $x$ for the map $x\mapsto x$. The same notation will be used when defining weighted Sobolev norms.    
    \item $M(E)$ denotes the space of probability measures on a topological space $E$. For $\mu \in M(E)$ and a $\mu$-integrable function $f: E \to \mathbb{R}$ we define the pairing  $ \langle f, \mu \rangle$ via
    \[
    \langle f, \mu \rangle : = \int_E f(x)\mu(dx).
    \]
Moreover, $M_{1}(E)$ denotes the space of probability measures on $E$ with finite first moment, i.e.~$M_{1}(E)\coloneqq \{\mu\in M(E)\colon \langle \mu,|\idx|\rangle<\infty\}$, equipped with the Wasserstein-1 distance (see \cite[Definition 6.8]{villani_09_optimal_transport_old_and_new} for the notion of convergence in $M_{1}(E)$).
 The Wasserstein distance between two measures $\mu$ and $\nu$ will be denoted by $\Wd(\mu,\nu)$.
 \item $\|\cdot\|_{\mathcal{L}(X,Y)}$ denotes the operator norm of an operator mapping from the Banach space $X$ into the Banach space $Y$.
    
\end{itemize}
We follow the convention that $C$ denotes a generic numerical constant that may change from line to line. If we want to highlight a particular dependence, we add the corresponding parameters as a subscript. In the case that a parameter of interest does not appear in the subscript, the constant is uniform with respect to this parameter.

\subsection{Main results}
We are now ready to present the main results concerning the convergence of the particle system \eqref{eqn:X_Ni_SDE}, existence and uniqueness of the SPDE \eqref{eqn:intro_spde}, its connection with the McKean-Vlasov SDE \eqref{eqn:MKV_SDE_introduction} and the conditional propagation of chaos result.

Let $[0,T]$ be a fixed, but arbitrary time horizon. Consider a filtered probability space $\left(\Omega,\mathcal{F},(\mathcal{F}_{t})_{0 \leq t \leq T},\Prob\right)$. Let $W^{0}$ be a one-dimensional Brownian motion thereon. We denote by  $\mathcal{W}^{0}_{t}$ the sigma algebra generated by  $W^{0}$  up to time $0 \leq t \leq T$ and by $\mathcal{W}^{0}$  the sigma algebra $\bigvee_{0 \leq t \leq T}\mathcal{W}^{0}_{t}$. We assume that the filtration $(\mathcal{F}_{t})_{0 \leq t \leq T}$ is compatible with $W^{0}$ in the sense that there exists a complete filtration $(\mathcal{G}_{t})_{0 \leq t \leq T}$, such that for each $0 \leq t \leq T$, $\mathcal{G}_{t}$ is independent of $\mathcal{W}^{0}_{t}$ and $\mathcal{F}_{t}=\sigma\left(\mathcal{G}_{t}\cup \mathcal{W}^{0}_{t} \right)$. 
An example for the filtration $(\mathcal{F}_{t})_{0 \leq t \leq T}$ is given by $\mathcal{F}^{W^{0},W^{(N)},\xi}_{t}=\sigma(\mathcal{W}^{0}_{t}\cup \sigma(W^{(N)}_{t})\cup \sigma(\xi))$, where $\sigma(W^{(N)}_{t})$ is the sigma-algebra induced by an independent N-dimensional Brownian motion $W^{(N)}$ up to time $0 \leq t \leq T$ and $\sigma(\xi)$ corresponds to the sigma-algebra generated by the initial condition. We remark that for any $(\mathcal{F}_{t})_{0 \leq t \leq T}$-adapted stochastic process $X\colon [0,T]\times\Omega\rightarrow \R^{d}$,
\begin{align*}
    \E[X(t)|\mathcal{W}^{0}]=\E[X(t)|\mathcal{W}^{0}_{t}].
\end{align*}

Recall that the N-particle system introduced in  \eqref{eqn:X_Ni_SDE} is given by
\begin{align*}
    \d X_{i}(t)=\frac{\beta}{N}\sum_{j=1}^{N}X_{j}(t)\dt+\sqrt{\frac{\alpha}{N}}\sqrt{X_{i}(t)}\sqrt{\sum_{j=1}^{N}X_{j}(t)}\dW_{t}^{i}+\sqrt{1-\frac{\alpha}{N}}X_{i}(t)\dW^{0}_{t}, 
\end{align*}
where $(W^0,W^{(N)})=(W^0,W^1, \ldots, W^N)$ is a standard $N+1$-dimensional Brownian motion and $\alpha$ and $\beta$ are constants.
Throughout the paper, we shall consider the following two sets of assumptions for this model.

\begin{assumption}\label{A:Assumptions_A}
\begin{enumerate}[label=\normalfont(A\arabic*)]
    \item \label{A:A1_alpha_beta} $\alpha, \beta$ are real constants such that $\beta \geq \alpha/2$, $\alpha\geq 0$.
    \item \label{A:A_2_X(0)_positive} The initial values $X_{1}(0),\dots,X_{N}(0)$ are strictly positive random variables.
    \item \label{A:A_3_X(0)_Z(0)}

    The laws  satisfy 
        \begin{align}
            \Law\bigg(\frac{1}{N}\sum_{i=1}^{N}\delta_{X_{i}(0)}\bigg)_{N\in\N} \xrightarrow[M(M_{1}(\mathbb{R}_{++}))]{} \delta_{\lambda},
        \end{align}
        for some $\lambda \in M_{1}(\R_{++})$ whose first moment is finite. Recall that $M$ is equipped with the topology of weak convergence and $M_{1}$ with the Wasserstein-1 topology.
 Moreover, we define $m_{\lambda}:=\int_{\mathbb{R}_+}x\lambda(\dx)$ and assume      
    \begin{align}\label{eqn:A:E_sum_X_i_converges}
        &\lim_{N\rightarrow\infty}\frac{\E\sum_{i=1}^{N}X_{i}(0)}{N}=m_{\lambda},&\lim_{N\rightarrow\infty}\frac{\E\bigg(\sum_{i=1}^{N}X_{i}(0)\bigg)^{1+\zeta}}{N^{1+\zeta}}=m_{\lambda}^{1+\zeta},
    \end{align}
    for some $\zeta>0$. If the strong law of large numbers holds for $X_1(0), X_2(0), \ldots$ 
    the above holds automatically true.
    In addition, we require
    \begin{align}\label{eqn:A:assumption_average_second_moment_initial_condition}
        \sup_{N\in\N}\frac{1}{N}\sum_{i=1}^{N}\E X_{i}(0)^{1+\zeta}<\infty.
    \end{align}
\end{enumerate}
\end{assumption}
\begin{assumption}\label{A:Assumptions_B}
    \begin{enumerate}[label=\normalfont(B\arabic*)]
        \item \label{A:B1_alpha,beta} $\alpha, \beta$ are real constants such that $\beta \geq \alpha/2$, $\alpha\geq 0$.
           \item \label{A:B2_X(0)_iid} The initial values $X_{1}(0),\dots,X_{N}(0)$ are strictly positive and iid. It holds that 
           \begin{align*}
               \E[X_{1}(0)]<\infty,\quad\text{ and} \quad  \E[X_{1}(0)^{2}]<\infty.
           \end{align*}
           \item \label{A:B3_convergence}
               The laws 
               \begin{align}
            \Law\bigg(\frac{1}{N}\sum_{i=1}^{N}\delta_{X_{i}(0)}\bigg)_{N\in\N} \xrightarrow[M(M_{1}(\mathbb{R}_{++}))]{} \delta_{\lambda},
        \end{align}
        for some $\lambda \in M_{1}(\R_{++})$ whose first moment is finite.
    \end{enumerate}
\end{assumption}

Clearly Assumptions \ref{A:A1_alpha_beta}--\ref{A:A_3_X(0)_Z(0)} follow from \ref{A:B1_alpha,beta}--\ref{A:B3_convergence}. For most results, we will only rely on Assumptions \ref{A:Assumptions_A}, since the iid assumption is rather restrictive considering our application.
\begin{remark}
The condition $\beta \geq \alpha/2$ in Assumption \ref{A:A1_alpha_beta} is motivated by the financial context and yields strict positivity of the solutions.
From a purely mathematical point of view, it suffices to assume $\beta \geq 0$, $\alpha\geq 0$. This concerns also \ref{A:A_2_X(0)_positive} where non-negativity instead of positivity would be sufficient. 
\end{remark}

\begin{definition}\label{def:strong_solution}
Recall the set $\mathcal{S}_{0}\coloneqq \{\varphi\in \mathcal{S}\colon \varphi(0)=0\}$ of Schwartz functions vanishing at the origin.
Given a stochastic basis $\big(\Omega,\mathcal{F},(\mathcal{F}_{t})_{t\in[0,T]},\Prob \big)$, a one-dimensional Brownian motion $W^{0}$ and an initial condition $\rho_{0}\in L^{2}(\Omega, M_{1}(\R_{+}))$,
then a probabilistically strong solution of  
\begin{equation}
\begin{split}\label{eqn:SPDE_classical_formulation}
    \d \rho_{t}= \frac{\alpha}{2}\langle \rho_{t},\idx\rangle\partial_{x}^{2}(\idxo \rho_{t})+\frac{1}{2}\partial_{x}^{2}(\idxo^{2} \rho_{t})-\beta\langle \rho_{t},\idx\rangle \partial_{x} \rho_{t}\dt-\partial_{x}(\idxo \rho_{t})\dW^{0}_{t}
    \end{split}
\end{equation}
 is a continuous $M_{1}(\R_{+})$--valued $(\mathcal{F}_{t})_{t\in[0,T]}$--adapted process $\rho$ such that for all $\varphi\in \mathcal{S}_{0}$ almost surely the equality
\begin{align}\label{eqn:SPDE_weak_formulation}
\langle\rho_{t},\varphi\rangle&=\langle\rho_{0},\varphi\rangle+\int_{0}^{t}\frac{\alpha}{2}\langle \rho_{r},\idxo\partial_{x}^{2}\varphi\rangle  \langle \rho_{r},\idx\rangle+\frac{1}{2}\langle \rho_{r},\idxo^{2}\partial_{x}^{2}\varphi\rangle +\beta \langle \rho_{r},\partial_{x}\varphi\rangle \langle \rho_{r},\idx\rangle\dr\\
    &\phantom{xx}{}+\int_{0}^{t} \langle \rho_{r},\idxo\partial_{x}\varphi\rangle\dW^{0}_{r}\nonumber,
\end{align}
holds for all $t\in[0,T]$. In particular, $\langle \rho_{t}, 1\rangle =1$ for every $t\in [0,T]$ and $\Prob$-almost every $\omega$, whenever $\langle \rho_{0}, 1\rangle =1$ for every $t\in [0,T]$ and $\Prob$-almost every $\omega$ as assumed.
\end{definition}
The main result of this work is summarized in the following theorem. For its formulation recall that the
empirical measure of $X_1(t), \ldots, X_N(t)$ is denoted by
$$\rho^{N}_{t}:= \frac{1}{N}\sum_{i=1}^{N}\delta_{X_{i}(t)}.$$

\begin{theorem}\label{thm:introduction_convergence_to_SPDE}
Let Assumptions \ref{A:Assumptions_A} be satisfied. Each convergent subsequence of $(\rho^{N}_{\cdot})_{N\in\N}$ converges a.s. in $C([0,T],M_{1}(\R_{+}))$, to the pathwise unique probabilistically strong, $ M_{1}(\R_{+})$--valued solution $\rho$ of the stochastic partial differential equation \eqref{eqn:SPDE_classical_formulation}
in the sense of Definition \ref{def:strong_solution}.
\end{theorem}

\begin{proof}
The result follows from Proposition \ref{prop:e-ufinite}, Proposition \ref{prop:convergence_to_weak_solution_SPDE}, Lemma \ref{prop:uniqueness_SPDE_weighted_space} and Proposition \ref{prop:Yamada-Watanabe_type_result}.
\end{proof}

The proof of Theorem \ref{thm:introduction_convergence_to_SPDE} is structured as follows: first, we shall establish existence, uniqueness, and (strict) positivity of solutions to the finite-dimensional system \eqref{eqn:X_Ni_SDE}. This will be the content of Proposition \ref{prop:e-ufinite}.
Next, we use tightness arguments to extract convergent sub-sequences of $(\rho^{N})_{N\in\N}$ and show that they converge to a probabilistically weak solution of \eqref{eqn:SPDE_classical_formulation}. This yields Proposition \ref{prop:convergence_to_weak_solution_SPDE} and the existence of, potentially non-unique, solutions to \eqref{eqn:SPDE_classical_formulation}. In Lemma \ref{lem:convergence_of_mean_rho_N}, we derive an explicit representation for the non-local term $\langle\rho^{N}_{\cdot},\idx\rangle$, which will prove to be quite useful.
We then introduce a family of weighted Sobolev spaces which are suited to study differential operators with unbounded coefficients of polynomial degrees. This allows us to obtain the estimates in Proposition \ref{prop:uniqueness_SPDE_weighted_space} and conclude that the solution is pathwise unique. A Yamada-Watanabe type argument (see Proposition \ref{prop:Yamada-Watanabe_type_result}) then yields the full statement of Theorem \ref{thm:introduction_convergence_to_SPDE}.\newline

\begin{remark}\label{rem:boundary_behaviour_SPDE}
   Equation \eqref{eqn:SPDE_classical_formulation} a priori only makes sense in a distributional setting, so $\rho$ should be viewed as a Schwartz distribution. In our analysis however, we will see that the regularity of the initial condition is propagated and $\rho$ remains a probability measure. 
   
   The correspondence of \eqref{eqn:SPDE_weak_formulation} and \eqref{eqn:SPDE_classical_formulation} is understood in the distributional sense. The space $\mathcal{S}_{0}$ is connected to this and implicitly encodes homogeneous Dirichlet boundary conditions in the sense that the (formal) identities   $\langle \partial_{x}^{2}(x\rho),\varphi\rangle=\langle \rho,x\partial_{x}^{2}\varphi \rangle$  and $\langle \partial_x \rho , \varphi \rangle=- \langle \rho, \partial_x \varphi \rangle$ only hold, if at least one of the involved objects $\rho$ or $\varphi$ vanishes at the boundary. 
Due to the condition
$\beta \geq \frac{\alpha}{2}$ which guarantees strict positivity of 
 the finite dimensional particle system \eqref{eqn:X_Ni_SDE}, the correct domain of $\rho^N$ is $\R_{+}$ with 0 boundary conditions, i.e. formally $\rho^{N}_{t}(0)=0$ for all $t\in[0,T]$. 
 A priori, it is however not clear that this holds true in the limit
and that a solution $\rho$ of \eqref{eqn:SPDE_weak_formulation} does not have mass at $0$. \newline
   
    Nevertheless, to understand the ``natural'' boundary behavior of \eqref{eqn:SPDE_classical_formulation}
   it is possible to adapt the following deterministic arguments. Indeed, note that the $dt$-part of \eqref{eqn:SPDE_classical_formulation} corresponds to the following (non-linear) Fokker-Plank equation
\[
\d \rho_{t}= (\frac{\alpha}{2}\langle \rho_{t},\idx\rangle\partial_{x}^{2}(\idxo \rho_{t})+\frac{1}{2}\partial_{x}^{2}(\idxo^{2} \rho_{t})-\beta\langle \rho_{t},\idx\rangle \partial_{x} \rho_{t})\dt,
 \]
which can formally be interpreted as the marginal laws of an SDE with drift $\beta\langle \rho_{t},\idx\rangle $ and diffusion $\alpha x \langle \rho_{t},\idx\rangle   + x^2$. Assuming, a priori non-negativity of $\langle \rho_{t},\idx\rangle$ for all $t \in [0,T]$, we can apply the boundary non-attainment conditions of  \cite[Theorem 5.7 (i)]{filipovic_larsson_16_polynomial_diffusion} to see that this SDE also stays strictly positive. Indeed, \cite[Condition (5.4)]{filipovic_larsson_16_polynomial_diffusion}, reads as
\begin{align}\label{eq:strictpos}
2 \beta \langle \rho_t,\idx\rangle - \alpha \langle \rho_t,\idx\rangle \geq 0,
\end{align}
which is exactly in line with our assumption that $\beta \geq \frac{\alpha}{2}$ (as long as  $\langle \rho_t,\idx\rangle$ is non-negative).
Hence, the deterministic version of \eqref{eqn:SPDE_classical_formulation}
does not require prescribed (Dirichlet)
boundary conditions since they are automatically satisfied, i.e.~in this case $\rho_t(0)=0$ holds true.
Note that condition \eqref{eq:strictpos} can also be found by applying the criterion introduced by Fichera  
in \cite{fichera_56_book}, and the associated maximum principle, directly to the PDE (see also \cite[Lemma 1.1.2, Theorem 1.1.2]{oleinik_12_second_order_equations}, or \cite{feehan_12_maximum_principles_PDE} for further extensions). For the full equation \eqref{eqn:SPDE_classical_formulation}, heuristically the noise term should not impact the behavior at the boundary, since it can be interpreted as a ``drift''  whose coefficient is $x$ and thus vanishes at $0$. 
 A rigorous proof is however beyond the scope of this work.

\end{remark}

The following theorem establishes the link between the solution to the SPDE \eqref{eqn:SPDE_classical_formulation}
and the law of the McKean-Vlasov SDE that arises as the limit of \eqref{eqn:X_Ni_SDE}. We also show the intriguing property that together with its conditional expectation this McKean-Vlasov SDE is again a polynomial process. Furthermore, we establish an existence and Besov regularity result for certain modifications of the marginal densities.

\begin{theorem}
\label{thm:collected_results_McKean-Vlasov_SDE_introduction}
 Let Assumption \ref{A:A1_alpha_beta} be satisfied and assume that the initial value of \eqref{eqn:MKV_SDE_introduction} is deterministic and satisfies $Y(0) > 0$.
\begin{enumerate}[label=(\Roman*)]
    \item\label{prop:P_1:stochastic_representation_rho_introduction} The McKean-Vlasov SDE
    \eqref{eqn:MKV_SDE_introduction}
with independent, one-dimensional, Brownian motions $B$ and $W^{0}$ has a unique (in law) weak solution and $\rho=\Law(Y(\cdot)|\mathcal{W}^{0})$, where $\rho$ is the unique solution of \eqref{eqn:SPDE_classical_formulation} with $\rho_{0}=\delta_{Y(0)}$.
\item\label{prop:P_2:McKean-Vlasov_SDE_and_mean_jointly_polynomial_introduction}  The two-dimensional process $(Y,\langle\rho,\idx\rangle)= (Y, \mathbb{E}[Y | \mathcal{W}^0])$ is a polynomial diffusion on $\mathbb{R}^2_{++}$ which is unique in law. 
\item \label{prop:P_3:McKean-Vlasov_SDE_regularity_of_density_introduction} Let $Y$ be the solution of the McKean-Vlasov SDE  \eqref{eqn:MKV_SDE_introduction} and denote by $\mu_{Y}$ its law. Let $\frac{1}{3}> \delta >0$, 
 and $m=\lceil \frac{3\left(1-\delta\right)^{2}}{2\delta}\rceil$. Then, for every $t\in (0,T]$, the measure $\min\{1,x\}^{m}\mu_{Y(t)}(\dx)$ is absolutely continuous with respect to the Lebesgue measure on $\mathbb{R}_{+}$ and its density lies in the Besov space $B^{1/2-\widetilde{\eps}}_{1,\infty}$, for any arbitrarily small but fixed $\widetilde{\eps}=\frac{3\delta}{2}>0$.
\end{enumerate}
\end{theorem}

\begin{proof}
The assertion follows from Proposition \ref{prop:McKeanE&U} and Proposition \ref{prop:regularity_density_Y}.
\end{proof}

The limiting behavior of the finite-dimensional particle systems, as $N\rightarrow \infty$, can also be considered on the level of the individual particles. Indeed, as we let $N\rightarrow \infty$, we deduce with the help of Theorem \ref{thm:introduction_convergence_to_SPDE} that the particles will become independent when conditioned on the common noise $W^{0}$.  Mathematically, this can be expressed in the following (conditional) propagation of chaos result:

\begin{proposition}\label{prop:conditional_propagation_of_chaos_introduction}
Let Assumptions \ref{A:Assumptions_B} hold. Furthermore let $(X_{i}^{(N)})_{i=1,\dots,N}$ denote the solution to the N-particle system \eqref{eqn:X_Ni_SDE} and $\rho$ the unique solution of \eqref{eqn:SPDE_classical_formulation}.
    The interacting particles $(X_{i}^{(N)})_{i=1,\dots,N}$ are, conditional with respect to $\mathcal{W}^{0}$, chaotic in the sense that for each $k\in \N$ and $\varphi_{1},\dots,\varphi_{k}\in C_{b}(\R)$, $t\in[0,T]$, we have
\begin{align}\label{eqn:conditionally_chaotic}
    \lim_{N\rightarrow \infty}\E\bigg|\E\big[\varphi_{1}(X_{1}^{(N)}(t))\cdot\dots\cdot\varphi_{k}(X_{k}^{(N)}(t))\big|\mathcal{W}^{0}\big]-\Pi_{i=1}^{k}\langle \rho_{t},\varphi_{i}\rangle \bigg|=0.\quad \forall t\in [0,T].
\end{align}
\end{proposition}
Property \eqref{eqn:conditionally_chaotic} is also referred to as $\rho$ being chaotic (see \cite[Lemma 3.6]{Gess_Coghi_19_stochastic_fokker_planck}).

\subsection{Structure of the paper}
The remainder of the paper is structured as follows. Section \ref{sec:Convergence_to_SPDE} is dedicated to the mean-field limit of the finite-dimensional particle system (Subsection \ref{sub_sec:existence_solutions}) and the existence of equation~\ref{eqn:SPDE_weak_formulation}. In Subsection \ref{sec:uniqueness}, we prove pathwise-uniqueness for \eqref{eqn:SPDE_weak_formulation} to obtain the full statement of Theorem \ref{thm:introduction_convergence_to_SPDE}. In Section \ref{sec:connection_McKean-Vlasov_SDE} we treat the McKean-Vlasov SDE and show Theorem \ref{thm:collected_results_McKean-Vlasov_SDE_introduction} (see Subsections \ref{sub_sec:stochastic_representation} and \ref{sub_sec:Regularity_of_density}). The conditional propagation of chaos result of Proposition \ref{prop:conditional_propagation_of_chaos_introduction} is obtained in Section \ref{sub_sec:Conditional_propagation_of_chaos}. In Section \ref{sec:numerics} we present numerical studies regarding the qualitative behavior of the particle system and the limiting SPDE. In the Appendix, we collect several auxiliary results needed in the proofs.

\section{Convergence to the SPDE / Proof of Theorem \ref{thm:introduction_convergence_to_SPDE}}\label{sec:Convergence_to_SPDE}
We split the proof of Theorem \ref{thm:introduction_convergence_to_SPDE} into three distinct parts: first we study the finite-dimensional system and its properties. Next, we move on to prove the existence of solutions to the SPDE and finally, we consider the question regarding the uniqueness of these solutions.
\subsection{Solutions to the finite-dimensional problem}
As a first step, we will show that the finite particle system has a solution. The theory developed in \cite{BP03_degenerate_SDE} for degenerate SDEs and in \cite{filipovic_larsson_16_polynomial_diffusion} for polynomial diffusions allows to construct (probabilistically) weak solutions to equation \eqref{eqn:X_Ni_SDE}. In the absence of common noise these solutions can also be expressed in terms of time-changed squared Bessel processes (see \cite{FK05_relative_arbitrage_volatility_stabilized}).
\begin{proposition}\label{prop:e-ufinite}
Let Assumptions \ref{A:Assumptions_A} be satisfied.
\begin{enumerate}[label=(\Roman*)]
\item
The finite-dimensional particle system given by \eqref{eqn:X_Ni_SDE} admits a weak solution with values in $\mathbb{R}_{++}$. 
\item This solution is also pathwise unique. Hence we get the existence of a unique strong solution.
\end{enumerate}
\end{proposition}

\begin{proof}
The existence result of part (i) can be deduced from  \cite[Theorem 1.2]{BP03_degenerate_SDE} but also follows from \cite[Theorem 5.3 and Proposition 6.4]{filipovic_larsson_16_polynomial_diffusion} since $X$ is a polynomial process. The (strict) positivity can be deduced from \cite[Theorem 5.7 (i)]{filipovic_larsson_16_polynomial_diffusion}.
The second part (ii) follows from analogous arguments as in 
\cite{P14_generalized_volatility_stabilized_processes}.
\end{proof}
\begin{remark}
    Assumptions \ref{A:A1_alpha_beta} and \ref{A:A_2_X(0)_positive} can be relaxed to $\beta\geq 0, \alpha \geq 0$ and the initial values being non-negative. In this case one can use \cite[Theorem 1.2]{BP03_degenerate_SDE} and It{\^o}'s formula for a smoothed version of the negative part of a function to get the existence of a non-negative solution.
\end{remark}
\subsection{Existence of solutions to the SPDE}\label{sub_sec:existence_solutions}
We start by introducing the notion of a
 probabilistically weak solution to the SPDE
 \eqref{eqn:SPDE_classical_formulation}.

\begin{definition}\label{def:weak_solution}
Let $\mathcal{S}_{0}\coloneqq \{\varphi\in \mathcal{S}\colon \varphi(0)=0\}$. A probabilistically weak solution of  
\begin{align*}
    \d \rho_{t}= \frac{\alpha}{2}\langle \rho_{t},\idx\rangle\partial_{x}^{2}(\idxo \rho_{t})+\frac{1}{2}\partial_{x}^{2}(\idxo^{2} \rho_{t})-\beta\langle \rho_{t},\idx\rangle \partial_{x} \rho_{t}\dt-\partial_{x}(\idxo \rho_{t})\dW^{0}_{t},
\end{align*}
 is a collection $\big((\Omega,\mathcal{F},(\mathcal{F}_{t})_{t\in[0,T]},\Prob), (\rho,W^{0})\big)$, such that $(\Omega,\mathcal{F},\Prob)$ is a probability space, $(\mathcal{F}_{t})_{t\in[0,T]}$  a complete filtration of $\mathcal{F}$, $W^{0}$  an $(\mathcal{F}_{t})_{t\in[0,T]}$--Brownian motion and $\rho$ a continuous $M_{1}(\R_{+})$--valued $(\mathcal{F}_{t})_{t\in[0,T]}$--adapted process such that for all $\varphi\in \mathcal{S}_{0}$ almost surely the equality
 \begin{align*}
    \langle\rho_{t},\varphi\rangle&=\langle\rho_{0},\varphi\rangle+\int_{0}^{t}\beta \langle \rho_{r},\partial_{x}\varphi\rangle \langle \rho_{r},\idx\rangle+\frac{\alpha}{2}\langle \rho_{r},\idxo\partial_{x}^{2}\varphi\rangle  \langle \rho_{r},\idx\rangle+\frac{1}{2}\langle \rho_{r},\idxo^{2}\partial_{x}^{2}\varphi\rangle\dr
    +\int_{0}^{t} \langle \rho_{r},\idxo\partial_{x}\varphi\rangle\dW^{0}_{r},
 \end{align*}
holds for all $t\in[0,T]$.
Note that we implicitly require $\langle \rho_{t}, 1\rangle =1$ for every $t\in [0,T]$ and $\Prob$-almost every $\omega$, whenever $\langle \rho_{0}, 1\rangle =1$ for every $t\in [0,T]$ and $\Prob$-almost every $\omega$ as assumed.
\end{definition}

\begin{remark}
    Alternatively, since $\R_{++}$ is Schwartz-equivalent (see \cite[Definition 2.7]{prywes2022schwartz}), it would also be natural to consider test functions from the space 
    \begin{align*}
        \mathcal{S}(\R_{++})\coloneqq \bigcap_{x\in \R \backslash \R_{++}}\bigcap_{k\in \N\cup \{0\}} \{\varphi\in \mathcal{S}(\R)\colon \partial_{x}^{k}\varphi(x)=0\}.
    \end{align*}
    Clearly, by extending each $\varphi \in \mathcal{S}(\R_{++})$ and its derivatives by $0$, we obtain a subspace of $\mathcal{S}_{0}$.
\end{remark}
By the definition of probabilistically weak solutions, 
we highlight that the new stochastic basis is part of the solution and might not coincide with the original one.

We shall devote the remainder of this subsection to the proof of the following Proposition.
\begin{proposition}\label{prop:convergence_to_weak_solution_SPDE} Let Assumptions \ref{A:Assumptions_A} be satisfied. Then there exists at least one probabilistically weak, $M_{1}(\R_{+})$--valued solution $\rho$ of the stochastic partial differential equation \eqref{eqn:SPDE_classical_formulation}
in the sense of Definition \ref{def:weak_solution}.
\end{proposition}
\begin{remark}
Since we require the terms $x\partial_{x}\varphi(x)$ and $x^{2}\partial_{x}^{2} \varphi$ in \eqref{eqn:SPDE_weak_formulation} to be continuous and bounded on $\R_{+}$ for the duality in \eqref{eqn:SPDE_weak_formulation} to be well defined, we will consider the equation a-priori in the space of Schwartz distributions $\mathcal{S}_{0}'$. Since we expect our solution to still be a measure, one could consider a larger test function space for the limiting procedure than $\mathcal{S}_{0}$, such as a suitable weighted Sobolev space like those which will be introduced in Section \ref{sec:uniqueness}. 
\end{remark}

\subsubsection{Tightness of Laws}
In order to prove Proposition \ref{prop:convergence_to_weak_solution_SPDE}, we need to
identify a convergent subsequence as well as a potential limit of the sequence of empirical measures $\rho^{N}$, as $N\rightarrow \infty$. This is done by considering the laws of $\rho^{N}$ on $C([0,T],M_{1}(\R_{+}))$ and use Prokhorov's theorem to show that the sequence is relatively compact. In order to use Prokhorov's theorem, we will require the following result. 

\begin{theorem}\label{thm:tightness}
 Let Assumptions \ref{A:Assumptions_A} be satisfied. 
The sequence $(Q^{N}_{T})_{N\in\N}$, where $Q^{N}_{T}$ denotes the law of the process $\rho^{N}_{\cdot}$ on $C([0,T];M_{1}(\R_{+}))$, is tight on $C([0,T],M_{1}(\R_{+}))$.
\end{theorem}

\begin{remark}
Proposition \ref{prop:e-ufinite}
guarantees, starting from a positive initial condition, that the solution of \eqref{eqn:X_Ni_SDE} remains positive and hence for every $t\in[0,T]$ and almost every $\omega$, $\rho^{N}_{t}((0,\infty))=1$.
 In the limit however, we can only claim that a limit point $\rho$ is supported on $[0,\infty)$ since by the Portmanteau-Theorem,
\begin{align*}
    &1=\limsup_{N\rightarrow \infty}\rho^{N}_{t}([0,\infty))\leq \rho_{t}([0,\infty))\\
    &\rho_{t}((0,\infty))\leq \liminf_{N\rightarrow \infty}\rho^{N}_{t}((0,\infty))=1.
\end{align*}

\end{remark}

\begin{proof}[Proof of Theorem \ref{thm:tightness}] Note that the space $C([0,T],M_{1}(\R_{+}))$ is equipped with the topology induced by the distance $d_{Z}(\mu,\nu)\coloneqq \sup_{t\in [0,T]}|\Wd(\mu_{t},\nu_{t})|$. To establish the claimed tightness, we verify the conditions of Theorem \ref{thm:tightness_ethier_kurtz} with $E$ replaced by $M_{1}(\R_{+})$, equipped with the Wasserstein-1 topology. For the first condition, we require a characterization of (relatively) compact sets in $M_{1}(\R_{+})$ which is provided by Lemma \ref{lem:compactness_in_M1}. 
We want to find a compact set $K^{0}_{\eps}=K^{0}_{t,\eps}\subseteq (M_{1}(\R_{+}),\Wd)$ such that, for every $t\in[0,T]$ and every $N\in\N$, $\rho^{N}_{t}$ lies in $K^{0}_{\eps}$ with probability greater than $1- \eps$. 
Let $\eps>0$. Given two positive constants  $0<M_{\eps,1}, M_{\eps,2} <\infty$ and $\zeta >0$, we propose the following set: $K^{0}_{\eps}\coloneqq \left\{\mu\colon \int_{\R_{+}}1+|x|^{1+\zeta}\mu(\dx) \leq M_{\eps,1}, \Wd(\mu,\delta_{0})\leq M_{\eps,2} \right\}$. The compactness of $K^{0}_{\eps}$ is a consequence of of Lemma \ref{lem:compactness_in_M1}. Indeed, the (uniform) bound $\int_{\R_{+}}1+|x|^{1+\zeta}\mu(\dx)$ implies the uniform integrability \ref{compactness_M1_C2} condition and the tightness \ref{compactness_M1_C1}. We can now verify \ref{item:tightness_ek_tightness} of Theorem \ref{thm:tightness_ethier_kurtz}:
\begin{align*}
    \sup_{N}\Prob\left(\rho_t^{N}\in K^{0}_{\eps}\right)=\sup_{N}\Prob\left(\rho^{N}_{t}\in \left\{\mu\colon \int_{\R_{+}}1+|x|^{1+\zeta}\mu(\dx) \leq M_{\eps,1}, \Wd(\mu,\delta_{0})\leq M_{\eps,2} \right\}\right)> 1-\eps.
\end{align*}
This is equivalent to showing that the probability of being an element of the complement of $K^{0}_{\eps}$, denoted by $(K^{0}_{\eps})^{c}$, is smaller than or equal to $\eps$. It suffices to show that $ \sup_{N}\Prob\left(\int_{\R_{+}}1+|x|^{1+\zeta}\rho_t^{N}(\dx) > M_{\eps,1}\right)\leq ~\frac{\eps}{2}$ and $ \sup_{N}\Prob\left(\Wd(\rho_t^N,\delta_{0})>  M_{\eps,2}\right)\leq ~\frac{\eps}{2}$. Using Assumptions \ref{A:Assumptions_A}, we get by the Markov inequality,
\begin{align*} \Prob\left(\int_{\R_{+}}1+|x|^{1+\zeta}\rho_t^{N}(\dx) > M_{\eps,1}\right)&\leq \Prob\bigg(1+\frac{1}{N}\sum_{i=1}^{N}|X_{i}(t)|^{1+\zeta}> M_{\eps,1}\bigg)\\
    &\leq \Prob\bigg(1+\frac{1}{N}\sum_{i=1}^{N}|X_{i}(t)|^{1+\zeta}> M_{\eps,1}\bigg)\\
    &\leq  \frac{1}{M_{\eps,1}}+\frac{\frac{1}{N}\sum_{i=1}^{N}\E\sup_{t\in[0,T]}|X_{i}(t)|^{1+\zeta}}{ M_{\eps,1}}\\
    &\leq  \frac{1}{M_{\eps,1}}+C_{T,\alpha,\beta,\zeta}\frac{\E\frac{1}{N}\sum_{i=1}^{N}|X_{i}(0)|^{1+\zeta}}{ M_{\eps,1}}\\
    &\leq  \frac{1}{M_{\eps,1}}+\frac{C_{T,\alpha,\beta,\zeta,X_{0}}}{ M_{\eps,1}}.
\end{align*}
We used Lemma \ref{lem:EsupXp_EXp_estimates} from which we obtained the constant $C_{T,\alpha,\beta,\zeta}$ and \eqref{eqn:A:assumption_average_second_moment_initial_condition} for the last inequality. Choosing $M_{\eps,1}$ such that  $M_{\eps,1} >\max \left\{\frac{4}{\eps},\frac{4C_{T,\alpha,\beta,\zeta,X_{0}}}{\eps}\right\}$  yields the first estimate.
By the Kantorovich-Rubinstein duality formula, we can express the Wasserstein-1 distance as follows
\begin{align*}
    \Wd (\mu, \nu)=\sup_{f\in \operatorname{Lip}} \left\{\int_{\mathbb{R}^d} f d(\mu-\nu) \mid\|f\|_{\operatorname{Lip}} \leq 1\right\},
\end{align*}
where $\operatorname{Lip}$ denotes the Lipschitz functions on $\R_{+}$ and  $\|f\|_{\operatorname{Lip}}$  their Lipschitz constant. We will denote the set of Lipschitz functions with Lipschitz constant $1$ by $\operatorname{Lip}_{1}$.
Since for every $f\in \operatorname{Lip}_{1}$,
\begin{align*}
    \left|\frac{1}{N}\sum_{i=1}^{N}f(x_{i})-f(0)\right|\leq \frac{1}{N}\sum_{i=1}^{N} |x_{i}|,
\end{align*}
\begin{align*}
    \Prob\left(\Wd(\rho_t^N,\delta_{0})>  M_{\eps,2}\right)&\leq  \Prob\left(\int_{\R_{+}}|x|\rho_t^{N}(\dx)> M_{\eps,2}\right)\\
     &\leq \Prob\bigg(\left(1+\frac{1}{N}\sum_{i=1}^{N}|X_{i}(t)|^{1+\zeta}\right)> M_{\eps,2}\bigg)\\
    &\leq  \frac{1}{M_{\eps,2}}+\frac{C_{T,\alpha,\beta,\zeta,X_{0}}}{ M_{\eps,2}}.
\end{align*}
As above, we can choose $M_{\eps,2}>\max \left\{\frac{4}{\eps},\frac{4 C_{T,\alpha,\beta,\zeta,X_{0}}}{\eps}\right\}$. In total, we obtain $  \sup_{N}\Prob\left(\rho_t^{N}\in (K^{0}_{\eps})^{c}\right)\leq \eps$.

This yields condition \ref{item:tightness_ek_tightness} of Theorem \ref{thm:tightness_ethier_kurtz}. \\

We now verify condition \ref{item:tightness_Aldous_condition}.
Recall that, by the Kantorovich-Rubinstein duality formula, we have for $f \in \operatorname{Lip}_1$
\begin{align*}
    \bigg|\Wd(\rho^{N}_{t},\rho^{N}_{s})\bigg|&=\sup_{f\in \operatorname{Lip}_{1}}\bigg|\frac{1}{N}\sum_{i=1}^{N}f(X_{i}(t))-\frac{1}{N}\sum_{i=1}^{N}f(X_{i}(s))\bigg|\\
    &\leq  \frac{1}{N}\sum_{i=1}^{N}\left|X_{i}(t)-X_{i}(s)\right|.
\end{align*}
Hence,
\begin{align*}
        \sup_{0\leq \theta \leq \delta}&\Prob\left(\bigg|\Wd(\rho^{N}_{\tau+\theta},\rho^{N}_{\tau})\bigg| >\eta\right)\leq   \sup_{0\leq \theta \leq \delta}\Prob\left(\frac{1}{N}\sum_{i=1}^{N}\left|X_{i}(\tau+\theta)-X_{i}(\tau)\right| >\eta\right)\\
     &\leq  \sup_{0\leq \theta \leq \delta}\Prob\left(\frac{1}{N}\sum_{i=1}^{N}\left|\int_{\tau}^{\tau+\theta}\frac{\beta}{N}\sum_{j=1}^{N }X_{j}(r)\dr\right| >\eta/2\right)\\
     &\phantom{xx}{}+\sup_{0\leq \theta \leq \delta}\Prob\left(\frac{1}{N}\sum_{i=1}^{N}\left|\int_{\tau}^{\tau+\theta}\sqrt{\frac{\alpha}{N}}\sqrt{X_{i}(r)}\sqrt{\sum_{j=1}^{N }X_{j}(r)}\dW^{i}_{r}\right| >\eta/4\right)\\
      &\phantom{xx}{}+\sup_{0\leq \theta \leq \delta}\Prob\left(\frac{1}{N}\sum_{i=1}^{N}\left|\int_{\tau}^{\tau+\theta}\sqrt{1-\frac{\alpha}{N}}X_{i}(r)\dW^{0}_{r}\right| >\eta/4\right)=I+II+III.
\end{align*}
We will consider each term separately.\newline
Using Markov's inequality the first integral can be estimated in a straightforward way,
\begin{align*}
    I=\sup_{0\leq \theta \leq \delta}\Prob\left(\frac{1}{N}\sum_{i=1}^{N}\left|\int_{\tau}^{\tau+\theta}\frac{\beta}{N}\sum_{j=1}^{N }X_{j}(r)\dr\right| >\eta/2\right)&\leq \sup_{0\leq \theta \leq \delta} \frac{2}{\eta}\E\left(\frac{1}{N}\sum_{i=1}^{N}\bigg|\int_{\tau}^{\tau+\theta}\frac{\beta}{N}\sum_{j=1}^{N }X_{j}(r)\dr\bigg|\right)\\
    &\leq \frac{2}{\eta} \delta \beta\E\sup_{r\in [0,T]}\left(\frac{1}{N}\sum_{j=1}^{N }X_{j}(r)\right),
\end{align*}
where we used the non-negativity of $X_{j}$ and that the integral term is independent of the summation variable $i$. The non-negativity of the integrand further yields
\begin{align*}
    \frac{2}{\eta} \delta \beta\E\sup_{t\in [0,T]}\left(\frac{1}{N}\sum_{j=1}^{N }X_{j}(t)\right)& \leq \frac{C_{T,\alpha,\beta}}{\eta} \delta \E\left(\frac{1}{N}\sum_{j=1}^{N }X_{j}(0)\right)\leq \frac{C_{T,\alpha,\beta}}{\eta} \delta,
\end{align*}
due to Lemma \ref{lem:EsupZp_EZp_estimates} and convergence \eqref{eqn:A:E_sum_X_i_converges} from Assumptions \ref{A:Assumptions_A}, by the boundedness of real-valued, convergent sequences.\newline
Next, we estimate the second term. By the Markov, Burkholder-Davis-Gundy and Jensen inequalities,
\begin{align*}
    II=\sup_{0\leq \theta \leq \delta}&\Prob\left(\frac{1}{N}\sum_{i=1}^{N}\left|\int_{\tau}^{\tau+\theta}\sqrt{\frac{\alpha}{N}}\sqrt{X_{i}(r)}\sqrt{\sum_{j=1}^{N }X_{j}(r)}\dW^{i}_{r}\right| >\eta/4\right)\\
    &\leq   \sup_{0\leq \theta \leq \delta}\frac{4}{\eta}\E\left(\frac{1}{N}\sum_{i=1}^{N}\left|\int_{\tau}^{\tau+\theta}\sqrt{\frac{\alpha}{N}}\sqrt{X_{i}(r)}\sqrt{\sum_{j=1}^{N }X_{j}(r)}\dW^{i}_{r}\right|\right)\\
     &\leq  \sup_{0\leq \theta \leq \delta}\frac{C_{\alpha}}{\eta}\E\left(\int_{\tau}^{\tau+\theta}\left(\frac{1}{N}\sum_{i=1}^{N}X_{i}(r)\right)^{2}\dr\right)^{1/2}\\
     &\leq \frac{C_{\alpha}}{\eta} \delta^{1/2}\E \sup_{t\in[0,T]}\left(\frac{1}{N}\sum_{i=1}^{N}X_{i}(t)\right)\leq \frac{C_{\alpha,\beta,T}}{\eta} \delta^{1/2} \E \left(\frac{1}{N}\sum_{i=1}^{N}X_{i}(0)\right)\leq \frac{C_{\alpha,\beta,T}}{\eta} \delta^{1/2}.
\end{align*}

We used Lemma \ref{lem:EsupZp_EZp_estimates} and convergence \eqref{eqn:A:E_sum_X_i_converges} from Assumptions \ref{A:Assumptions_A}.\newline
The last term can be bounded similarly.
\begin{align*}
   III&=\sup_{0\leq \theta \leq \delta}\Prob\left(\frac{1}{N}\sum_{i=1}^{N}\left|\int_{\tau}^{\tau+\theta}\sqrt{1-\frac{\alpha}{N}}X_{i}(r)\dW^{0}_{r}\right| >\eta/4\right)\\
    &\leq \sup_{0\leq \theta \leq \delta}C \sqrt{1-\frac{\alpha}{N}}\frac{4}{\eta}\frac{1}{N}\sum_{i=1}^{N}\E\left(\int_{\tau}^{\tau+\theta}X_{i}(r)^{2}\dr\right)^{1/2}\leq 
    C_{\alpha}\frac{\delta^{1/2}}{\eta}\frac{1}{N}\sum_{i=1}^{N}\E\sup_{t\in [0,T]}X_{i}(t).
\end{align*}

Following the proof of Lemma \ref{lem:EsupXp_EXp_estimates}, we see that 
\begin{align*}
      \E \sup_{t\in [0,T]} X_{i}(t)&=\E X_{i}(0)+C_{\alpha,\beta}\E\int_{0}^{T}X_{i}(s)\ds +C_{\alpha,\beta}\E\int_{0}^{T}\frac{1}{N}\sum_{j=1}^{N}X_{j}(s)\ds,
\end{align*}
with constants which only depend on $\alpha,\beta$. Hence
\begin{align*}
    \frac{1}{N}\sum_{i=1}^{N}\E\sup_{t\in [0,T]}X_{i}(t)\leq \E\frac{1}{N}\sum_{i=1}^{N}X_{i}(0)+C_{\alpha,\beta}\E\int_{0}^{T}\frac{1}{N}\sum_{i=1}^{N}X_{i}(s)\ds.
\end{align*}
We can use Lemma \ref{lem:EsupZp_EZp_estimates} and \eqref{eqn:A:E_sum_X_i_converges} from Assumptions \ref{A:Assumptions_A} to conclude that $\sup_{0\leq \theta \leq \delta}\Prob\left(\bigg|\Wd(\rho^{N}_{\tau+\theta},\rho^{N}_{\tau})\bigg| >\eta\right)$ is bounded by a constant times $\frac{\delta^{1/2}}{\eta}$.
\end{proof}

\subsubsection{Convergence of the spatial mean}\label{sec:convergence_mean}
Let $f \in C^{2}(\R)$. By It{\^o}'s formula we see that $\rho^{N}$ satisfies the following SPDE
\begin{equation}
\begin{split}\label{eqn:SPDE_rho_N}
   \langle \rho^{N}_{t},f\rangle & = \langle \rho^{N}_{0},f\rangle+\int_{0}^{t}\beta \langle \rho^{N}_{r},\partial_{x}f\rangle \langle \rho^{N}_{r},\idx\rangle+\frac{\alpha}{2}\langle \rho^{N}_{r},\idxo\partial_{x}^{2}f\rangle  \langle \rho^{N}_{r},\idx\rangle+\bigg(\frac{1}{2}-\frac{\alpha}{2N}\bigg)\langle \rho^{N}(r),\idxo^{2}\partial_{x}^{2}f\rangle\dr\\
    &\phantom{xx}{}+\frac{1}{N}\sum_{i=1}^{N}\int_{0}^{t}\partial_{x}f(X_{i}(r))\sqrt{\alpha}\langle \rho^{N}_{r},\idx\rangle\frac{\sqrt{X_{i}(r)}}{\sqrt{\langle \rho^{N}_{r},\idx\rangle}}\dW^{i}_{r}+\int_{0}^{t} \langle \rho^{N}_{r},\idxo\partial_{x}f\rangle\sqrt{1-\frac{\alpha}{N}}\dW^{0}_{r}.
\end{split}
\end{equation}
We expect equation \eqref{eqn:SPDE_rho_N}, when taking the limit as $N\rightarrow\infty$, to converge to a non-linear, non-local SPDE, depending on terms of the form $\langle\rho_{t},\idx\rangle$. The goal of this section is to investigate the explicit structure of $\langle\rho^{N}_{t},\idx\rangle$ as $N  \to \infty$. A ``convenient'' representation of the limit of $\langle\rho^{N}_{t},\idx\rangle$ might allow us to work with a simpler linear equation with random coefficients.
We define $Z_{(N)}\coloneqq \langle\rho^{N}_{t},\idx\rangle=\frac{1}{N}\sum_{j=1}^{N}X_{j}$. Taking the empirical average of system \eqref{eqn:X_Ni_SDE} yields
\begin{align*}
    \d Z_{(N)}(t)=\beta Z_{(N)}(t)\dt+\frac{\sqrt{\alpha}}{N}\sum_{i=1}^{N}\sqrt{X_{i}(t)}\sqrt{Z_{(N)}(t)}\dW_{t}^{i}+\sqrt{1-\frac{\alpha}{N}}Z_{(N)}(t)\dW^{0}_{t}.
\end{align*}
The martingale representation theorem allows us to express $\frac{\sqrt{\alpha}}{N}\sum_{i=1}^{N}\int_{0}^{t}\sqrt{X_{i}(s)}\sqrt{Z_{(N)}(s)}\dW_{s}^{i}$ in terms of an integral with respect to a single Brownian motion, which will be denoted by $B$. In particular,
\begin{align*}
    \frac{\sqrt{\alpha}}{N}\sum_{i=1}^{N}\int_{0}^{t}\sqrt{X_{i}(s)}\sqrt{Z_{(N)}(s)}\dW_{s}^{i}=\int_{0}^{t}\sqrt{\frac{\alpha}{N}}Z_{(N)}(s)\d B_{s}.
\end{align*}
Although not strictly necessary, this step improves the readability of the following calculations.
The equation for $Z_{(N)}$ can be rewritten as
\begin{align}\label{eqn:equation_mean_Z}
    \d Z_{(N)}(t)=\beta Z_{(N)}(t)\dt+\sqrt{\frac{\alpha}{N}}Z_{(N)}(t)\d B_{t}+\sqrt{1-\frac{\alpha}{N}}Z_{(N)}(t)\dW^{0}_{t},
\end{align}
which has an explicit solution given by
\begin{align*}
    Z_{(N)}(t)&=
    Z_{(N)}(0)\exp{\bigg(\bigg(\beta-\frac{1}{2}\bigg)t+\sqrt{\frac{\alpha}{N}}B_{t}+\sqrt{1-\frac{\alpha}{N}}W^{0}_{t}\bigg)}.
\end{align*}
Taking the limit $N\rightarrow \infty$, yields 
\begin{align*}
    Z_{(N)}(0)\exp{\bigg(\bigg(\beta-\frac{1}{2}\bigg)t+\sqrt{\frac{\alpha}{N}}B_{t}+\sqrt{1-\frac{\alpha}{N}}W^{0}_{t}\bigg)}\rightarrow m_{\lambda}\exp{\bigg(\bigg(\beta-\frac{1}{2}\bigg)t+W^{0}_{t}\bigg)},
\end{align*}
with $m_{\lambda}$ given in \ref{A:A_3_X(0)_Z(0)}.
More precisely, we consider the sequence of functions $\big(t\mapsto\langle \rho^{N}_t,\idx\rangle\big)_{N\in\N}$ and show that it converges  in probability in  $C([0,T],\R)$ to the function $t\mapsto m_{\lambda}\exp{\bigg(\bigg(\beta-\frac{1}{2}\bigg)t+W^{0}_{t}\bigg)}$,
\begin{lemma}\label{lem:convergence_of_mean_rho_N} Let Assumptions \ref{A:Assumptions_A} be satisfied. 
The sequence $Z_{(N)}=\langle \rho^{N},\idx\rangle$ of $C([0,T],\R)$--valued random variables converges in probability to $\langle\rho_{\cdot}, \idx\rangle\coloneqq m_{\lambda}\exp{\bigg(\bigg(\beta-\frac{1}{2}\bigg)(\cdot)+W^{0}_{\cdot}\bigg)}$. 
Additionally, we have
\begin{align}
   \lim_{N \to \infty} \E\int_{0}^{t}\bigg|\langle \rho^{N}_{s},\idx \rangle-m_{\lambda}\exp{\bigg(\bigg(\beta-\frac{1}{2}\bigg)s+W^{0}_{s}\bigg)}\bigg|\ds\rightarrow 0.
\end{align}
\end{lemma}
\begin{proof}
We have the following estimate
\begin{align*}
    \Prob&\bigg(\sup_{0\leq t\leq T}\bigg|\langle \rho^{N}_{t},\idx\rangle-m_{\lambda}\exp{\bigg(\bigg(\beta-\frac{1}{2}\bigg)t+W^{0}_{t}\bigg)}\bigg|>\eps\bigg)\\
    &=\Prob\bigg(\sup_{0\leq t\leq T}\bigg|Z_{(N)}(0)\exp{\bigg(\bigg(\beta-\frac{1}{2}\bigg)t+\sqrt{\frac{\alpha}{N}}B_{t}+\sqrt{1-\frac{\alpha}{N}}W^{0}_{t}\bigg)}-m_{\lambda}\exp{\bigg(\bigg(\beta-\frac{1}{2}\bigg)t+W^{0}_{t}\bigg)}\bigg|>\eps\bigg)\\
    &\leq  \Prob\bigg(\sup_{0\leq t\leq T}\bigg|Z_{(N)}(0)\exp{\bigg(\sqrt{\frac{\alpha}{N}}B_{t}+\sqrt{1-\frac{\alpha}{N}}W^{0}_{t}-\frac{t}{2}\bigg)}-m_{\lambda}\exp{\bigg(W^{0}_{t}-\frac{t}{2}\bigg)}\bigg|>\eps\exp{\bigg(-\frac{T}{2}-\bigg(\beta-\frac{1}{2}\bigg)T\bigg)}\bigg).
\end{align*}

The exponential factors correspond to the Doleans-Dade exponential of $B+W^{0}$ and $W^{0}$, respectively. Hence we can apply Doob's maximal inequality to get rid of the supremum,
\begin{align*}
   \Prob\bigg(&\sup_{0\leq t\leq T}\bigg|Z_{(N)}(0)\exp{\bigg(\sqrt{\frac{\alpha}{N}}B_{t}+\sqrt{1-\frac{\alpha}{N}}W^{0}_{t}-\frac{t}{2}\bigg)}-m_{\lambda}\exp{\bigg(W^{0}_{t}-\frac{t}{2}\bigg)}\bigg|>\eps\exp{(-\beta T)}\bigg)\\
   &\leq \frac{\E\bigg|Z_{(N)}(0)\exp{\bigg(\sqrt{\frac{\alpha}{N}}B_{T}+\sqrt{1-\frac{\alpha}{N}}W^{0}_{T}-\frac{T}{2}\bigg)}-m_{\lambda}\exp{\bigg(W^{0}_{T}-\frac{T}{2}\bigg)}\bigg|}{\eps\exp{\bigg(\beta T\bigg)}}\\
   &\leq  C\exp{\bigg(\frac{T}{2}\bigg)}\frac{\left(\E\left|Z_{(N)}(0)-m_{\lambda}\right|^{1+\zeta}\right)^{\frac{1}{1+\zeta}}\left(\E\bigg(\exp{\bigg(\sqrt{\frac{\alpha}{N}}B_{T}+\sqrt{1-\frac{\alpha}{N}}W^{0}_{T}\bigg)}^{\frac{1+\zeta}{\zeta}}\bigg)\right)^{\frac{\zeta}{1+\zeta}}}{\eps\exp{\bigg(\beta T\bigg)}}\\
   &\phantom{xx}{}+C\exp{\bigg(\frac{T}{2}\bigg)}\frac{\left(\E\big(m_{\lambda}\big)^{1+\zeta}\right)^{\frac{1}{1+\zeta}}\left(\E\left|\exp{\bigg(\sqrt{\frac{\alpha}{N}}B_{T}+\sqrt{1-\frac{\alpha}{N}}W^{0}_{T}\bigg)}-\exp{\bigg(W^{0}_{T}\bigg)}\right|^{\frac{1+\zeta}{\zeta}}\right)^{\frac{\zeta}{1+\zeta}}}{\eps\exp{\bigg(\beta T\bigg)}}\\
   &= :I+II.
\end{align*}
We use the explicit representation of the Laplace transform of Gaussian random variables to rewrite $I$ as
\begin{align*}
    & C\exp{\bigg(\frac{T}{2}\bigg)}\frac{\left(\E\left|Z_{(N)}(0)-m_{\lambda}\right|^{1+\zeta}\right)^{\frac{1}{1+\zeta}}\left(\exp{\bigg(\left(\frac{1+\zeta}{\zeta}T\right)^{2}\bigg)}\right)^{\frac{\zeta}{1+\zeta}}}{\eps\exp{\bigg(\beta T\bigg)}}.
\end{align*}
This term vanishes as $N\rightarrow\infty$ by Assumption \ref{A:A_3_X(0)_Z(0)}. The term $II$ converges to $0$, since
\begin{align*}
    &\E\left|\exp{\bigg(\sqrt{\frac{\alpha}{N}}B_{T}+\sqrt{1-\frac{\alpha}{N}}W^{0}_{T}\bigg)}-\exp{\bigg(W^{0}_{T}\bigg)}\right|^{\frac{1+\zeta}{\zeta}}\\
    &=\E\left[\left|\exp{\bigg(\sqrt{\frac{\alpha}{N}}B_{T}+\sqrt{1-\frac{\alpha}{N}}W^{0}_{T}\bigg)}-\exp{\bigg(W^{0}_{T}\bigg)}\right|\left|\exp{\bigg(\sqrt{\frac{\alpha}{N}}B_{T}+\sqrt{1-\frac{\alpha}{N}}W^{0}_{T}\bigg)}-\exp{\bigg(W^{0}_{T}\bigg)}\right|^{\frac{1+\zeta}{\zeta}-1}\right]\\
    &\leq \sqrt{\E\left[\left(\exp{\bigg(\sqrt{\frac{\alpha}{N}}B_{T}+\sqrt{1-\frac{\alpha}{N}}W^{0}_{T}\bigg)}-\exp{\bigg(W^{0}_{T}\bigg)}\right)^{2}\right]}\\
    &\phantom{xxxxxxxx}\times\sqrt{\E\left[\left|\exp{\bigg(\sqrt{\frac{\alpha}{N}}B_{T}+\sqrt{1-\frac{\alpha}{N}}W^{0}_{T}\bigg)}-\exp{\bigg(W^{0}_{T}\bigg)}\right|^{2\left(\frac{1+\zeta}{\zeta}-1\right)}\right]}\\
    &\leq \sqrt{\left(2\exp{\bigg(2T^{2}\bigg)}-2\exp{\bigg(\frac{T^{2}}{2}\bigg(\frac{\alpha}{N}+\bigg(\sqrt{1-\frac{\alpha}{N}}+1\bigg)^{2}\bigg)\bigg)}\right)}\\
    &\phantom{xxxxxxxx}\times\sqrt{\E\left[\left(\left|\exp{\bigg(\sqrt{\frac{\alpha}{N}}B_{T}+\sqrt{1-\frac{\alpha}{N}}W^{0}_{T}\bigg)}\right|+\left|\exp{\bigg(W^{0}_{T}\bigg)}\right|\right)^{2\left(\frac{1+\zeta}{\zeta}-1\right)}\right]}\\
    &\leq C_{\zeta, T}\sqrt{\left(2\exp{\bigg(2T^{2}\bigg)}-2\exp{\bigg(\frac{T^{2}}{2}\bigg(\frac{\alpha}{N}+\bigg(\sqrt{1-\frac{\alpha}{N}}+1\bigg)^{2}\bigg)\bigg)}\right)}.
\end{align*}
The continuity of the involved functions yields the convergence to $0$ and proves the first assertion. 
From the previous calculations, we already know that the integrand goes to $0$, uniformly in $t$, in probability. To prove the second assertion, note that by Lemma \ref{lem:EsupZp_EZp_estimates} the function
\begin{align*}
(t, \omega) \mapsto g(t,\omega)\coloneqq \bigg|\langle \rho^{N}_{t},\idx \rangle\bigg|(\omega)+\bigg|m_{\lambda}\exp{\bigg(\bigg(\beta-\frac{1}{2}\bigg)t+W^{0}_{t}\bigg)}\bigg|(\omega)    
\end{align*}
is uniformly integrable since  $\E\langle \rho^{N}_t, \idx \rangle^{1+\zeta}<\infty$,
by Assumption \ref{A:A_3_X(0)_Z(0)}. Hence we can apply Vitali's theorem, which yields the assertion.
\end{proof}

Our next goal is to show that the limit points of (suitable sub-sequences of) the sequence $\rho^{N}$ satisfy the stochastic partial differential equation \eqref{eqn:SPDE_classical_formulation} in the sense of Definition \ref{def:weak_solution}.
\begin{remark}
To be precise, since we are using the results of Lemma \ref{lem:convergence_of_mean_rho_N}, we argue that our limit process $\rho$ is a solution to a linear SPDE with random coefficients (depending on $W^0$) given by 
\begin{align}
    \d \rho_{t}= m_{\lambda}\exp{\big((\beta-1/2)t+W^{0}_{t}\big)}\big(\frac{\alpha}{2}\partial_{x}^{2}(x \rho_{t})-\beta\partial_{x} \rho_{t} \big)\dt+\frac{1}{2}\partial_{x}^{2}(x^{2} \rho_{t})\dt+\partial_{x}(x \rho_{t})\dW^{0}_{t}.\label{eqn:SPDE_random}
\end{align}
We will show in Lemma \ref{lem:rho_x} that, under certain conditions, any solution of equation \eqref{eqn:SPDE_classical_formulation}, whose initial condition is a positive and finite measure, also satisfies 
\begin{align*} \langle\rho_{t},\idx\rangle=m_{\lambda}\exp{\big((\beta \theta-1/2)t+W^{0}_{t}\big)},
\end{align*}
where $\theta = \langle \rho_{0},1\rangle$.
The explicit expression for $\langle\rho,\idx\rangle$ will be also useful later on to show the uniqueness of $\rho$.
\end{remark}
\begin{lemma}\label{lem:rho_x}
  Let $\rho$ be any (signed), finite measure-valued solution of \eqref{eqn:SPDE_classical_formulation} with $\int_{\R_{+}}|x||\rho|_{t}(\dx)<\infty$ for a.e. $t$ and $\omega$, such that
\begin{align*}
    &\langle \rho_{0}, 1\rangle=\theta>0,\quad \langle \rho_{0}, \idx\rangle=m_{\lambda}>0,\quad \textnormal{for almost every } \omega,
\end{align*}
then
\begin{enumerate}[label=(\Roman*)]
\item \leavevmode\vspace*{-\dimexpr\abovedisplayskip + \baselineskip}
\begin{align*}
   \langle \rho_{t}, 1\rangle=\theta,\quad \textnormal{for almost every } t,\omega,
\end{align*}
\item \leavevmode\vspace*{-\dimexpr\abovedisplayskip + \baselineskip} \begin{align*}
    \langle\rho_{t},\idx\rangle=m_{\lambda}\exp{((\beta \theta-1/2)t+W^{0}_{t})}.
\end{align*}
\end{enumerate}
\end{lemma}

\begin{proof}
Let $R\in \R,\,R>0$. We define a test function, which is symmetric around $0$, by
\begin{align*}
    \varphi(x)=
    \begin{cases} 1, \quad \text{for } x \in B_{R}(0), \\
                  0, \quad \text{for } x \in \R \backslash B_{R+4}(0) ,
    \end{cases}
\end{align*}
where $B_{R}(y)\coloneqq \{x\in \R\colon |x-y|<R\}$ and $\left|D^{k} \varphi(x)\right| \leq 1$ for all $ x \in B_R, \, k \leq 2$ and $\varphi(\eps \cdot) \rightarrow 1$ as $\eps\rightarrow 0$ monotonically. We introduce the re-scaled function $ \varphi(\eps x)$, for which we have
$\left|D^{k} \varphi(\eps x)\right| \leq C \eps^{k} $ for all  $x \in \R, k \leq 2$. An explicit construction of such a function is given in Appendix \ref{sec:contruction_of_test_function}. 
Without changing the notation, we can restrict the constructed function to $\R_{+}$.
 By the assumptions on $\varphi$, $x\mapsto D^{k}\varphi(\eps x)$, for $k=1,2$, is supported in the annulus $[R/\eps, (R+4)/\eps]$ and $\left(\frac{x}{\eps}\right)^{k}\eps^{k}\approx R+4$. Let $0\leq m,n \leq 2$, then the following estimate holds
\begin{align}\label{eqn:rho_xmDn_phi_estimate}
    \langle \rho_{r}, x^{m} D^{n}(\varphi(\eps\cdot)) \rangle &= C_{n}\langle \rho_{r}, x^{m} \eps^{n} (D^{n}\varphi)(\eps\cdot) \rangle\leq  C_{n} \langle |\rho_{r}|,\eps^{n}|(\cdot_{x})|^{m}
     \chi_{[R/\eps,(R+4)/\eps]}(|\cdot_{x}|)\rangle \nonumber\\
     &=C_{n} \int_{\R_{+}}\eps^{n} |x|^{m} \chi_{[R/\eps,(R+4)/\eps]}(x)|\rho|_{r}(\dx) \nonumber\\
    &\leq C_{n} \int_{R/\eps}^{(R+4)/\eps} \eps^{n} |x|^{m} |\rho|_{r}(\dx) =C_{n}\int_{R}^{(R+4)} \eps^{n} |x|^{m} (\eps_{\text{\#}}|\rho|_{r})(\dx)\\
    &\leq  C_{n} \eps^{n}(R+4)^{m}\left(|\rho|_{r}([\eps R,\eps (R+4)])\right)\leq C_{n} \eps^{n}(R+4)^{m}|\rho|_{r}(\R_{+}).\nonumber
\end{align}

To prove the assertions, we
 use $\idx^{k}\varphi(\eps \cdot)$ ($x\mapsto x^{k}\varphi(\eps x)$) as  test function. To help identify the terms that will appear in each case, we introduce the indicator functions $\chi_{\{\cdot\}}$ which will be $1$ if the condition in the subscript is met and $0$ otherwise. This yields

\begin{align}\label{eqn:test}
     \langle \rho_{t},&\idx^{k}\varphi(\eps \cdot)\rangle 
    =\langle \rho_{0},\idx^{k}\varphi(\eps \cdot)\rangle+\int_{0}^{t}\beta \chi_{\{k\geq 0\}} \langle \rho_{r},\eps\idx^{k}\partial_{x}\varphi(\eps\cdot)\rangle\langle\rho_{r},\idx\rangle\dr\nonumber\\
     &\phantom{xx}{}+\int_{0}^{t}\chi_{\{k\geq 0\}}\frac{\alpha}{2}\langle \rho_{r},\eps^{2}\idx^{k+1}\partial_{x}^{2}\varphi(\eps\cdot)\rangle \langle\rho_{r},\idx\rangle\dr+\int_{0}^{t}\chi_{\{k\geq 0\}}\frac{1}{2}\langle \rho_{r},\idx^{k+2}\eps^{2}\partial_{x}^{2} \varphi(\eps\cdot)\rangle\dr \nonumber\\
     &\phantom{xx}+\int_{0}^{t}\chi_{\{k\geq 0\}} \langle \rho_{r},\eps \idx^{k} \partial_{x}\varphi(\eps\cdot)\rangle\dW^{0}_{r}\\
     %%%% k-1 terms
     &\phantom{xx}{}+\int_{0}^{t}\chi_{\{k\geq 1\}}\beta \langle \rho_{r}, k\idx^{k-1}\varphi(\eps\cdot) \rangle\langle\rho_{r},\idx\rangle\dr +2\int_{0}^{t}\chi_{\{k\geq 1\}}\frac{\alpha}{2}\langle \rho_{r},x\eps^{2}k\idx^{k-1}\partial_{x}\varphi(\eps\cdot)\rangle \langle\rho_{r},\idx\rangle\dr \nonumber\\
    &\phantom{xx}{}+\int_{0}^{t}\chi_{\{k\geq 1\}}\langle \rho_{r},x^{2}k\idx^{k-1}\eps^{2}\partial_{x} \varphi(\eps\cdot)\rangle\dr +\int_{0}^{t}\chi_{\{k\geq 1\}} \langle \rho_{r},x k\idx^{k-1}\varphi(\eps\cdot)\rangle\dW^{0}_{r} \nonumber\\
       %%%% k-2 terms
    &\phantom{xx}{}+\int_{0}^{t}\chi_{\{k\geq 2\}}\frac{\alpha}{2}\langle \rho_{r},x\eps^{2}(k(k-1))\idx^{k-2}\varphi(\eps\cdot)\rangle \langle\rho_{r},\idx\rangle\dr +\int_{0}^{t}\chi_{\{k\geq 2\}}\langle \rho_{r},x^{2}k(k-1)\idx^{k-2}\eps^{2} \varphi(\eps\cdot)\rangle\dr.\nonumber
\end{align}
Due to the restriction $k=0,1$, terms including $\idx^{k-2}$ are $0$ and terms including $\idx^{k-1}$ are only present in the case $k=1$.
The deterministic integrals in \eqref{eqn:test} can now be easily estimated by \eqref{eqn:rho_xmDn_phi_estimate}, e.g. the first and third one by
\begin{align*}
    &\int_{0}^{t}\beta \langle \rho_{r},\eps \idx^{k} \partial_{x}\varphi(\eps\cdot) \rangle\langle\rho_{r},\idx\rangle\dr \leq  C_{1} \int_{0}^{t}\beta \eps(R+4)^{k} |\rho|_{r}(\R) \left|\langle |\rho_{r}|,\idx\rangle\right|\dr,
\end{align*}
and
\begin{align*}
    &\int_{0}^{t}\frac{1}{2}\langle \rho_{r},x^{2}\idx^{k}\eps^{2}\partial_{x}^{2} \varphi(\eps\cdot)\rangle\dr \leq  C_{2}  \int_{0}^{t} \eps^{2}(R+4)^{k+2} |\rho|_{r}(\R) \dr,
\end{align*}
respectively.
In the case $k=0$, as $\eps\rightarrow 0$ (for fixed $R$) all remaining integral deterministic integrals vanish. The stochastic integrals can be treated by similar estimates and a Borell-Cantelli argument which yields that, along a suitable subsequence, they converge to $0$ a.s. as $\eps\rightarrow 0$.
 Since $\eps \mapsto \varphi(\eps x)$ is monotonically increasing by assumption, and as $\varphi(\eps \cdot) \rightarrow 1$ as $\eps\rightarrow 0$, it follows from the monotone convergence theorem that
\begin{align*}
    \langle \rho_{t}, 1\rangle=\langle \rho_{0}, 1\rangle=\theta.
\end{align*}
In the case $k=1$, as $\eps\rightarrow 0$ (for fixed $R$) all integrals except for $\int_{0}^{t}\beta \langle \rho_{r}, \varphi(\eps\cdot) \rangle\langle\rho_{r},\idx\rangle\dr$ and $\int_{0}^{t} \langle \rho_{r},x \varphi(\eps\cdot)\rangle\dW^{0}_{r}$ vanish. 
Since $\varphi_{\eps}(x)$ is monotonically increasing, as $\eps\rightarrow 0$ and $\varphi_{\eps}\rightarrow 1$, be can argue by the monotone convergence theorem, that
\begin{align*}
    \int_{0}^{t}\beta \langle \rho_{r}, \varphi_{\eps}\rangle\langle\rho_{r},\idx\rangle\dr\rightarrow 
    \int_{0}^{t}\beta \langle \rho_{r}, 1\rangle\langle\rho_{r},\idx\rangle\dr.
\end{align*}
A similar argument, using the dominated convergence theorem, holds for the stochastic integral, yielding
\begin{align*}
    \int_{0}^{t} \langle \rho_{r},(\cdot_{x})^{2}\eps \partial_{x}\varphi(\eps\cdot)+(\cdot_{x})\varphi_{\eps}\rangle\dW^{0}_{r}\rightarrow \int_{0}^{t} \langle \rho_{r},\idx\rangle\dW^{0}_{r}.
\end{align*}
In the limit, we obtain
\begin{align*}
       \langle \rho_{t},\idx\rangle &=\langle \rho_{0},\idx \rangle+\int_{0}^{t}\beta \langle \rho_{r},1\rangle\langle\rho_{r},\idx\rangle+\int_{0}^{t} \langle \rho_{r},\idx\rangle\dW^{0}_{r},
\end{align*}
whose unique solution is given by 
\begin{align*}
    \langle \rho_{t},x\rangle=\langle \rho_{0},x\rangle\exp{((\beta \langle \rho_{t},1\rangle -1/2)t+W^{0}_{t})}.
\end{align*}
\end{proof}
\subsubsection{Proof of Proposition \ref{prop:convergence_to_weak_solution_SPDE}}

\begin{proof}
Let $\varphi\in \mathcal{S}_{0}$.
We first rewrite the dynamics of our particle system on our original probability space.

 \begin{align*}   
    \langle \rho^{N}_{t},\varphi\rangle &=\langle \rho^{N}_{0},\varphi\rangle\\
    &\phantom{xx}{}+\int_{0}^{t}\beta \langle \rho^{N}_{r},\partial_{x}\varphi\rangle \langle \rho^{N}_{r},\idx\rangle+\frac{\alpha}{2}\langle \rho^{N}_{r},\idxo\partial_{x}^{2}\varphi\rangle  \langle \rho^{N}_{r},\idx\rangle+\bigg(\frac{1}{2}-\frac{\alpha}{2N}\bigg)\langle \rho^{N}_{r},\idxo^{2}\partial_{x}^{2}\varphi\rangle\dr\\
    &\phantom{xx}{}+\frac{1}{N}\sum_{i=1}^{N}\int_{0}^{t}\partial_{x}\varphi(X_{i}(r))\sqrt{\alpha}\sqrt{\langle \rho^{N}_{r},\idx\rangle}\sqrt{X_{i}(r)}\dW^{i}_{r}\\
    &\phantom{xx}{}+\int_{0}^{t} \langle \rho^{N}_{r},\idxo\partial_{x}\varphi\rangle\sqrt{1-\frac{\alpha}{N}}\dW^{0}_{r}.
\end{align*}
Let us now compute the quadratic variaton of $M_{t}^{\partial_{x}\varphi,N}\coloneqq\frac{1}{N}\sum_{i=1}^{N}\int_{0}^{t}\partial_{x}\varphi(X_{i}(r))\sqrt{\alpha}\sqrt{X_{i}(r)}\sqrt{\frac{1}{N}\sum_{j=1}^{N }X_{j}(r)}\dW^{i}_{r}$ 
in terms of $\rho^{N}$,
\begin{align*}
    [M_{\cdot}^{\partial_{x}\varphi,N},M_{\cdot}^{\partial_{x}\varphi,N} ]_{t}=\frac{\alpha}{N^{2}}\sum_{i=1}^{N}\int_{0}^{t}(\partial_{x}\varphi(X_{i}(r)))^{2}X_{i}(r)\frac{1}{N}\sum_{j=1}^{N }X_{j}(r)\dr=\frac{\alpha}{N}\int_{0}^{t}\langle \rho^{N}_{r},\idxo(\partial_{x}\varphi)^{2} \rangle \langle \rho^{N}_{r},\idx\rangle \dr.
\end{align*}
There exists an auxiliary Brownian motion, denoted by $B$, independent of $W^{0}$, such that 
\begin{align*}
    M_{t}^{\partial_{x}\varphi,N}=\sqrt{\frac{\alpha}{N}}\int_{0}^{t}\sqrt{\langle \rho^{N}_{r}, \idxo(\partial_{x}\varphi)^{2}\rangle \langle\rho^{N}_{r},\idx\rangle}\d B_{s}.
\end{align*}
Hence we obtain the expression
\begin{equation}\label{eq:SPDE_N}
\begin{split}
    \langle \rho^{N}_{t},\varphi\rangle &=\langle \rho^{N}_{0},\varphi\rangle\\
     &\phantom{xx}{}+\int_{0}^{t}\beta \langle \rho^{N}_{r},\partial_{x}\varphi\rangle \langle \rho^{N}_{r},\idx\rangle+\frac{\alpha}{2}\langle \rho^{N}_{r},\idxo\partial_{x}^{2}\varphi\rangle  \langle \rho^{N}_{r},\idx\rangle+\bigg(\frac{1}{2}-\frac{\alpha}{2N}\bigg)\langle \rho^{N}_{r},\idxo^{2}\partial_{x}^{2}\varphi\rangle\dr\\
    &\phantom{xx}{}+\sqrt{\frac{\alpha}{N}}\int_{0}^{t}\sqrt{\langle \rho^{N}_{r}, \idxo(\partial_{x}\varphi)^{2} \rangle \langle \rho^{N}_{r},\idx\rangle}\d B_{s}+\int_{0}^{t} \langle \rho^{N}_{r},\idxo\partial_{x}\varphi\rangle\sqrt{1-\frac{\alpha}{N}}\dW^{0}_{r}.
\end{split}
\end{equation}
We recall the space $C([0,T],M_{1}((\R_{+}))$, which we equipped with the topology induced by the distance $d_{Z}(\mu,\nu)\coloneqq \sup_{t\in [0,T]}|\Wd(\mu_{t},\nu_{t})|$.
Due to the tightness of the laws $(Q^N_T)_{N \in \N}$ of $(\rho^{N})_{N\in\N}$ on $C([0,T],M_{1}((\R_{+}))$ obtained in Theorem \ref{thm:tightness} and Prohorov's thorem, there exists a subsequence $N_{k}$ and a limiting object $Q_{T}$, such that $Q^{N_{k}}_{T}\rightharpoonup Q_{T}$, where ``$\rightharpoonup$'' denotes again the weak convergence of measures. In other words, for every continuous and bounded function $f$ from $C([0,T],M_{1}(\R_{+}))$ to $\R$,
\begin{align*}
    \int_{C([0,T],M_{1})}f(\mu)Q^{N_{k}}_{T}(\d \mu)\rightarrow \int_{C([0,T],M_{1})}f(\mu)Q_{T}(\d \mu).
\end{align*}
We can apply the Skorohod representation theorem (see for example \cite{Jakubowski98_Skorohod}), which yields the existence of a subsequence that is not relabeled, a probability space $(\widetilde{\Omega},\widetilde{\mathcal{F}},\widetilde{\Prob})$ and, on this space,
$C([0,T],M_{1})\times C([0,T];\R)\times C([0,T];\R)$-valued random variables $\big(\widetilde{\rho}, \widetilde{B}, \widetilde{W}\big)$
and $\big(\widetilde{\rho}^{N}, \widetilde{B}^{N}, \widetilde{W}^{N}\big)_{n\in \N}$ such that
\begin{align*}
    \big(\widetilde{\rho}^{N}, \widetilde{B}^{N}, \widetilde{W}^{N}\big) \text{ has the same law as }\big(\rho^{N},B, W^{0}\big),
\end{align*}
on the Borel sigma-algebra  $\mathcal{B}(C([0,T],M_{1})\times C([0,T];\R)\times C([0,T];\R))$ and, as $N\to \infty$,
\begin{align*}
    \big(\widetilde{\rho}^{N}, \widetilde{B}^{N}, \widetilde{W}^{N}\big)\to \big(\widetilde{\rho}, \widetilde{B}, \widetilde{W}\big)\quad \mbox{in }C([0,T],M_{1})\times C([0,T];\R)\times C([0,T];\R)\quad\widetilde{\Prob}\mbox{-a.s.}
\end{align*}

We change to the new variables $\big(\widetilde{\rho}^{N}, \widetilde{B}^{N}, \widetilde{W}^{N}\big)$ obtained by Skorohod's representation theorem. 
The first observation, see \cite[Theorem 2.9.1]{Breit_Feireisl_Hofmanova18stochastic_fluid_flows} or \cite{Bensoussan95_stochastic_navier_stokes,Brzezniak10_stochastic_reaction_diffusion_jump_processes}, is that they also satisfy equation \eqref{eq:SPDE_N} $\widetilde{\Prob}$--a.s., replacing $(\rho^N, B, W^0)$ by $(\widetilde{\rho}^N, \widetilde{B}^N, \widetilde {W}^N)$.

Next, we want to use the $\widetilde{\Prob}$-a.s. convergence of $\widetilde{\rho}^{N}$ to $\widetilde{\rho}$ in $C([0,T],M_{1})$, i.e.
\begin{align*}
    \widetilde{\Prob}\left(\lim_{N\rightarrow \infty}\sup_{t\in [0,T]}|\Wd(\widetilde{\rho}^{N}_{t},\widetilde{\rho}_{t})|=0 \right)=1
\end{align*}
to show that $\widetilde{\rho}$ satisfies \eqref{eqn:SPDE_weak_formulation}, or in other words, $(\widetilde{\Omega},\widetilde{\mathcal{F}},\widetilde{\Prob})$, $\big(\widetilde{\rho}, \widetilde{W}\big)$ is a solution given by Definition \ref{def:weak_solution}. 
Hence, we will show that the following convergences hold
\begin{align}
    &\lim_{N\rightarrow \infty}\widetilde{\E}\left|\langle \widetilde{\rho}^{N}_{t}-\widetilde{\rho}_{t},\varphi\rangle\right|=0 \label{eqn:convergence_rho_t}\\
    &\lim_{N\rightarrow \infty}\widetilde{\E}\left|\langle \widetilde{\rho}^{N}_{0}-\widetilde{\rho}_{0},\varphi\rangle\right|=0\label{eqn:convergence_rho_0}\\
    &\lim_{N\rightarrow \infty}\widetilde{\E}\left|\int_{0}^{t}\beta \langle \widetilde{\rho}^{N}_{r},\partial_{x}\varphi\rangle \langle \widetilde{\rho}^{N}_{r},\idx\rangle-\beta \langle \widetilde{\rho}_{r},\partial_{x}\varphi\rangle \langle \widetilde{\rho}_{r},\idx\rangle\dr\right|=0\label{eqn:convergence_dx_phi_rho_x}\\
    &\lim_{N\rightarrow \infty}\widetilde{\E}\left|\int_{0}^{t}\frac{\alpha}{2}\langle \widetilde{\rho}^{N}_{r},\idxo\partial_{x}^{2}\varphi\rangle  \langle \widetilde{\rho}^{N}_{r},\idx\rangle-\frac{\alpha}{2}\langle \widetilde{\rho}_{r},\idxo\partial_{x}^{2}\varphi\rangle  \langle \widetilde{\rho}_{r},\idx\rangle \dr\right|=0\label{eqn:convergence_xd2x_phi_rho_x}\\
    &\lim_{N\rightarrow \infty}\widetilde{\E}\left|\int_{0}^{t}\bigg(\frac{1}{2}-\frac{\alpha}{2N}\bigg)\langle \widetilde{\rho}^{N}_{r},\idxo^{2}\partial_{x}^{2}\varphi\rangle-\frac{1}{2}\langle \widetilde{\rho}_{r},\idxo^{2}\partial_{x}^{2}\varphi\rangle\dr\right|=0\label{eqn:convergence_x2d2x_phi}\\
    &\lim_{N\rightarrow \infty}\widetilde{\E}\left|\sqrt{\frac{\alpha}{N}}\int_{0}^{t}\sqrt{\langle \widetilde{\rho}^{N}_{r}, \idxo(\partial_{x}\varphi)^{2} \rangle \langle \widetilde{\rho}^{N}_{r},\idx\rangle}\d \widetilde{B}^{N}_{r}\right|=0\label{eqn:convergence_dB_integral}\\
    &\lim_{N\rightarrow \infty}\widetilde{\E}\left|\int_{0}^{t} \langle \widetilde{\rho}^{N}_{r},\idxo\partial_{x}\varphi\rangle\sqrt{1-\frac{\alpha}{N}}\d\widetilde{W}^{N}_{r}-\int_{0}^{t} \langle \widetilde{\rho}_{r},\idxo\partial_{x}\varphi\rangle\d\widetilde{W}_{r}\right|=0.\label{eqn:convergence_dW_integral}
\end{align}

Since $\widetilde{\Law}_{\widetilde{\rho}^{N}}=\widetilde{\Prob}(\widetilde{\rho}^{N}\in \cdot)=\Prob(\rho^{N}\in \cdot)=\Law_{\rho^N}$ and $\widetilde{\rho}_{t}\in M_{1}(\R)$, using Assumption \ref{A:A_3_X(0)_Z(0)} yields
\begin{equation}\label{eqn:E_1+zeta_rho_n-rho}
\begin{split}
    \widetilde{\E}\left|\langle \widetilde{\rho}^{N}_{t}-\widetilde{\rho}_{t},\varphi\rangle\right|^{1+\zeta}&\leq C\widetilde{\E}\left|\langle \widetilde{\rho}^{N}_{t},\varphi\rangle\right|^{1+\zeta}+C\widetilde{\E}\left|\langle \widetilde{\rho}_{t},\varphi\rangle\right|^{1+\zeta}\\
    &\leq  C\widetilde{\E}\left|\langle \widetilde{\rho}^{N}_{t},\varphi\rangle\right|^{1+\zeta}+\liminf_{N\rightarrow \infty} C\|\varphi\|_{\infty}\widetilde{\E}\left|\langle \widetilde{\rho}^{N}_{t},1\rangle\right|^{1+\zeta}<\infty.
\end{split}
\end{equation}
Hence, $\left|\langle \widetilde{\rho}^{N}_{t}-\widetilde{\rho}_{t},\varphi\rangle\right|$ is uniformly integrable.
We note that, for every $\psi \in C_{b}(\R_{+})$, the sequences of continuous functions $t\mapsto   \langle \widetilde{\rho}^{N}(t),\psi\rangle$ and $t\mapsto   \langle \widetilde{\rho}^{N}(t),\idx\rangle$ converge in the space $C([0,T];\R)$, $\widetilde{\Prob}$-a.s.
This yields \eqref{eqn:convergence_rho_t} and \eqref{eqn:convergence_rho_0}.
Since the derivatives of Schwartz-functions will be continuous and bounded, we can replace the terms $\partial_{x}\varphi$, $x\partial^{2}_{x}$ and $x^{2}\partial^{2}_{x}$ in \eqref{eqn:convergence_dx_phi_rho_x}, \eqref{eqn:convergence_xd2x_phi_rho_x} and \eqref{eqn:convergence_x2d2x_phi}, respectively, by a continuous and bounded function $\psi$ and use \eqref{eqn:E_1+zeta_rho_n-rho}. Abusing the notation and in order to keep the argument more concise, we shall not specify to which particular term $\psi$ corresponds to. In order to deal with \eqref{eqn:convergence_dx_phi_rho_x}, \eqref{eqn:convergence_xd2x_phi_rho_x} and \eqref{eqn:convergence_x2d2x_phi} we estimate

\begin{align*}
    \widetilde{\E}\bigg|\int_{0}^{t}\langle \widetilde{\rho}^{N}_{r},\psi\rangle&\langle \widetilde{\rho}^{N}_{r},\idx\rangle\dr-\int_{0}^{t}\langle \widetilde{\rho}_{r},\psi\rangle\langle \widetilde{\rho}_{r},\idx\rangle\dr\bigg|\leq \widetilde{\E}\int_{0}^{t} \bigg|\langle \widetilde{\rho}^{N}_{r},\psi\rangle\langle \widetilde{\rho}^{N}_{r},\idx\rangle-\langle \widetilde{\rho}_{r},\psi\rangle\langle \widetilde{\rho}_{r},\idx\rangle\bigg|\dr\\
    &\leq \widetilde{\E}\int_{0}^{t} \bigg|\langle \widetilde{\rho}^{N}_{r},\psi\rangle\bigg|\bigg|\langle \widetilde{\rho}^{N}_{r},\idx\rangle-\langle \widetilde{\rho}_{r},\idx\rangle\bigg|\dr+\widetilde{\E}\int_{0}^{t} \bigg|\langle \widetilde{\rho}^{N}_{r},\psi\rangle-\langle \widetilde{\rho}_{r},\psi\rangle\bigg|\bigg|\langle \widetilde{\rho}_{r},\idx\rangle\bigg|\dr\\
    &\leq \widetilde{\E}\sup_{r\in[0,t]}\bigg|\langle \widetilde{\rho}^{N}_{r},\idx\rangle-\langle \widetilde{\rho}_{r},\idx\rangle\bigg|\int_{0}^{t} \bigg|\langle \widetilde{\rho}^{N}_{r},\psi\rangle\bigg|\dr+\widetilde{\E}\sup_{r\in [0,t]}\bigg|\langle \widetilde{\rho}^{N}_{r},\psi\rangle-\langle \widetilde{\rho}_{r},\psi\rangle\bigg|\int_{0}^{t} \bigg|\langle \widetilde{\rho}_{r},\idx\rangle\bigg|\dr\\
    &\leq \bigg(\widetilde{\E}\bigg(\sup_{r\in [0,t]}\bigg|\langle \widetilde{\rho}^{N}_{r},\idx\rangle-\langle \widetilde{\rho}_{r},\idx\rangle\bigg|\bigg)^{p}\bigg)^{1/p}\bigg(\widetilde{\E}\bigg(\int_{0}^{t} \bigg|\langle \widetilde{\rho}^{N}_{r},\psi\rangle\bigg|\dr\bigg)^{q}\bigg)^{1/q}\\
    &\phantom{xx}{}+\bigg(\widetilde{\E}\bigg(\sup_{r\in [0,t]}\bigg|\langle \widetilde{\rho}^{N}_{r},\psi\rangle-\langle \widetilde{\rho}_{r},\psi\rangle\bigg|\bigg)^{q}\bigg)^{1/q}\bigg(\widetilde{\E}\bigg(\int_{0}^{t} \bigg|\langle \widetilde{\rho}_{r},\idx\rangle\bigg|\dr\bigg)^{p}\bigg)^{1/p} = \underbrace{I_{a}I_{b}}_{I}+\underbrace{II_{a}II_{b}}_{II},
\end{align*}
where $1< p\leq 1+\eps$, for any $0<\eps<\zeta$, for $\zeta$ from Assumption \ref{A:A_3_X(0)_Z(0)}. Our goal is to show that $I$ and $II$ converge to $0$. Indeed, note first that $\widetilde{\E}\bigg(\sup_{r\in [0,t]}\bigg|\langle \widetilde{\rho}^{N}_{r},\idx\rangle-\langle \widetilde{\rho}_{r},\idx\rangle\bigg|\bigg)^{1+\zeta}$ is bounded by  Lemma \ref{lem:EsupZp_EZp_estimates} and Assumption \ref{A:A_3_X(0)_Z(0)}, where we used that $\widetilde{\rho}^{N}$ and $\rho^{N}$ have the same law. Since $p<1+\zeta$,  the term $\left(\sup_{r\in [0,t]}\bigg|\langle \widetilde{\rho}^{N}_{r},\idx\rangle-\langle \widetilde{\rho}_{r},\idx\rangle\bigg|\right)^{p}$ is uniformly integrable.  
Moreover, $\sup_{r\in[0,t]}\bigg|\langle \widetilde{\rho}^{N}_{r},\idx\rangle-\langle \widetilde{\rho}_r, \idx \rangle\bigg|$ converges in probability by Lemma \ref{lem:convergence_of_mean_rho_N}.
Hence $I_{a}$ converges to $0$ as $N\rightarrow \infty$. The term $I_{b}$ does not pose any issues as we can use the boundedness of $\psi$ to estimate $|\langle \rho^{N},\psi\rangle|\leq\|\psi\|_{\infty}$.

The term $II_{a}$ converges to $0$ by Vitali's theorem. Indeed, let $\delta>0$, then it holds that
\begin{align*}
    \sup_{r\in [0,t]}\bigg|\langle \widetilde{\rho}^{N}_{r},\psi\rangle-\langle \widetilde{\rho}_{r},\psi\rangle\bigg|^{q+\delta}\leq C \sup_{r\in [0,t]}\bigg|\langle \widetilde{\rho}^{N}_{r},\psi\rangle\bigg|^{q+\delta}+\sup_{r\in [0,t]}\bigg|\langle \widetilde{\rho}_{r},\psi\rangle\bigg|^{q+\delta}
\end{align*}
and 
\begin{align*}
    \widetilde{\E}&  \sup_{r\in [0,t]}\bigg|\langle \widetilde{\rho}^{N}_{r},\psi\rangle\bigg|^{q+\delta}+\widetilde{\E}\sup_{r\in [0,t]}\bigg|\langle \widetilde{\rho}_{r},\psi\rangle\bigg|^{q+\delta}\\
    &\leq  \E  \sup_{r\in [0,t]}\bigg|\langle \rho^{N}(r),\psi\rangle\bigg|^{q+\delta}+\liminf_{N\rightarrow \infty}\E  \sup_{r\in [0,t]}\bigg|\langle \rho^{N}(r),\psi\rangle\bigg|^{q+\delta}\leq 2\|\psi\|_{\infty}^{q+\delta}.
\end{align*}
Hence, we obtain uniform 
integrability of $ \sup_{r\in [0,t]}\bigg|\langle \widetilde{\rho}^{N}_{r},\psi\rangle-\langle \widetilde{\rho}_{r},\psi\rangle\bigg|^{q}$ and can use Vitali's convergence theorem to conclude that $II_{a}\rightarrow 0$ as $N\rightarrow \infty$. 
The term $II_{b}$ can be bounded by a constant due to Lemma \ref{lem:EsupZp_EZp_estimates} and Assumption \ref{A:A_3_X(0)_Z(0)}.

This argument covers \eqref{eqn:convergence_dx_phi_rho_x} and \eqref{eqn:convergence_xd2x_phi_rho_x}. Equation \eqref{eqn:convergence_x2d2x_phi} follows from the above argument with $\langle\cdot,\idx\rangle$ replaced by $1$. It remains to prove the convergence of the stochastic integrals. These are handled slightly differently. 
By the Burkholder-Davis-Gundy (BDG) and Jensen's inequality,
 \begin{align*}
     \widetilde{\E}&\left|\sqrt{\frac{\alpha}{N}}\int_{0}^{t}\sqrt{\langle \widetilde{\rho}^{N}_{r},\idxo(\partial_{x}\varphi)^{2} \rangle \langle \widetilde{\rho}^{N}_{r},\idx\rangle}\d \widetilde{B}^{N}_{s}\right|=\sqrt{\frac{\alpha}{N}}\E\left(\int_{0}^{T}|\langle \widetilde{\rho}^{N}_{r}, \idxo(\partial_{x}\varphi)^{2} \rangle \langle \widetilde{\rho}^{N}_{r},\idx\rangle\ds\right)^{1/2}\\
     &\leq \sqrt{\frac{\alpha}{N}}\|\idxo(\partial_{x}\varphi)^{2}\|_{\infty}^{1/2}\left(\E\int_{0}^{T}| \langle \widetilde{\rho}^{N}_{r},\idx\rangle|\ds\right)^{1/2}\leq C_{\alpha}\sqrt{\frac{\|\idxo(\partial_{x}\varphi)^{2}\|_{\infty}}{N}}\rightarrow 0,
 \end{align*}
 as $N \to \infty$, since $\varphi\in\mathcal{S}$. This yields \eqref{eqn:convergence_dB_integral}. The argument for the last integral is a direct adaptation of \cite[Lemma 2.6.6, step 3]{Breit_Feireisl_Hofmanova18stochastic_fluid_flows} or \cite[Lemma 2.1]{DGHT11_Local_martingale_solutions_abstract_fluids} to the case of convergence in $L^{1}(\Omega)$ by the BDG inequality. 
By standard arguments, see e.g. \cite{Brzezniak_17_invariant_measures_navier_stokes}, we conclude that $(\widetilde{\Omega},\widetilde{\mathcal{F}},\widetilde{\Prob})$ and $\big(\widetilde{\rho}, \widetilde{W}\big)$ satisfy the conditions of Definition~\ref{def:weak_solution} and are a probabilistically weak solution to the SPDE \eqref{eqn:SPDE_classical_formulation}. 
\end{proof}

\subsection{Uniqueness of the SPDE solution}\label{sec:uniqueness} Next, we proceed with the second part of the proof of Theorem \ref{thm:introduction_convergence_to_SPDE}, namely the uniqueness of solutions for equation \eqref{eqn:SPDE_weak_formulation}.

SPDEs with coefficients degenerating at the boundary of a domain have been studied extensively, see \cite{krylov_99_sobolev_space_constant_coefficient_half_space,krylov_99_weighted_spaces_Rd_plus,lototsky_00_sobolev_spaces_with_weights,lototsky_01_SPDE_degenerating_boundary_domain,kim_07_sobolev_theory_degenerating} and the references therein. The spaces used in these works however do not harmonize very well with operators of the form $(x+x^{2})\Delta_{x}$, due to the two different powers of the multiplication operator, which makes them not applicable. We also want to highlight the more recent work \cite{zhang_11_SPDE_unbounded_degenerate}, which, after some technical adaptations, could be applicable to our situation and provide an unweighted solution theory. In order to avoid these technicalities, we opt to work in so-called ``well-weighted'' Sobolev spaces \cite{Gyongy_Krylov_90_unbounded_coefficients_1,Gyongy_Krylov_90_unbounded_coefficients_2,Gyongy_Krylov_92_unbounded_coefficients_3}. For convenience of the reader, we recall the basics and for the sake of readability, we focus on the case of $\R_{+}$ with specific weights. The general setting can be found in \cite[Section 3]{Gyongy_Krylov_90_unbounded_coefficients_1}.

\begin{definition}\label{def:fractionalSobolev}
Let $m\in \N$, $p\geq 1$. We define the space $W^{m,p}(a,b)=W^{m,p}(a,b,\R_{+})$ as the normed space of locally integrable real functions on $[0,\infty)$, whose weak or generalized (see e.g.~\cite[Chapter 8]{Brezis_10_functional_analysis}) derivatives up to order $m$ are functions such that
\begin{align}\label{eqn:weighted_Sobolev_norm}
    \|f\|_{W^{m,p}(a,b)}\coloneqq \bigg(\sum_{\gamma=0}^{m}\int_{\R_{+}}\big|(1+|x|^{2})^{a+\gamma b} \partial_{x}^{\gamma}f(x)\big|^{p}\bigg)^{1/p}<\infty.
\end{align}
The closure of $C_{c}^{\infty}(\mathbb{R}_{++})$, the smooth functions with compact support in $\mathbb{R}_{++}$, with respect to the norm $\|\cdot\|_{W^{m,p}(a,b)}$ will be denoted by $W^{m,p}_{0}(a,b)$ and corresponds to the $W^{m,p}(a,b)$ functions which vanish at the boundary. (To be exact, one should speak of functions whose boundary trace is equal to $0$, see e.g.~\cite[Chapter 5.5]{evans_22_partial}).
\end{definition}

\begin{remark}
    Compared to what is often found in the literature, namely $b=0$, this more involved definition of a weighted space enables us to fine-tune and match the weights appearing in our equation. By the product rule, we see that $\rho$ in equation \eqref{eqn:SPDE_classical_formulation} itself does not appear in combination with any power of $x$. Only derivatives of $\rho$ include such factors. This indicates that in order to control a solution of equation \eqref{eqn:SPDE_classical_formulation}, we will have to control the $L^p$-norms of certain weighted derivatives, but only an unweighted $L^p$-norm of the solution itself. It would be possible to perform the remaining arguments solely with a base weight $a$ (without the additional weight $b$ for the derivatives), which is ``large enough'' to account for all the weights appearing, at the potential cost of some regularity/integrability of the solution.
\end{remark}

To keep the paper self-contained we recall the definition of Hilbert scales and Gelfand triples.

\begin{definition}\leavevmode
\begin{enumerate}[label=(\Roman*)]
    \item 
   A family $\left\{H_{r}, r \in \R\right\}$ of Hilbert spaces is called a Hilbert scale if
   \begin{enumerate}[label=(\roman*)]
       \item  for every $\gamma<v$, the space $H_{v}$ is dense and continuously embedded in $H_{\gamma}$;
       \item for every $\alpha<\beta<\gamma$ and $x \in H_{\gamma}$,
\begin{align*}
\|x\|_{H_{\beta}} \leq\|x\|_{H_{\alpha}}^{(\gamma-\beta) /(\gamma-\alpha)}\|x\|_{H_{\gamma}}^{(\beta-\alpha) /(\gamma-\alpha)} .
\end{align*}
   \end{enumerate}
  \item A normal (or Gelfand) triple of Hilbert spaces is an ordered collection $\left(V, H, V^{\prime}\right)$ of three Hilbert spaces with the following properties:
  \begin{enumerate}[label=(\roman*)]
      \item  $V$ is dense and continuously embedded in $H$,
      \item $H$  is dense and continuously embedded in $V^{\prime}$,
      \item The inequality
\begin{align*}
\left|\langle v, h\rangle_H\right| \leq\|v\|_V\|h\|_{V^{\prime}}
\end{align*}
holds for all $v \in V, h \in H$,
  \end{enumerate}
\end{enumerate}
\end{definition}
\begin{notation}
    Unweighted spaces ($a=b=0$) will be denoted by $W^{m,p}$, analogously for $W^{m,p}_0$.
    We will write $L^{p}(a)$ for $W^{0,p}(a,b)$ and $L^p$ for $W^{0,p}$ and follow the convention to denote the dual space of $W^{m,p}(a,b)$, identified via the triple  \begin{align*}
    &W^{m,2}(a,b)\hookrightarrow L^{2}\equiv (L^{2})^{\prime}\hookrightarrow (W^{m,2}(a,b))',
\end{align*} by $(W^{m,p}(a,b))'$. In the customary way, we denote the dual space of $W^{m,p}_{0}(a,b)$ by $W^{-m,p}(a,b)$. 
\end{notation}
With regard to the spaces, our central object of interest will be the normal triple
 \begin{align*}
    &W^{m,2}_{0}(a,b)\hookrightarrow L^{2}\equiv (L^{2})^{\prime}\hookrightarrow W^{-m,2}(a,b).
\end{align*}
To perform a more fine-grained analysis, we introduce the Hilbert-scale associated with our triple of spaces by relying on the following theorems.
\begin{theorem}\label{thm:hilbert_scale_unique}

\leavevmode
\begin{enumerate}[label=(\Roman*)]
    \item  If $\left\{H_{r}, r \in \R\right\}$ is a Hilbert scale, $\gamma \in \R$ and $m>0$, then $\left(H_{\gamma+m}, H_{\gamma}, H_{\gamma-m}\right)$ is a normal triple.
    \item If $\left(V, H, V^{\prime}\right)$ is a normal triple, then there is a unique Hilbert scale $\left\{H_{r}, r \in \R\right\}$ such that $H_{1}=V, H_{0}=H, H_{-1}=V^{\prime}$.
\end{enumerate}
\end{theorem}
\begin{theorem}\label{thm:hilbert_scale_operator_Lambda}
    Let $H_{1}$ and $H_{0}$ be two Hilbert spaces such that $H_{1}$ is densely and continuously embedded into $H_{0}$. Then there exists a uniquely determined  Hilbert scale $\left\{H_{r}, r \in \R\right\}$ and a unique family of positive-definite, self-adjoint operators $\Lambda^r: D(\Lambda^r) \to H_0$
with $D(\Lambda^r)$ a dense subset of $H_0$ if $r >0$ and otherwise a superset of $H_0$ such that
    \begin{enumerate}[label=(\Roman*)]
        \item for every $r>0$, 
        $H_{r}$ is the closure of the domain of $\Lambda^r$ in the norm $\left\|\Lambda^r \cdot\right\|_{H_{0}}$;
        \item for every $r \leq 0$, 
        $H_{r}$ is the closure of $H_{0}$ in the norm $\left\|\Lambda^r \cdot\right\|_{H_{0}}$.

   \item for $r=1$ we have $H_1=D(\Lambda^1)$.
    \end{enumerate}
\end{theorem}
The proofs of the previous statements can be found in \cite[Chapter 4.10]{krein_02_interpolation} (or \cite[Proposition 7, Lemma 8]{braukhoff2023global} for a more detailed version).

Setting $H_{1}=W^{m,2}_{0}(a,b)$, $H_{0}=L^{2}$ in the notation of Theorem \ref{thm:hilbert_scale_operator_Lambda} (or of \cite{krein_02_interpolation}), we can thus construct a unique operator $\Lambda$ characterizing a Hilbert-scale including both spaces. Indeed, setting $\Lambda$ as $\Lambda^1$ from Theorem \ref{thm:hilbert_scale_operator_Lambda}, we obtain a positive, self adjoint operator $\Lambda: W^{m,2}_{0}(a,b) \to L^2$.

Its inverse $\Lambda^{-1}\colon L^{2}\rightarrow W^{m,2}_{0}(a,b)$ is self-adjoint, positive and compact by the compact embedding of $W^{m,2}_{0}(a,b)\hookrightarrow L^{2}$ (see Lemma \ref{lem:weighted_spaces_embeddings}, assertion (3)). Hence we can find an orthonormal basis $(f_{k})_{k\in\N}$ of $L^{2}$ consisting of eigen-functions of $\Lambda$. 
From the construction in \cite[Chapter 4.10, Theorem 1.12 (p.236)]{krein_02_interpolation}, the norm on $W^{m,2}_{0}(a,b)$ is given by $\|\Lambda \cdot\|_{L^{2}}$ and one can check that $\{\Lambda^{-1}f_{n}\}_{n\in\N}$ forms an orthonormal basis (ONB) of $W^{m,2}_{0}(a,b)$.
Moreover, $\Lambda^{-1}$ can be extended such that its domain $D(\Lambda^{-1})=(W^{m,2}_{0}(a,b))^{\prime}$  and the dual space of $W^{m,2}_{0}(a,b)$ can thus be identified with $D(\Lambda^{-1})$. 

In this identification and the following, we will consider the operators $\Lambda, \Lambda^{-1}$ mapping between different spaces on our Hilbert scale. In these cases, without changing the notation of $\Lambda$, we will use the corresponding extension (via continuous linear extensions see the BLT-theorem \cite[Theorem I.7]{reed2012methods}) and isometric identifications) or restriction. 
For the following, it will be useful to recall the norm on the interpolation spaces when $r$ is not an integer. 
\begin{definition}\label{def:hilbert_scale_unweighted_duality}
Let $m \in \mathbb{N}$ be arbitrary, but fixed, and consider the operators $\{\Lambda^r, r >0\}$ associated to the Hilbert scale generated by  $H_1=W_0^{m,2}(a,b)$ and $H_0=L^2$. Then, for $r>0 $, the spaces $H_r(a(r),b)$ are defined as the closure of $D(\Lambda^r)$ with respect to the norm $\|\Lambda^{r}\cdot\|_{L^{2}}$. We use here the notation $H_r(a(r),b)$ to highlight that these spaces depend on the weights where the first weight component $a$ can additionally depend on $r$. When there is no ambiguity, we denote these spaces by $H_{r}$. This notation also applies to $r\in \N$.
   The norm on $ H_r(a(r),b)$ can be expressed by
    \begin{align}\label{eqn:fractional_norm}
         \|\Lambda^{r}u\|_{L^{2}}^{2}=\sum_{k=1}^{\infty} \lambda_{k}^{2 r}\langle u,f_{k}\rangle_{L^{2}}^{2},
    \end{align}
    with $\{f_{k}\}_{k\in\N}$ denoting the eigenfunctions of $\Lambda=\Lambda^1$ and $\{\lambda_{k}\}_{k\in\N}$ the corresponding eigenvalues.

    The dual spaces of $H_r(a(r),b)$ will be denoted by $H_{-r}(a(-r),b)$. 
\end{definition}

Theorem \ref{thm:hilbert_scale_unique} implies that for $r>0$ the normal (or Gelfand) triple
\begin{align*}
    H_{r}\hookrightarrow L^{2}\equiv (L^{2})^{\prime}\hookrightarrow H_{-r}
\end{align*}
is part of our Hilbert scale and the duality pairing is given by
\begin{align*}
    \langle \cdot,\cdot\rangle_{H_{-r}\times H_{r}}=\langle \Lambda^{-r}\cdot,\Lambda^{r}\cdot\rangle_{L^{2}}.
\end{align*}

Unless specified otherwise we shall work with this Hilbert scale where $H_1=W^{m,2}_0(a,b)$, $H_0=L^2$ and $H_{-1}=W^{-m,2}(a,b)=H_{-1}(a(-1),b)$.

We will however need a second scale that is similar to the previous one, but uses a different intermediary space.
\begin{definition}\label{def:hilbert_scale_weighted_duality}
    Let $m \in \mathbb{N}$ be arbitrary, but fixed, and consider the operators $\{\Gamma^r, r >0\}$ associated to the Hilbert scale generated by  $H_1=W_0^{m,2}(a,b)$ and $\widetilde{H}_0=L^2(a)$. Then, for $r>0 $, the spaces $W_0^{r,2}(a,b)$ are defined as the closure of $D(\Gamma^r)$ with respect to the norm $\|\Gamma^{r}\cdot\|_{L^{2}}$. For $r\in \N \cup \{0\}$, the norm corresponds to \eqref{eqn:weighted_Sobolev_norm} and in the general case $r\in \R$ the norm is given, via the spectral theorem, by
   \begin{align*}
       \|\Gamma^{r}u\|_{L^{2}(a)}=\sqrt{\int_{0}^{\infty} \gamma^{2r} \langle \d E_{\gamma} u,u\rangle},
   \end{align*}
   where $E_{\lambda}$ denotes the spectral resolution corresponding to $\Gamma$. Although we will not require them, the dual spaces on this scale will be denoted by $(W_0^{r,2}(a,b))^{\star}$.  
\end{definition}
\begin{remark}\label{rem:transferrable_mapping_properties}
     Let $\Gamma\colon  W_{0}^{m,2}(a,b)\rightarrow L^{2}(a)$ be the operator generating the Hilbert scale with $H_{1}=W_{0}^{m,2}(a,b)$ and $\widetilde{H}_{0}=L^{2}(a)$ and let $\Lambda\colon W_{0}^{m,2}(a,b)\rightarrow L^{2}$ be the operator generating the Hilbert scale with $H_{1}=W_{0}^{m,2}(a,b)$ and $H_{0}=L^{2}$. Then, $\Lambda^r=((1+x^{2})^{a}\Gamma)^r $ 
    (on the intersection of the respective domains). 
   The claim follows from the uniqueness of the Hilbert scale connection $W_{0}^{m,2}(a,b)$ with $L^2$ (see \cite[IV, Theorem 1.12]{krein_02_interpolation}). 

   This fact allows us to transfer mapping properties of operators between the two scales associated to $\Gamma$ (Definition \ref{def:hilbert_scale_weighted_duality}) and $\Lambda$ (Definition \ref{def:hilbert_scale_unweighted_duality}). Indeed,  \cite[IV, Theorem 1.12]{krein_02_interpolation} shows that for $\mathbb{N} \ni n\leq m$, $W_{0}^{n,2}(a,b)\subseteq H_{\frac{n}{m}}$. Considering the normal triples
$(W^{2m,2}_{0}(a,b),W^{m,2}_{0}(a,b), L^{2}(a))$ and $(H_{\frac{2m}{m}},W^{m,2}_{0}(a,b),L^{2})$, it can be seen that the reverse inclusions hold for the spaces where $n\geq m$, i.e., $W^{n,2}_{0}(a,b) \supseteq H_{\frac{n}{m}}$. 
\end{remark}

The main difference between the two scales we introduced is that, for $m>\frac{1}{2}$, the embedding $W_0^{m,2}(a,b)\hookrightarrow L^2$ is compact, whereas the embedding $W_0^{m,2}(a,b)\hookrightarrow L^2(a)$ is only continuous. Although the scale generated by $\Gamma$ has a more convenient structure when analyzing the mapping properties of certain operators appearing in our equation, the compactness of the previously mentioned embedding provides a sufficiently regular orthonormal basis of $L^{2}$ which enables us to derive the desired estimates.
\begin{remark}
Since our solution theory did not involve these particular spaces, we want to convince ourselves, that any solution for \eqref{eqn:SPDE_classical_formulation} we obtained from Proposition \ref{prop:convergence_to_weak_solution_SPDE} can be embedded into one of those spaces. We recall  (see also Definition \ref{def:weak_solution}) that our solution $\rho$, as well as initial condition $\rho_{0}$ are (for almost every $\omega$) elements of $M_{1}(\R_{+})$, equipped with the Wasserstein-1 distance.  
By Morrey's inequality, the space of signed measures on $\R$ and hence also $M_{1}(\mathbb{R}_+)$ can be embedded (see e.g. \cite[Section 2.1]{fernandez_97_hilbertian}) in the (fractional) unweighted Sobolev Slobodeskji space $W^{-1/2-\eps,2}$ (see \cite{di2012hitchhikers} for a definition).
This allows us to consider $\rho$ as an element of an unweighted Hilbert space.
For $m \geq 1/2 +\eps$, $\eps >0$, $a\geq 0$, $b\geq 0$, we have by Lemma~\ref{lem:weighted_spaces_embeddings} the continuity of the 
 following chain of embeddings
 \begin{align*}
    &H_{1}\hookrightarrow W_{0}^{1/2+\eps,2}\hookrightarrow L^{2}\equiv (L^{2})^{\prime}\hookrightarrow W^{-1/2-\eps,2}\hookrightarrow H_{-1}.
\end{align*}
It should be noted that the weighted and unweighted spaces are technically part of yet another Hilbert-scale, the one determined by the unweighted triple $W_{0}^{1/2+\eps,2}\hookrightarrow L^{2}\equiv (L^{2})^{\prime}\hookrightarrow W^{-1/2-\eps,2}$. In this argument, we are however only interested in the embeddings, and since all the dualities are understood via the extended $L^{2}$-inner duality we can conclude the desired embeddings of the dual spaces.
\end{remark}

In the following, we will assume $m$ to be ``sufficiently large''. Quantitatively this means $m\geq 3$.

The next lemma will establish a more explicit relation between the two scales we introduced.

The following lemma provides the essential estimate for the uniqueness of solutions to equation \eqref{eqn:SPDE_classical_formulation}. We will state it for a larger class of solutions than we would require based on Proposition \ref{prop:convergence_to_weak_solution_SPDE}. For example, we will only assume $L^{2}([0,T])$-integrability in time, instead of continuity on the interval $[0,T]$. We also allow for a potentially random, measure-valued, initial condition.

\begin{lemma}\label{prop:uniqueness_SPDE_weighted_space}
Let $\rho$ denote a solution of \eqref{eqn:SPDE_classical_formulation} in the sense of Definition \ref{def:weak_solution} with coefficients $\alpha, \beta\geq 0$ (applied to the slightly more general setting here).
Assume $\rho_{0}\in L^{2}(\Omega,M_{1}(\R_{+}))$ and let $m\geq 3$, $a>0, b \geq \frac{1}{2}$. Then the following statements hold.
\begin{enumerate}[label=(\Roman*)]
    \item $\rho$ satisfies the estimate
    \begin{align}\label{eqn:estimate_weighted_dual}
      \E \| \rho_{t}\|_{H_{-1}}^{2}&\leq 2\E\| \rho_{0}\|_{H_{-1}}^{2}+\E \int_{0}^{t} C_{m,\alpha,\beta,a,b,\sigma}(1+\langle\rho_{s},\idx\rangle)\|\rho_{s}\|_{H_{-1}}^{2}\ds,
\end{align}
where $C_{m,\alpha,\beta,a,b,\sigma}$ denotes a constant depending on $m, a, b$, the model parameters $\alpha,\beta$ and $\sigma$ which will be specified in the proof.
\item 
The solution of \eqref{eqn:SPDE_weak_formulation} is pathwise unique.
\end{enumerate}
\end{lemma}
\begin{proof}
We recall that by Lemma \ref{lem:rho_x}, we know that
\begin{align*}
    \langle \rho_{t},\idx\rangle = m_{\lambda}\exp{((\beta-1/2)t+W^{0}_{t})}.
\end{align*}
Hence, instead of dealing with non-linear operators of the form 
\begin{align*}
    \rho\mapsto \partial_{x}\left(\left(\langle\rho,\idx\rangle\frac{\alpha}{2}x+\frac{x^{2}}{2}\right)\partial_{x}\rho\right),
\end{align*}
it suffices to consider linear operators of the type
\begin{align}\label{eqn:linearized_operator}
    \rho\mapsto \partial_{x}\left(\left(m_{\lambda}\exp{((\beta-1/2)t+W^{0}_{t})}\frac{\alpha}{2}x+\frac{x^{2}}{2}\right)\partial_{x}\rho\right).
\end{align}
This will eliminate any ambiguity regarding the definition of the adjoints of the appearing operators.
Let $\Lambda:  W^{m,2}_0(a,b) \to L^2$ be the operator introduced in Definition \ref{def:hilbert_scale_unweighted_duality} and $(f_{k})_{k\in\N}$ the associated orthonormal basis of $L^{2}$. Consider $\langle\rho,e_{j}\rangle^{2}$, where $e_{j}=\Lambda^{-1}f_{j}\in W^{m,2}_0(a,b)$ is used as a test function in the original equation.   
For $m\geq 3$, the image of $e_{j}$ under an operator of the form \eqref{eqn:linearized_operator} remains continuous and bounded by Lemma \ref{lem:weighted_spaces_embeddings}. This allows us to pair this image with $\rho_{s}$ and guarantee that all terms in equation \eqref{eqn:SPDE_weak_formulation} are well defined.

 We highlight that the pairing with only one space in the subscript will be understood as the inner product of the respective space. The pairing without any subscript remains to be understood as the pairing of measures in $M_{1}$ and continuous functions of at most linear growth.

 For convenience, we use $x\partial_{x}^{2}\varphi=\partial_{x}\left(x\partial_{x}\varphi\right)-\partial_{x}\varphi$ and $x^{2}\partial_{x}^{2}\varphi=\partial_{x}\left(x^{2}\partial_{x}\varphi\right)-2x\partial_{x}\varphi$ to rewrite our operators in divergence-form and apply It{\^o}'s formula (see e.g.~\cite{Krylov_13_Ito} in this context) to the squared dual pairing between $W^{m,2}_{0}(a,b)=H_{1}$ and $H_{-1}$,  
 which we denote by $\langle \cdot , \cdot \rangle_{H_{-1}\times H_{1}}=\langle \Lambda^{-1}\cdot,\Lambda \cdot\rangle_{L^{2}}$. Then,
\begin{align*}
    \langle \rho_{t},e_{j}\rangle_{H_{-1}\times H_{1}}^{2}&=\langle \rho_{0},e_{j}\rangle_{H_{-1}\times H_{1}}^{2}+2\int_{0}^{t} \langle \rho_{s},e_{j}\rangle_{H_{-1}\times H_{1}}\left\langle  \rho_{s},\partial_{x}\left(\left(\langle\rho_{s},\idx\rangle\frac{\alpha}{2}\idxo+\frac{\idxo^{2}}{2}\right)\partial_{x} e_{j}\right) \right\rangle_{H_{-1}\times H_{1}}\ds\\
 &\phantom{xx}{}-2\int_{0}^{t} \langle \rho_{s},e_{j}\rangle_{H_{-1}\times H_{1}}\left\langle \rho_{s},\left(\langle\rho_{s},\idx\rangle\frac{\alpha}{2}-\beta\langle\rho_{s},\idx\rangle+\idxo\right)\partial_{x} e_{j} \right\rangle_{H_{-1}\times H_{1}}\ds\\
 &\phantom{xx}{}+2\int_{0}^{t}\langle \rho_{s},e_{j}\rangle_{H_{-1}\times H_{1}}\left\langle \rho_{s},e_{j}-\partial_{x}\left(\idxo e_{j}\right)\right\rangle_{H_{-1}\times H_{1}}\dW^{0}_{s}\\
 &\phantom{xx}{}+\int_{0}^{t} \left\langle \rho_{s},e_{j}-\partial_{x}\left(\idxo e_{j}\right)\right\rangle_{H_{-1}\times H_{1}}^{2}\ds. 
\end{align*}
Since $\langle \Lambda^{-1}\cdot,\Lambda e_{j}\rangle^{2}_{L^{2}}=\langle \Lambda^{-1}\cdot,f_{j}\rangle^{2}_{L^{2}}$, summing over $j$ yields $\|\Lambda^{-1}\cdot\|^{2}_{L^{2}}=\|\cdot\|_{H_{-1}}^{2}$ and we obtain

\begin{align}
\label{eqn:ito_weighted_space_norm}
    \| \rho_{t}\|_{H_{-1}}^{2}&= \| \rho_{0}\|_{H_{-1}}^{2}+2\int_{0}^{t} \left\langle  \rho_{s},\left(\partial_{x}\left(\left(\langle\rho_{s},\idx\rangle\frac{\alpha}{2}\idxo+\frac{\idxo^{2}}{2}\right)\partial_{x} (\cdot)\right)\right)^{*}\rho_{s} \right\rangle_{H_{-1}}\ds\nonumber\\
    &\phantom{xx}{}-2\int_{0}^{t} \left\langle  \rho_{s},\left(\left(\langle\rho_{s},\idx\rangle\frac{\alpha}{2}-\beta\langle\rho_{s},\idx\rangle+\idxo\right)\partial_{x} (\cdot)\right)^{*}\rho_{s} \right\rangle_{H_{-1}}\ds\nonumber\\
    &\phantom{xx}{}+2\int_{0}^{t}\left\langle \rho_{s},\left((\cdot)-\partial_{x}\left(\idxo (\cdot)\right)\right)^{*}\rho_{s}\right\rangle_{H_{-1}}\dW^{0}_{s}\nonumber\\
    &\phantom{xx}{}+\int_{0}^{t}\left\langle \left((\cdot)-\partial_{x}\left(\idxo (\cdot)\right)\right)^{*}\rho_{s},\left((\cdot)-\partial_{x}\left(\idxo (\cdot)\right)\right)^{*}\rho_{s}\right\rangle_{H_{-1}}^{2}\ds\nonumber\\
    &=\| \rho_{0}\|_{H_{-1}}^{2}+2\int_{0}^{t} \left\langle  \rho_{s},L^{*}\rho_{s} \right\rangle_{H_{-1}}\ds+\int_{0}^{t} \left\| \sigma^{*}\rho_{s} \right\|^{2}_{H_{-1}}\ds + 2\int_{0}^{t}\left\langle \rho_{s},\sigma^{*}\rho_{s}\right\rangle_{H_{-1}}\dW^{0}_{s},
\end{align}
with 
 \begin{align*}
     L u &\coloneqq
     \partial_{x}\left(\left(\frac{\alpha}{2}\langle \rho_{t},\idx\rangle \idxo+\frac{1}{2}\idxo^{2}\right)\partial_{x}u\right)-\left(\frac{\alpha}{2}\langle \rho_{t},\idx\rangle- \beta\langle \rho_{t},\idx\rangle+\idxo\right)\partial_{x} u,
 \end{align*}
 \begin{align*}
     \sigma(u)\coloneqq u -\partial_{x}(x u)=x\partial_{x}u.
 \end{align*}
 The $L^{2}$ adjoints of these operators are given by
 \begin{align*}
     L^{*} u &\coloneqq \partial_{x}\left(\left(\frac{\alpha}{2}\langle \rho_{t},\idx\rangle \idxo+\frac{1}{2}\idxo^{2}\right)\partial_{x}u\right)+\partial_{x}\left(\left(\frac{\alpha}{2}\langle \rho_{t},\idx\rangle- \beta\langle \rho_{t},\idx\rangle+\idxo\right)u\right),
 \end{align*}
 \begin{align*}
     \sigma^{*}(u)\coloneqq \idxo \partial_{x}u+u= \partial_{x}(\idxo u).
 \end{align*}

Now the goal is to show that 
\begin{align}\label{eq:rhonorm}
\left\langle  \rho_{s},L^{*}\rho_{s} \right\rangle_{H_{-1}}&=\left\langle  L (\Lambda^{-1})^{*}(\Lambda^{-1})(\rho_{s}), (\Lambda^{-1})^{*}(\Lambda^{-1})(\rho_{s}) \right\rangle_{H_{1}}.
    \end{align}

For this, note first that similarly to \cite[Proof of Theorem 2.5]{gyongy1992stochastic_unbounded} it can be shown that $L\colon W_{0}^{m+1,2}(a,b)\rightarrow W_{0}^{m-1,2}(a,b)$,  
which is however not the Hilbert scale we work with, but the one introduced in Definition \ref{def:hilbert_scale_weighted_duality}. 
Therefore we quickly verify that $L$ has similar mapping properties between the $H_{r}$ spaces belonging to the scale introduced in Definition \ref{def:hilbert_scale_unweighted_duality}, on which we perform our analysis. Let $m+2 \geq n\geq m$, then $W_{0}^{n,2}(a,b)\supseteq H_{\frac{n}{m}}$ and $W_{0}^{n-2,2}(a,b)\subset H_{\frac{n-2}{m}}$ (see Remark \ref{rem:transferrable_mapping_properties}). Hence, $L\colon H_{\frac{n}{m}}\rightarrow H_{\frac{n-2}{m}}$, showing that one also loses $2$ degrees of regularity.

As a first step, consider functions $\varphi\in C_{c}^{\infty}(\R_{+})$. 
Since $\Lambda$ and $\Lambda^{-1}$ are self-adjoint as operators on $L^{2}$,
\begin{equation}\label{eq:equality}
\begin{split}
    \langle\varphi, L^{*}\varphi\rangle_{H_{-1}}&=\langle \Lambda^{-1}(\varphi),\Lambda^{-1} L^{*}\varphi\rangle_{L^{2}}=\langle (\Lambda^{-1})^{*}(\Lambda^{-1})(\varphi),L^{*}\varphi\rangle_{L^{2}}=\langle L(\Lambda^{-1})^{*}(\Lambda^{-1})(\varphi),\varphi\rangle_{L^{2}}\\
    &=\langle \Lambda  L(\Lambda^{-1})^{*}(\Lambda^{-1})(\varphi),\Lambda^{-1}\varphi\rangle_{L^{2}}\\
    &=\langle \Lambda  L(\Lambda^{-1})^{*}(\Lambda^{-1})(\varphi),\Lambda \Lambda^{-1}\Lambda^{-1}\varphi\rangle_{L^{2}}=\langle \Lambda  L(\Lambda^{-1})^{*}(\Lambda^{-1})(\varphi),\Lambda (\Lambda^{-1})^{*}\Lambda^{-1}\varphi\rangle_{L^{2}}\\
    &=\langle  L(\Lambda^{-1})^{*}(\Lambda^{-1})(\varphi),(\Lambda^{-1})^{*}\Lambda^{-1}\varphi\rangle_{H_{1}}.
\end{split}
\end{equation}

In the following we show  that $ \langle\varphi, L^{*}\varphi\rangle_{H_{-1}}=\langle L(\Lambda^{-1})^{*}(\Lambda^{-1})(\varphi), (\Lambda^{-1})^{*}(\Lambda^{-1})(\varphi)\rangle_{H_{1}}$ remains also valid for less regular functions $\varphi$, as long as we can define $\langle L(\Lambda^{-1})^{*}(\Lambda^{-1})(\varphi), (\Lambda^{-1})^{*}(\Lambda^{-1})(\varphi)\rangle_{H_{1}}$.

The next step is to identify, for which $\varphi$, we can define this bi-linear form. For this, we first note that the operator $(\Lambda^{-1})^{*}(\Lambda^{-1})$ can also be identified with a mapping from $H_{r}\rightarrow H_{r+2}$ for each $r\in \mathbb{R}$. Indeed, for $u\in H_{r}$ , we have $\|(\Lambda^{-1})^{*}(\Lambda^{-1})u\|_{H_{r+2}}=\|\Lambda^{r+2}(\Lambda^{-1})^{*}(\Lambda^{-1})u\|_{L^{2}}=\|\Lambda^{r}u\|_{L^{2}}=\|u\|_{H_r}$, which implies in  particular that $\|(\Lambda^{-1})^{*}(\Lambda^{-1})u\|_{H_{r+2}} <  \infty$. 
For  $\varphi\in H_{-1+\frac{1}{m}}$ we thus get
\begin{equation}
\begin{split}\label{eq:estimatesH}
    \langle  &L(\Lambda^{-1})^{*}(\Lambda^{-1})(\varphi),(\Lambda^{-1})^{*}\Lambda^{-1}\varphi\rangle_{H_{1}}\\
    &=\langle \Lambda  L(\Lambda^{-1})^{*}(\Lambda^{-1})(\varphi),\Lambda (\Lambda^{-1})^{*}(\Lambda^{-1})(\varphi)\rangle_{L^{2}}=\langle \Lambda^{1-\frac{1}{m}}  L(\Lambda^{-1})^{*}(\Lambda^{-1})(\varphi),\Lambda^{1+\frac{1}{m}}(\Lambda^{-1})^{*}(\Lambda^{-1})(\varphi)\rangle_{L^{2}}\\
    &\leq \|\Lambda^{1-\frac{1}{m}}\|_{\mathcal{L}(H_{1-\frac{1}{m}},L^{2})}\| L\|_{\mathcal{L}(H_{1+\frac{1}{m}},H_{1-\frac{1}{m}})}\|(\Lambda^{-1})^{*}(\Lambda^{-1})\|_{\mathcal{L}(H_{-1+\frac{1}{m}},H_{1+\frac{1}{m}})}\|\varphi\|_{H_{-1+\frac{1}{m}}}\underbrace{\|\Lambda^{\frac{1}{m}} \Lambda (\Lambda^{-1})^{*}(\Lambda^{-1})(\varphi)\|_{L^{2}}}_{\| \varphi\|_{H_{-1+\frac{1}{m}}}}\\
       &\leq C_{L} \|\varphi\|_{H_{-1+\frac{1}{m}}}^{2},
       \end{split}
\end{equation}
\phantom{xxxxxx} where we used the mapping property of $(\Lambda^{-1})^{*}(\Lambda^{-1})$ from $H_{-1 +\frac{1}{m}}$ to $H_{1 +\frac{1}{m}}$ to conclude that
$\|(\Lambda^{-1})^{*}(\Lambda^{-1})\|_{L\left(H_{-1+\frac{1}{m}},H_{1+\frac{1}{m}}\right)}$ is finite. 
This shows that $\langle L(\Lambda^{-1})^{*}(\Lambda^{-1})(\varphi),(\Lambda^{-1})^{*}\Lambda^{-1}\varphi\rangle_{H_{1}}$ is well defined for  $\varphi\in H_{-1+\frac{1}{m}}$, whence by density \eqref{eq:equality} also holds for $\varphi \in H_{-1+\frac{1}{m}} $.

What remains to show is that, for every $s\in [0,T]$, $\rho_{s}\in H_{-1+\frac{1}{m}}$. Since $\rho_{s}$ is an element of $M_{1}(\R_{+})\subseteq W^{-\frac{1}{2}-\eps}$, it is also an element of $ H_{-\frac{1}{m}}=H_{-1+\frac{m-1}{m}}\subseteq H_{-1+\frac{1}{m}}$ as long as $m\geq 2$, which is satisfied in our setting. Together with the previous arguments, this yields \eqref{eq:rhonorm} for every $s\in [0,T]$.

For $\sigma$, we use a different approach:
First, observe that $\|\sigma^{*}\rho_{s}\|_{H_{-1}}<\infty$ since $\rho_{s}\in H_{-1+\frac{1}{m}}$. Indeed, $\sigma^{*}\colon H_{-1+\frac{1}{m}}\rightarrow H_{-1}$, since
\begin{align*}
     \|\sigma^{*}\rho\|_{H_{-1}}&=\sup_{\varphi\in H_{1}\colon \|\varphi\|_{H_{1}}=1}|\langle \Lambda^{-1}\sigma^{*}\rho,\Lambda\varphi\rangle_{L^{2}}|=\sup_{\varphi\in H_{1}\colon \|\varphi\|_{H_{1}}=1}|\langle \Lambda ^{-1+\frac{1}{m}}\rho,\Lambda^{-\frac{1}{m}}\Lambda\sigma\varphi\rangle_{L^{2}}|\\
     &\leq \sup_{\varphi\in H_{1}\colon \|\varphi\|_{H_{1}}=1}\| \Lambda ^{-1+\frac{1}{m}}\rho\|_{L^{2}}\|\Lambda^{-\frac{1}{m}}\Lambda\sigma\varphi\|_{L^{2}}\leq C_{\sigma}\sup_{\varphi\in H_{1}\colon \|\varphi\|_{H_{1}}=1}\|\varphi\|_{H_{1}}\|\rho\|_{H_{-1+\frac{1}{m}}},
\end{align*}
where we used the dual norm and the last inequality follows from similar estimates as in \eqref{eq:estimatesH}.

In the following, we use again the dual norm to express  $\|\sigma^{*}\rho\|_{H_{-1}}$ as follows:
\begin{align*}
     \|\sigma^{*}\rho\|_{H_{-1}}=\sup_{\varphi\in H_{1}\colon \|\varphi\|_{H_{1}}=1}|\langle \Lambda(\Lambda ^{-1})^{*}\Lambda^{-1}\rho,\Lambda\sigma\varphi\rangle_{L^{2}}|=\sup_{\varphi\in H_{1}\colon \|\varphi\|_{H_{1}}=1}|\langle (\Lambda ^{-1})^{*}\Lambda^{-1}\rho,\sigma\varphi\rangle_{H_{1}}|.
\end{align*}

Now we are in the position to use Lemma \ref{lem:coercivity_L_sigma} (with $u=(\Lambda^{-1})^{*}(\Lambda^{-1})(\rho_{s})$ and recalling Lemma \ref{lem:rho_x}) and bound the right-hand side of \eqref{eqn:ito_weighted_space_norm} further.
\begin{align*}
    &2\left\langle  \rho_{s},L^{*}\rho_{s} \right\rangle_{H_{-1}}+\|\sigma^{*}\rho_{s}\|_{H_{-1}}^{2}\\
        &= 2\left\langle  L (\Lambda^{-1})^{*}(\Lambda^{-1})(\rho_{s}), (\Lambda^{-1})^{*}(\Lambda^{-1})(\rho_{s}) \right\rangle_{H_{1}}
        +\sup_{\varphi\in H_{1}\colon \|\varphi\|_{H_{1}}=1}|\langle (\Lambda ^{-1})^{*}\Lambda^{-1}\rho,\sigma\varphi\rangle_{H_{1}}|^{2}\\
    &\leq C_{m,\alpha,\beta,a,b}(1+\langle\rho_{s},\idx\rangle)\|(\Lambda^{-1})^{*}(\Lambda^{-1})(\rho_{s})\|_{H_{1}}^{2}\leq C_{m,\alpha,\beta,a,b}(1+\langle\rho_{s},\idx\rangle)\|\rho_{s}\|_{H_{-1}}^{2}.
\end{align*}
From this, we obtain
\begin{align}\label{eqn:rho_estimate_before_expectation}
        \| \rho_{t}\|_{H_{-1}}^{2}&\leq \| \rho_{0}\|_{H_{-1}}^{2}+C_{m,\alpha,\beta,a,b}\int_{0}^{t} (1+\langle\rho_{s},\idx\rangle)\|\rho_{s}\|_{H_{-1}}^{2}\ds+ 2\int_{0}^{t}\left\langle \rho_{s},\sigma^{*}\rho_{s}\right\rangle_{H_{-1}}\dW^{0}_{s}.
\end{align}
Since $\int_{0}^{t}\left\langle \rho_{s},\sigma^{*}\rho_{s}\right\rangle_{H_{-1}}\dW^{0}_{s}$ is a continuous local martingale, we can use a standard localization argument, take expectation and obtain \eqref{eqn:estimate_weighted_dual}.

If we consider two solutions $\rho^{1}, \rho^{2}$, starting at the same initial condition, it is clear that the difference $\rho_{1}-\rho_{2} \in H_{-1}$ solves equation \eqref{eqn:SPDE_random} with initial condition 0. By the continuity of sample paths in time, we can introduce the sequence of stopping times $\tau_{K}\coloneqq \inf\{ t \colon |W^{0}_{t}|>K\}$ and perform the previous steps with $\rho$ replaced by $\rho^{1}-\rho^{2}$, which leads to
\begin{align*}
       \E \| \rho^{1}_{t\wedge \tau_{K}}-\rho^{2}_{t\wedge \tau_{K}}\|_{H_{-1}}^{2}&\leq C_{m,\alpha,\beta,a,b,\sigma}\E \int_{0}^{t\wedge \tau_{K}} (1+\langle\rho_{s},\idx\rangle)\|\rho^{1}_{s}-\rho^{2}_{s}\|_{H_{-1}}^{2}\ds\\
           &\leq C_{m,\alpha,\beta,a,b,\sigma,K}\int_{0}^{t} \E\|\rho^{1}_{s}-\rho^{2}_{s}\|_{H_{-1}}^{2}\chi_{s\in [0,t\wedge \tau_{K}]}\ds.
\end{align*}
Gronwall's inequality now leads to
\begin{align*}
     \E \| \rho^{1}_{t}-\rho^{2}_{t}\|_{H_{-1}}^{2}=0.
\end{align*}
Letting $K\rightarrow \infty$ implies that the solution is pathwise-unique on [0,T]. This yields the second statement of the Lemma.
\end{proof}

The pathwise uniqueness directly leads to the conclusion that we have indeed a probabilistically strong solution. In other words, $\rho$ is defined on the initial stochastic basis.
\begin{proposition}\label{prop:Yamada-Watanabe_type_result}
 The sequences $(\rho^{N}, W^{0})$ converge (up to a sub-sequence) almost surely relative to the initially chosen stochastic basis.
\end{proposition}
The proof is a straightforward adaptation of the arguments from \cite{DGHT11_Local_martingale_solutions_abstract_fluids}.
\begin{remark}
    If one wanted to study the existence of solutions to \eqref{eqn:SPDE_classical_formulation} isolated from the particle system consideration, one could add an additional regularizing term of the form $+\eps \Delta \rho$ on the right-hand side of \eqref{eqn:SPDE_classical_formulation}, perform standard arguments (see e.g. \cite[Chapter 5]{Roeckner_15_SPDE}) to obtain a solution and pass to the limit $\eps\rightarrow 0$. The estimates to verify the conditions in \cite[Chapter 5]{Roeckner_15_SPDE} and pass to the limit are almost identical to the ones performed in the uniqueness argument. 
\end{remark}

\section{Connection with the McKean-Vlasov SDE / Proof of Theorem \ref{thm:collected_results_McKean-Vlasov_SDE_introduction}}\label{sec:connection_McKean-Vlasov_SDE}
\subsection{Stochastic representation as conditional law/ Theorem \ref{thm:collected_results_McKean-Vlasov_SDE_introduction} \ref{prop:P_1:stochastic_representation_rho_introduction} and \ref{prop:P_2:McKean-Vlasov_SDE_and_mean_jointly_polynomial_introduction}}\label{sub_sec:stochastic_representation}
In this section, we want to establish the connection between $\rho$, being the solution of \eqref{eqn:SPDE_classical_formulation}, and the conditional law of the solution to the McKean-Vlasov SDE given in \eqref{eqn:MKV_SDE_introduction}, i.e.,
\begin{align*}
    \d Y(t)=\beta \E[Y(t)|\mathcal{W}^{0}]\dt+\sqrt{\alpha}\sqrt{Y(t)\E[Y(t)|\mathcal{W}^{0}]}\dB_{t}+Y(t)\dW^{0}_{t},
\end{align*}
with independent Brownian motions $B$ and $W^{0}$. Recall that $\mathcal{W}^{0}$ denotes the filtration generated by the increments of $W^{0}$. In the sequel $\Law(Y(t)|\mathcal{W}^{0})$ will denote the family $(\Law(Y(t)|\mathcal{W}^{0}))_{t\geq 0}$
of conditional  probability measures, such that for every $t\in [0,T]$ and $\varphi\in C_b(\R_{+})$,
\begin{align}
    \langle \Law(Y(t)|\mathcal{W}^{0}), \varphi \rangle\coloneqq \E [\varphi(Y(t))|\mathcal{W}^{0}]=\E [\varphi(Y(t))|\mathcal{W}^{0}_{t}]\quad\Prob \mbox{ -a.s.} \label{eqn:tested_conditional_law}
\end{align}
The next proposition establishes the existence and uniqueness of \eqref{eqn:MKV_SDE_introduction}, as well as the property that the solution $Y$, together with its conditional expected value, is a polynomial process.

\begin{proposition} \label{prop:McKeanE&U}
Let $Y(0)\in L^{2}(\Omega)$ be independent of $\mathcal{W}^0$ and 
$Y(0)\geq 0$ a.s.  Then the following assertions hold true.
\begin{enumerate}[label=(\Roman*)]
    \item \label{prop:MKV_existence_representation_existence_of_solution}
    The McKean-Vlasov SDE \eqref{eqn:MKV_SDE_introduction}
has a non-negative weak solution if $\alpha, \beta \geq 0$ and a strictly positive weak solution if $Y(0)>0$ a.s. and $\alpha, \beta$ satisfy \ref{A:A1_alpha_beta}.

    \item \label{prop:MKV_existence_representation_stoachstic_representation_connection_to_SPDE} Let 
    $Y(0)>0$ a.s., $\alpha, \beta$ satisfy \ref{A:A1_alpha_beta} and  $\rho$ be the unique solution of \eqref{eqn:SPDE_classical_formulation} with initial condition $\rho_{0}=\Law(Y(0))$. Then any solution to \eqref{eqn:MKV_SDE_introduction} satisfies $\rho=\Law(Y(\cdot)|\mathcal{W}^{0})$ as well as
\begin{align*}
    \E[Y(t)|\mathcal{W}^{0}]=\langle\rho_{t},\idx\rangle.
\end{align*}
\item \label{prop:polynomialprop}
The two-dimensional process $(Y,\langle\rho,\idx\rangle)=(Y, \E[Y|\mathcal{W}^{0}])$ is a polynomial diffusion on $\mathbb{R}_+^2$ which is unique in law. In particular, the weak solution of \eqref{eqn:MKV_SDE_introduction} is unique in law.
\end{enumerate}
\end{proposition}
\begin{proof}
We first prove \ref{prop:MKV_existence_representation_existence_of_solution}.
Taking conditional expectations, yields
\begin{align*}
      \d \E[Y(t)|\mathcal{W}^{0}]=\beta \E[Y(t)|\mathcal{W}^{0}]\dt+\E[Y(t)|\mathcal{W}^{0}]\dW^{0}_{t}.
\end{align*}
Hence, $\E[Y(t)|\mathcal{W}^{0}]=\langle\rho_{t},\idx\rangle= m_{\lambda}\exp{\bigg(\bigg(\beta-\frac{1}{2}\bigg)t+W^{0}_{t}\bigg)}$, by Lemma \ref{lem:convergence_of_mean_rho_N}, where $\rho$ denotes the unique solution of \eqref{eqn:SPDE_classical_formulation} with initial condition $\rho_{0}=\Law(Y_{0})$.
 Hence \eqref{eqn:MKV_SDE_introduction} can be seen as a factor of a two-dimensional system of equations, given by
    \begin{align}
        &\d Y(t)=\beta S(t)\dt+\sqrt{\alpha}\sqrt{S(t)}\sqrt{Y(t)}\dB_{t}+Y(t)\dW^{0}_{t}\label{eqn:MKV_proof_MKV_equation_rewritten_as_system}
\\&\d S(t)=\beta S(t)\dt+S(t)\dW^{0}_{t}. \label{eqn:condexp}   
    \end{align}
If $\alpha, \beta \geq 0$, then the existence of a weak solution of this system with values in $\mathbb{R}_+^2$ is a consequence of 
\cite[Theorem 5.3, Proposition 6.4]{filipovic_larsson_16_polynomial_diffusion}.
The strict positivity, under Assumption \ref{A:A1_alpha_beta}, follows from \cite[Theorem 5.7]{filipovic_larsson_16_polynomial_diffusion}.

To prove \ref{prop:MKV_existence_representation_stoachstic_representation_connection_to_SPDE},
note that by It{\^o}'s formula,
\begin{align*}
  \varphi(Y(t))&=\varphi(Y(0))+\int_{0}^{t}\partial_{x}\varphi (Y(r))\beta\langle\rho_{r},\idx\rangle\dr\\
    &\phantom{xx}{}+\frac{1}{2}\int_{0}^{t}\partial_{x}^{2}\varphi (Y(r))\alpha Y(r)\langle\rho_{r},\idx\rangle+\partial_{x}^{2}\varphi (Y(r))Y^{2}(r)\dr\\
    &\phantom{xx}{}+\int_{0}^{t}\partial_{x}\varphi (Y(r))\sqrt{\alpha}\sqrt{Y(r)}\sqrt{\langle\rho_{r},\idx\rangle}\dB_{r}+\int_{0}^{t}\partial_{x}\varphi (Y(r))Y(r)\dW^{i}_{r}.
\end{align*}
Taking conditional expectation with respect to $\mathcal{W}^{0}_{t}$ yields
\begin{align*}
  \E[\varphi(Y(t))|\mathcal{W}^{0}_{t}]&=
\E[\varphi(Y(0))|\mathcal{W}^{0}_{t}]+\E\bigg[\int_{0}^{t}\partial_{x}\varphi (Y(r))\beta\langle\rho_{r},\idx\rangle\dr\bigg|\mathcal{W}^{0}_{t}\bigg]\\
    &\phantom{xx}+\frac{1}{2}\E\bigg[\int_{0}^{t}\partial_{x}^{2}\varphi (Y(r))Y(r)\langle\rho_{r},\idx\rangle+\partial_{x}^{2}\varphi (Y(r))Y^{2}(r)\dr\bigg|\mathcal{W}^{0}_{t}\bigg]\\
    &\phantom{xx}{}+\int_{0}^{t}\E\big[\partial_{x}\varphi (Y(r))Y(r)\big|\mathcal{W}^{0}_{r}\big]\dW^{0}_{r}.
\end{align*}
where we used the independence of $B$ and $W^0$ and a Fubini-type argument. By \eqref{eqn:tested_conditional_law},
\begin{align*}
  \langle \Law(Y(t)|\mathcal{W}^{0}_{t}),\varphi\rangle&=\langle \Law(Y(0)|\mathcal{W}^{0}_{0}),\varphi\rangle+\E\int_{0}^{t}\langle \Law(Y(r)|\mathcal{W}^{0}_{r}),\partial_{x}\varphi (\cdot)\langle\rho_{r},\idx\rangle\rangle\dr\\
    &\phantom{xx}+\frac{1}{2}\int_{0}^{t}\alpha\langle \Law(Y(r)|\mathcal{W}^{0}_{r}),(\cdot)\partial_{x}^{2}\varphi (\cdot)\langle\rho_{r},\idx\rangle\rangle +\langle \Law(Y(r)|\mathcal{W}^{0}_{r}),(\cdot)^{2}\partial_{x}^{2}\varphi (\cdot)\rangle\dr\\
    &\phantom{xx}{}+\int_{0}^{t}\langle \Law(Y(r)|\mathcal{W}^{0}_{r}),(\cdot)\partial_{x}\varphi (\cdot)\rangle\dW^{0}_{r}.
\end{align*}
and we see that for $\varphi\in \mathcal{S}$, this coincides with equation \eqref{eqn:SPDE_classical_formulation}. Hence by the uniqueness of  \eqref{eqn:SPDE_classical_formulation}, its solution coincides with $\Law(Y(t)|\mathcal{W}^{0}_{t})$ and we conclude that $\rho_{t}=\Law(Y(t)|\mathcal{W}^{0}_{t})$.

The polynomial property as stated in \ref{prop:polynomialprop}, follows simply from the fact that $(Y(t), \langle \rho_t, \idx \rangle)$ satisfies \eqref{eqn:MKV_proof_MKV_equation_rewritten_as_system}
-\eqref{eqn:condexp}. Indeed, the drift characteristic being linear and the diffusion characteristic quadratic in the state variables yields the assertion (see \cite{cuchiero2012polynomial,
filipovic_larsson_16_polynomial_diffusion}).
To show the uniqueness of the weak solution of \eqref{eqn:MKV_proof_MKV_equation_rewritten_as_system}-\eqref{eqn:condexp}, note first that \eqref{eqn:condexp} admits a unique strong solution adapted to $\mathcal{W}^0$.
To verify the uniqueness in law of the two-dimensional system,
we use \cite[Theorem 4.4.2 a), p.184]{ethier_kurtz_09_markov}. Hence we have to verify
that all one-dimensional time-marginals of two solutions $(Y,S)$ and $(\widetilde{Y},S)$  coincide. This amounts to check that $\E[\varphi(Y(t),S(t)))]-\widetilde{\E}[\varphi(\widetilde{Y}(t), S(t))]=0$ 
for bounded continuous functions on $\mathbb{R}^2_{+}$. It is however sufficient to consider only the dense set of functions
\[
\{ (y,s)\mapsto \phi(y)\psi(s) \,|\, \phi \in \mathcal{S}_0 \cup 1, \psi \in C_b(\mathbb{R}_+)\}.
\]
Note that we need to add the constant function $1$ to $\mathcal{S}_{0}$ to include functions with non-zero values at zero.
We now use the law of total expectation, the fact that $S$ is adapted to $\mathcal{W}^0$ and \ref{prop:MKV_existence_representation_stoachstic_representation_connection_to_SPDE} to conclude that
\begin{align*}
    \left|\E[\phi(Y(t))\psi(S(t))]-\widetilde{\E}[\phi(\widetilde{Y}(t))\psi(S(t))]\right|&=\left|\E\left[\psi(S(t))\E\left[\phi(Y(t))\big|\mathcal{W}^{0}_{t}\right]\right]-\widetilde{\E}\left[\psi(S(t))\widetilde{\E}\left[\phi(\widetilde{Y}(t))\big|\widetilde{\mathcal{W}}^{0}_{t}\right]\right]\right|\\&=\left|\E\left[\psi(S(t))\langle  \rho_t, \varphi\rangle\right]-\widetilde{\E}\left[\psi(S(t))\langle  \widetilde{\rho}_t, \varphi\rangle\right]\right|
    = 0,
\end{align*}
where the last equality follows from the fact that the expectation of $\langle  \rho_t, \varphi\rangle$ and $\langle  \widetilde{\rho}_t, \varphi\rangle$
coincides by pathwise uniqueness of \eqref{eqn:SPDE_classical_formulation}. 

\end{proof}

\begin{remark}
Note that uniqueness in law of $(Y, \mathbb{E}[Y|\mathcal{W}^0])$ could also be obtained by adapting the arguments of \cite[Theorem 4.4]{filipovic_larsson_16_polynomial_diffusion} from the Lipschitz assumption on the volatility of $Y$ to the square root case.
\end{remark}

\subsection{Density with respect to the Lebesgue measure/ Theorem \ref{thm:collected_results_McKean-Vlasov_SDE_introduction} \ref{prop:P_3:McKean-Vlasov_SDE_regularity_of_density_introduction}}\label{sub_sec:Regularity_of_density}

In this section, we want to verify that, on $\R\backslash \{0\}$ the law of $Y$ has a density with respect to the Lebesgue measure. We follow the ideas presented in \cite{Romito_18_density_SDE} and introduce the notation
\begin{align*}
    &\left(\Delta_{h}^{1} f\right)(x)=f(x+h)-f(x),\\
    &\left(\Delta_{h}^{n} f\right)(x)=\Delta_{h}^{1}\left(\Delta_{h}^{n-1} f\right)(x)=\sum_{j=0}^{n}(-1)^{n-j}\left(\begin{array}{l}
        n \\
        j
        \end{array}\right) f(x+j h).
\end{align*}
For $s>0,1 \leq p \leq \infty, 1 \leq q \leq \infty$, the sums
\begin{align*}
    \|f\|_{L^{p}}+[f]_{B_{p, q}^{s}}
\end{align*}
are equivalent norms on the Besov spaces $B_{p, q}^{s}\left(\mathbb{R}\right)=B_{p, q}^{s}$ for the given range of parameters. Here we have set
\begin{align*}
    [f]_{B^{s}_{p, q}}:=\left\|h \mapsto \frac{\left\|\Delta_{h}^{m} f\right\|_{L^{p}}}{|h|^{s}}\right\|_{L^{q}\left(B_{1}(0) ; \frac{d h}{| h|^{d}}\right)}.
\end{align*}
where $m$ is any integer such that $s<m$, and $B_{1}(0)$ is the unit ball in $\mathbb{R}$. For $f\in C^{m}(\R)$, we will also make use of the inequality
\begin{align}\label{eqn:finite_difference_inequality}
    \|\Delta^{m}_{h}f\|_{L^{1}(\R)}\leq C_{m}|h|^{m}\|\partial_{x}^{m}f\|_{L^{1}(\R)},
\end{align}
which can be shown recursively by considering $\left(\Delta_{h}^{n} f\right)(x)=\Delta_{h}^{1}\left(\Delta_{h}^{n-1} f\right)(x)$ 
 and the linearity of the differential operator $\partial_{x}$.
\begin{proposition}\label{prop:regularity_density_Y}
Let $Y$ be the solution of the McKean-Vlasov SDE  \eqref{eqn:MKV_SDE_introduction}, with coefficients $\alpha,\beta\geq 0$ and $\mu_{Y}$ denote its law. Let $\frac{1}{3}> \delta >0$, $m=\lceil \frac{3\left(1-\delta\right)^{2}}{2\delta}\rceil$ and define $\eta(x)\coloneqq \min\{1,x\}^{m}$. Then, for every $t\in (0,T]$, the measure $\eta(x)\mu_{Y(t)}(\dx)$ is absolutely continuous with respect to the Lebesgue measure on $[0,\infty)$ and its density lies in the Besov space $B^{1/2-\widetilde{\eps}}_{1,\infty}$, for any arbitrarily small but fixed $\widetilde{\eps}=\frac{3\delta}{2}>0$.
\end{proposition}

\begin{remark}
We note that the measure $\eta(x)\mu_{Y(t)}(\dx)$ is equivalent to $\mu_{Y(t)}(\dx)$ on $\R_{>0}$, which is the interior of the state space of $Y(t)$. Hence we can conclude that the law of $Y(t)$ is absolutely continuous with respect to the Lebesgue measure on $\{y\colon y\neq 0\}$.
In analogy to squared Bessel processes (see e.g., \cite[Chapter 11]{revuz2013continuous}) one could expect that the density of the law of $Y$ might be even absolutely continuous on $[0,\infty)$ if $\beta >0$ or at least if $\frac{\alpha}{2}\leq\beta$ (as in the latter case  $\{0\}$ is polar).
Our methodology however does not fully capture the regularity at $0$, since the approximation used in the proof does not preserve the behavior of the original equation close to $0$. 
It seems that, locally away from $0$, we capture the regularity of the density well, since the local approach from \cite[Section 4]{Romito_18_density_SDE} applied to our setting would yield local $B^{1/2-\eps}_{1,\infty}$ regularity (away from $0$), for any $\eps>0$.
\end{remark}

\begin{proof}

Let $0<t\leq T$ and $a\in (0,1)$. In order to prove the result, let $\eps\in (0,1)$ and define 

\begin{align}\label{eqn:Y_approximation}
        Y^{\eps}(r)\coloneqq
        &Y(t-\eps)+\int_{t-\eps}^{t}\beta \E\left[Y(t-\eps)\big|\mathcal{W}^{0}_{t-\eps}\right]\ds+\int_{t-\eps}^{t}\sqrt{\alpha}\sqrt{Y(t-\eps)\E\left[Y(t-\eps)\big|\mathcal{W}^{0}_{t-\eps}\right]}\dB_{s}+\int_{t-\eps}^{t}Y(t-\eps)\dW^{0}_{s},
\end{align}
for $r\geq t-\eps$ and $Y^{\eps}(r)\coloneqq Y(r)$, for $r\leq t-\eps$. Our goal is to show that there exists a constant $C_{T}>0$, such that for all $\varphi\in C^{a}_{b}(\R)$ ($a-$H{\"o}lder continuous and bounded) and $h\in \overline{B_{1}(0)}$, one has
\begin{align*}
    \left|\E\left[\eta(Y(t))\Delta_{h}^{m}\varphi(Y(t))\right]\right|\leq C_{T}\left(|h|^{a}\eps^{1/2}+\eps^{3 a/4}+|h|^{m}\eps^{-m/2}\right),
\end{align*}
for $\eps\in (0,1\wedge t)$. We first expand $\left|\E\left[\eta(Y(t))\Delta_{h}\varphi(Y(t))\right]\right|$ with $\pm \eta(Y(t-\eps))\Delta_{h}^{m}\varphi(Y(t))\pm \eta(Y(t-\eps))\Delta_{h}^{m}\varphi(Y^{\eps}(t))$ and split it into three components,
\begin{align}\label{eqn:Y_density_split_terms}
    \left|\E\left[\eta(Y(t))\Delta_{h}^{m}\varphi(Y(t))\right]\right|&=\left|\E\left[\left|\eta(Y(t))-\eta(Y(t-\eps))\right|\left|\Delta_{h}^{m}\varphi(Y(t))\right|\right]\right|\\
    &\phantom{xx}{}+\left|\E\left[\eta(Y(t-\eps))\left(\Delta_{h}^{m}\varphi(Y(t))-\Delta_{h}^{m}\varphi(Y^{\eps}(t))\right)\right]\right|\nonumber\\ &\phantom{xx}{}+\left|\E\left[\eta(Y(t-\eps))\Delta_{h}^{m}\varphi(Y^{\eps}(t))\right]\right|\nonumber\\
    &=I+II+III.
\end{align}

Using $|\eta(x)-\eta(y)|\leq C_{m} (|x-y|\wedge 1)$, we immediately obtain an estimate for $I$ by
\begin{align}\label{eqn:Y_density_regularity_t1}
    \left|\E\left[\left|\eta(Y(t))-\eta(Y(t-\eps))\right|\left|\Delta_{h}^{m}\varphi(Y(t))\right|\right]\right|&\leq C_{m}\|\varphi\|_{C^{a}}|h|^{a}\E\left|Y(t)-Y(t-\eps)\right|\leq  C\|\varphi\|_{C^{a}}|h|^{a}\eps^{1/2}.
\end{align}
The last inequality follows from H{\"o}lder's inequality and Lemma \ref{lem:estimates_Y} wit $p=2$. \newline
The second and the third term in \eqref{eqn:Y_density_split_terms} can be estimated similarly:
To bound $II$, recall that $|\eta(x)|\leq 1$. Since 
\begin{align*}
      &\left|\E\left[\eta(Y(t-\eps))\left(\Delta_{h}^{m}\varphi(Y(t))-\Delta_{h}\varphi(Y^{\eps}(t))\right)\right]\right|\leq \E\left[\left|\Delta_{h}^{m}\varphi(Y(t))-\Delta_{h}^{m}\varphi(Y^{\eps}(t))\right|\right]\\
      &\leq  \|\varphi\|_{C^{a}}\E\left[\left|Y(t)-Y^{\eps}(t)\right|^{a}\right]\leq \|\varphi\|_{C^{a}}\E\left[\left|Y(t)-Y^{\eps}(t)\right|^{2}\right]^{a/2},
\end{align*}
it suffices to estimate $\E\left[\left|Y(t)-Y^{\eps}(t)\right|^{2}\right]$. This can be done by Lemma \ref{lem:Y_density_error_estimate} wit $p=2$, which yields
\begin{align}\label{eqn:Y_density_regularity_t2}
     II=\left|\E\left[\eta(Y(t-\eps))\left(\Delta_{h}^{m}\varphi(Y(t))-\Delta_{h}^{m}\varphi(Y^{\eps}(t))\right)\right]\right|\leq C\|\varphi\|_{C^{a}}\eps^{3a/4}.
\end{align}
The third and final term in \eqref{eqn:Y_density_split_terms} can be estimated by Lemma \ref{lem:regularity_density_Y_eps}. 

Combining \eqref{eqn:Y_density_regularity_t1}, \eqref{eqn:Y_density_regularity_t2} and \eqref{eqn:Y_density_regularity_t3}, we conclude that
\begin{align*}
    \left|\E\left[\eta(Y(t))\Delta_{h}^{m}\varphi(Y(t))\right]\right|\leq C_{T}\left(|h|^{a}\eps^{1/2}+\eps^{3a/4}+|h|^{m}\eps^{-m/2}\right).
\end{align*}
Since $\eps$ was arbitrary, we make the ansatz $\eps=h^{\theta}$,
\begin{align*}
    \left|\E\left[\eta(Y(t))\Delta_{h}^{m}\varphi(Y(t))\right]\right|\leq C_{T}\left(|h|^{a+\frac{\theta}{2}}+|h|^{\frac{3a\theta}{4}}+|h|^{m\left(1-\frac{\theta}{2}\right)}\right).
\end{align*}
We want to choose $a, \theta, m$ such that the conditions of  \cite[Lemma A1]{Romito_18_density_SDE} are satisfied, which yields $\frac{4}{3}<\theta<2$ with free choice of $a$ and sufficiently large $m$. 
Choosing $\theta=2-2\delta$ for any fixed, but arbitrarily small $\delta>0$, yields
\begin{align*}
    \left|\E\left[\eta(Y(t))\Delta_{h}^{m}\varphi(Y(t))\right]\right|&\leq C_{T}\left(|h|^{a+1-\delta}+|h|^{a+\frac{a}{2}-\frac{3a\delta}{2}}+|h|^{m\delta}\right).
\end{align*}
 Ba choosing $m=\lceil \frac{3a(1-\delta)}{2\delta}\rceil$, we obtain
\begin{align*}
    \left|\E\left[\eta(Y(t))\Delta_{h}^{m}\varphi(Y(t))\right]\right|&\leq C_{T} |h|^{a+\frac{a}{2}-\delta}.
\end{align*}
We apply \cite[Lemma A1]{Romito_18_density_SDE} (see also \cite{Debussche_Fournier_13_densities_SDEs,Fournier_10_absolute_continuity_one_dimension}), with $s=a+\frac{a}{2}-\delta$, implying that $\eta(x)\mu_{Y(t)}(\dx)$ has a $B^{\frac{a}{2}-\delta}_{1,\infty}$ density with respect to the Lebesgue measure. Setting $a=1-\delta$, yields that $\frac{a}{2}-\delta=\frac{1}{2}-\frac{3\delta}{2}$ and it follows that $\eta(x)\mu_{Y(t)}(\dx)$ has a density with respect to the Lebesgue measure, which lies in $B^{1/2-\widetilde{\eps}}_{1,\infty}$, for $\widetilde{\eps}=\frac{3\delta}{2}$.
\end{proof}
\begin{lemma}\label{lem:regularity_density_Y_eps}
    Let the assumptions of Proposition \ref{prop:regularity_density_Y} be satisfied and $Y^{\eps}$ be given by \eqref{eqn:Y_approximation}. Let $a\in (0,1)$, $\varphi\in C^{a}(\R)$ and $m$ as in Proposition \ref{prop:regularity_density_Y}, then
\begin{align}\label{eqn:Y_density_regularity_t3}
     \left|\E\left[\eta(Y(t-\eps))\Delta_{h}^{m}\varphi(Y^{\eps}(t))\right]\right|\leq C\frac{|h|^{m}}{\eps^{m/2}}.
\end{align}
\end{lemma}
\begin{proof}
We will use the following representation of \eqref{eqn:Y_density_regularity_t3}:
\begin{align}\label{eqn:lem_density_Y_third_term}
    &\left|\E\left[\eta(Y(t-\eps))\Delta_{h}^{m}\varphi(Y^{\eps}(t))\right]\right|=\left|\E\left[\E\left[\eta(Y(t-\eps))\Delta_{h}^{m}\varphi(Y^{\eps}(t))\big|\mathcal{F}_{t-\eps}\right]\right]\right|\nonumber\\
    &=\bigg|\E\bigg[\E\bigg[\eta(Y(t-\eps))\Delta_{h}^{m}\varphi\bigg(Y(t-\eps)+\int_{t-\eps}^{t}\beta \E\left[Y(t-\eps)\big|\mathcal{W}^{0}_{t-\eps}\right]\ds\\
    &\phantom{xxxx}{}+\int_{t-\eps}^{t}\sqrt{\alpha Y(t-\eps)\E\left[Y(t-\eps)\big|\mathcal{W}^{0}_{t-\eps}\right]}\dB_{s}+\int_{t-\eps}^{t}Y(t-\eps)\dW^{0}_{t-\eps}\bigg)\big|\mathcal{F}_{t-\eps}\bigg]\bigg]\bigg|,\nonumber
\end{align}
where $\mathcal{F}_{t-\eps}$ is the filtration generated by $B, W^{0}$. The key observation is that the first two terms inside the innermost brackets are $\mathcal{F}_{t-\eps}$ measurable, while the two stochastic integrals are Gaussian random variables independent of  $\mathcal{F}_{t-\eps}$. Let
\begin{align*}
    &\sigma(Y(t-\eps))=\sqrt{\alpha Y(t-\eps)\E\left[Y(t-\eps)\big|\mathcal{W}^{0}_{t-\eps}\right]},\\
    &\sigma_{0}(Y(t-\eps))=Y(t-\eps).
\end{align*}
Defining 
\begin{align*}
    g(y_{1},y_{2})\coloneqq \E\left[\eta(y_{1})\Delta_{h}^{m}\varphi\left(y_{2} +\int_{t-\eps}^{t}\sigma(Y(t-\eps))\dB_{s}+\int_{t-\eps}^{t}\sigma_{0}(Y(t-\eps))\dW^{0}_{s}\right)|\mathcal{F}_{t-\eps}\right],
\end{align*}
\eqref{eqn:lem_density_Y_third_term} can be rephrased as
\begin{align*}
    \E\left[\eta(Y(t-\eps))\Delta_{h}^{m}\varphi(Y^{\eps}(t))\right]&=\E g\left(Y(t-\eps),Y(t-\eps)+\eps \beta \E\left[Y_{t-\eps}\big|\mathcal{W}^{0}_{t-\eps}\right] \right).
\end{align*}
 Let $I_{\eps,t}\coloneqq \int_{t-\eps}^{t}\sigma(Y(t-\eps))\dB+\int_{t-\eps}^{t}\sigma_{0}(Y(t-\eps))\dW^{0}$. $\operatorname{Var}(I_{\eps,t})$ i.e. the variance of the Gaussian terms is given by
\begin{align*}
    \operatorname{Var}(I_{\eps,t})=\eps \sigma(Y(t-\eps))^{2}+\eps \sigma_{0}(Y(t-\eps))^{2}=\eps \left(\alpha Y(t-\eps)\E\left[Y(t-\eps)\big|\mathcal{W}^{0}_{t-\eps}\right]+Y(t-\eps)^{2} \right).
\end{align*}
To bound \eqref{eqn:lem_density_Y_third_term} further, let us define $f_{\eps,t}(x)\coloneqq \frac{1}{\sqrt{2\pi \operatorname{Var}(I_{\eps,t})}}e^{-\frac{|x|^{2}}{2\operatorname{Var}(I_{\eps,t})}}$. 
\begin{align*}
    &\left|\eta(y_{1})\int_{\R}\Delta_{h}^{m}\varphi(y_{2}+x)\frac{1}{\sqrt{2\pi \operatorname{Var}(I_{\eps,t})}}e^{-\frac{|x|^{2}}{2\operatorname{Var}(I_{\eps,t})}}\dx\right|=\left|\int_{\R}\eta(y_{1})\Delta_{h}^{m}\varphi(y_{2}+x)f_{\eps,t}(x)\dx\right|\\
    &=\left|\eta(y_{1})\int_{\R}\varphi(y_{2}+x)\Delta_{-h}^{m}f_{\eps,t}(x)\dx\right|\leq \|\varphi\|_{\infty}\|\eta(y_{1})\Delta_{-h}^{m}f_{\eps,t}\|_{L^{1}(\R)}\leq \left|\eta(y_{1})\right|C |h|^{m}\|\partial_{x}^{m}f_{\eps,t}\|_{L^{1}(\R)}.
\end{align*}
The derivative on the right-hand side will involve certain Hermite polynomials. Via Rodrigues' formula the Hermite polynomial $H_{n}$, for each $n\in \N\cup \{0\}$, satisfies
\begin{align*}
    H_n(x) = (-1)^n \, e^{x^2} \, \partial_{x}^{n} e^{-x^2}.
\end{align*}
The change of variables $x \mapsto  \frac{x-\mu}{\sigma\sqrt{2}}$ yields
\begin{align*}
      H_n\left(\frac{x-\mu}{\sigma\sqrt{2}}\right) = (-1)^n \, e^{\frac{1}{2}\frac{(x-\mu)^2}{\sigma^2}} \, 2^{\frac{n}{2}} \, \sigma^n \,  \partial_{x}^{n} e^{-\frac{1}{2}\frac{(x-\mu)^2}{\sigma^2}}  
\end{align*}
and we obtain
\begin{align*}
    \partial_{x}^{n} e^{-\frac{1}{2}\frac{(x-\mu)^2}{\sigma^2}} = (-1)^n \left( \frac{1}{2 \sigma^2} \right)^{\frac{n}{2}} \, e^{-\frac{1}{2}\frac{(x-\mu)^2}{\sigma^2}} \, H_n\left(\frac{x-\mu}{\sigma\sqrt{2}}\right) \; .
\end{align*}
Hence,
\begin{align*}
    \partial_{x}^{m}f_{\eps,t}(x-rh)&=(-1)^{m}\frac{1}{\sqrt{\pi}}\left( \frac{1}{\sqrt{2} \sqrt{\operatorname{Var}(I_{\eps,t})}} \right)^{m} \, e^{-\frac{1}{2}\frac{(x-rh)^2}{\operatorname{Var}(I_{\eps,t})}} \, H_m\left(\frac{x-rh}{\sqrt{\operatorname{Var}(I_{\eps,t})}\sqrt{2}}\right)\\
    &=(-1)^{m}\left( \frac{1}{\sqrt{2} \sqrt{\operatorname{Var}(I_{\eps,t})}} \right)^{m-1} \,  H_m\left(\frac{x-rh}{\sqrt{\operatorname{Var}(I_{\eps,t})}\sqrt{2}}\right)f_{\eps,t}(x-rh)
\end{align*}
 A change of variables yields
\begin{align*}
    \left|\eta(y)\right|C\|\partial_{x}^{m}f_{\eps,t}\|_{L^{1}(\R)}&\leq C\frac{|\eta(y)||h|^{m}}{\operatorname{Var}(I_{\eps,t})^{m/2}}\leq C\frac{|h|^{m}}{\eps^{m/2}}.
\end{align*}    
\end{proof}

\section{Conditional propagation of chaos/ Proof of Proposition \ref{prop:conditional_propagation_of_chaos_introduction}}

This section is dedicated to the proof of Proposition \ref{prop:conditional_propagation_of_chaos_introduction}. We start with the following estimate.

\label{sub_sec:Conditional_propagation_of_chaos}
\begin{lemma}\label{lem:conditional_LLN}
Let Assumptions \ref{A:Assumptions_B} hold, $Y$ denote the solution of \eqref{eqn:MKV_SDE_introduction} and recall that the solution $\rho$ of \eqref{eqn:SPDE_classical_formulation} coincides with the conditional law of $Y$, given $\mathcal{W}^{0}$. Let $Y_{1},\dots,Y_{N}$ be identically distributed and, conditioned on $\mathcal{W}^0$, independent copies of $Y$ obtained by taking $N$ independent Brownian motions in \eqref{eqn:MKV_SDE_introduction}, then
\begin{align*}
    \sqrt{\E\bigg[\bigg|\sum_{i=1}^{N}\frac{1}{N}\varphi(Y_{i}(t))-\langle\rho_{t},\varphi\rangle\bigg|^{2}\bigg|\mathcal{W}^{0}\bigg]}\leq \frac{C_{\varphi}}{\sqrt{N}},
\end{align*}
for almost every $t\in[0,T]$ and every continuous and bounded function $\varphi$. 
\end{lemma}
\begin{proof}
Recall that by Definition \ref{def:strong_solution},  $\rho\in C([0,T],M_{1}(\R_{+}))$ and the pairing with a continuous, bounded function is well defined. We expand $\E\bigg[\bigg|\frac{1}{N}\sum_{i=1}^{N}\varphi(Y_{i}(t))-\langle\rho_{t},\varphi\rangle\bigg|^{2}\bigg|\mathcal{W}^{0}\bigg]$  in the following parts,
\begin{align*}
     \E\bigg[\bigg|\frac{1}{N}\sum_{i=1}^{N}\varphi(Y_{i}(t))-\langle\rho_{t},\varphi\rangle\bigg|^{2}\bigg|\mathcal{W}^{0}\bigg]&= \E\bigg[\bigg(\frac{1}{N}\sum_{i=1}^{N}\varphi(Y_{i}(t))\bigg)^{2}\bigg|\mathcal{W}^{0}\bigg]\\
     &\phantom{xx}{}-2 \E\bigg[\frac{1}{N}\sum_{i=1}^{N}\varphi(Y_{i}(t))\langle\rho_{t},\varphi\rangle\bigg|\mathcal{W}^{0}\bigg]+\E\bigg[\langle\rho_{t},\varphi\rangle^{2}\bigg|\mathcal{W}^{0}\bigg].
\end{align*}
First,
\begin{align*}
    \E\bigg[\bigg(\frac{1}{N}\sum_{i=1}^{N}\varphi(Y_{i}(t))\bigg)^{2}\bigg|\mathcal{W}^{0}\bigg]&=
    \frac{1}{N^{2}}\sum_{i=1}^{N}\E\big[\varphi(Y_{i}(t))^{2}\big|\mathcal{W}^{0}\big]+\frac{1}{N^{2}}\sum_{j\neq i=1}^{N}\E\big[\varphi(Y_{i})\varphi(Y_{j}(t))\big|\mathcal{W}^{0}\big]\\
    &=\frac{1}{N}\E\big[\varphi(Y_{1}(t))^{2}\big|\mathcal{W}^{0}\big]+\frac{N-1}{N}\E\big[\varphi(Y_{1}(t))\varphi(Y_{2}(t))\big|\mathcal{W}^{0}\big].
\end{align*}
Secondly, since $\rho$ is the conditional law of $Y_{1}$, and $\langle\rho_{t},\varphi\rangle=\E\big[\varphi(Y_{1})\big|\mathcal{W}^{0}\big]$ is $\mathcal{W}^{0}$ measurable, we obtain

\begin{align*}
    \mathbb{E}\big[\langle\rho_{t},\varphi\rangle^{2}\big|\mathcal{W}^{0}\big]=\E\big[\big(\E[\varphi(Y_{1}(t))|\mathcal{W}^{0}]\big)^{2}\big|\mathcal{W}^{0}\big]=\big(\E\big[\varphi(Y_{1}(t))\big|\mathcal{W}^{0}\big]\big)^{2},
\end{align*}
\begin{align*}
    \E\bigg[\frac{1}{N}\sum_{i=1}^{N}\varphi(Y_{i}(t))\langle\rho_{t},\varphi\rangle\bigg|\mathcal{W}^{0}\bigg]=\frac{1}{N}\sum_{i=1}^{N}\E\big[\varphi(Y_{i})\big|\mathcal{W}^{0}\big]\E\big[\varphi(Y_{1})\big|\mathcal{W}^{0}\big]=(\E\big[\varphi(Y_{1}(t))\big|\mathcal{W}^{0}\big])^2.
\end{align*}
Hence,
\begin{align*}
     \E\bigg[\bigg|\frac{1}{N}\sum_{i=1}^{N}\varphi(Y_{i}(t))-\langle\rho_{t},\varphi\rangle\bigg|^{2}\bigg|\mathcal{W}^{0}\bigg]&=\frac{1}{N}\E\big[\varphi(Y_{1}(t))^{2}\big|\mathcal{W}^{0}\big]+\frac{N-1}{N}\E\big[\varphi(Y_{1}(t))\varphi(Y_{2}(t))\big|\mathcal{W}^{0}\big]\\
     &\phantom{xx}{}+\big(\E\big[\varphi(Y_{1}(t))\big|\mathcal{W}^{0}\big]\big)^{2}-2(\E\big[\varphi(Y_{1}(t))\big|\mathcal{W}^{0}\big])^2\\
     &\leq \frac{\|\varphi\|_{\infty}^{2}}{N}+\frac{\|\varphi\|_{\infty}^{2}}{N},
\end{align*}
where we used
the conditional independence of $Y_{1}, Y_{2}$ implying that 
$$\E\big[\varphi(Y_{1}(t))\varphi(Y_{2}(t))\big|\mathcal{W}^{0}\big]=(\E\big[\varphi(Y_{1}(t))\big|\mathcal{W}^{0}\big])^2.$$ This yields the assertion.
\end{proof}

We are now ready to show Proposition \ref{prop:conditional_propagation_of_chaos_introduction}. The proof follows roughly the standard argumentation, see e.g. \cite[Lemma 3.6]{Gess_Coghi_19_stochastic_fokker_planck}, with the difference that we are not able to obtain a quantitative convergence rate since the usual Lipschitz conditions (\cite[Assumption 3.5]{Gess_Coghi_19_stochastic_fokker_planck}) are not satisfied. At this point, we need to rely on the convergence result, we obtained via tightness.

\begin{lemma}
Let Assumptions \ref{A:Assumptions_B} hold. The interacting particles $(X_{i}^{(N)})_{i=1,\dots,N}$ are $\rho$ chaotic, conditional to $\mathcal{W}^{0}$, in the sense that for each $k\in \N$ and $\varphi_{1},\dots,\varphi_{k}\in C_{b}(\R)$, $t\in[0,T]$, we have
\begin{align*}
    \lim_{N\rightarrow \infty}\E\bigg|\E\big[\varphi_{1}(X_{1}^{(N)}(t))\cdot\dots\cdot\varphi_{k}(X_{k}^{(N)}(t))\big|\mathcal{W}^{0}\big]-\Pi_{i=1}^{k}\langle \rho_{t},\varphi_{i}\rangle \bigg|=0.
\end{align*}
\end{lemma}

\begin{proof}
Without loss of generality, we can assume that we only have $k=2$ particles. The general case can be obtained inductively, by ``grouping'' the previously studied particles and letting them play the role of $\varphi_{1}(X_{1}^{(N)})$ in the following steps. Let $Y_{i}$ and $Y_{j}$ denote two  solutions of \eqref{eqn:MKV_SDE_introduction} driven by the common noise $W^{0}$ as well as independent Brownian motions $B^{i}$ and $B^{j}$, respectively. We note that, conditioned on $\mathcal{W}^{0}$, the particles $Y_{i}$ and $Y_{j}$ are independent of each other as long as $i\neq j$. This implies that
\begin{align*}
    \E[\varphi_{1}(Y_{i}(t))\varphi_{2}(Y_{i}(t))|\mathcal{W}^{0}]=\langle \varphi_{1},\Law(Y_{i}(t)|\mathcal{W}^{0})\rangle\langle \varphi_{2},\Law(Y_{j}(t)|\mathcal{W}^{0})\rangle=\langle \varphi_{1},\rho_{t}\rangle\langle \varphi_{2},\rho_{t}\rangle,
\end{align*}
$\Prob$--a.s. By the symmetry of the system \eqref{eqn:X_Ni_SDE}, it is immediate that $\forall t\in [0,T]$,
\begin{align*}
    \Law ((X_{i}^{(N)}(t),X_{j}^{(N)}(t))|\mathcal{W}^{0})=\Law ((X_{j}^{(N)}(t),X_{i}^{(N)}(t))|\mathcal{W}^{0}).
\end{align*}
Let $t\in[0,T]$. It follows $\Prob$--a.s. that
\begin{align*}
    \E[&\varphi_{1}(X_{1}^{(N)}(t))\varphi_{2}(X_{2}^{(N)}(t))|\mathcal{W}^{0}]-\langle \varphi_{1},\rho_{t}\rangle\langle \varphi_{2},\rho_{t}\rangle\\
    &=\E\big[\varphi_{1}(X_{1}^{(N)}(t))\varphi_{2}(X_{2}^{(N)}(t)-\varphi_{1}(Y_{1}(t))\varphi_{2}(Y_{2}(t))|\mathcal{W}^{0}\big]\\
    &= \E\big[\varphi_{1}(X_{1}^{(N)}(t))\big(\varphi_{2}(X_{2}^{(N)}(t))-\varphi_{2}(Y_{2}(t))\big)+\varphi_{2}(Y_{2}(t))\big(\varphi_{1}(X_{1}^{(N)}(t))-\varphi_{1}(Y_{1}(t))\big)\big|\mathcal{W}^{0}\big].
\end{align*}
Since the particle pairs $(X_{1}^{(N)}, X_{2}^{(N)})$ and $(Y_{1}, Y_{2})$  can be replaced by any pair $(X_{i}^{(N)}, X_{j}^{(N)})$ and $(Y_{i}, Y_{j})$ with $i\neq j$, respectively, and
\begin{align*}
    \Law((X_{i}^{(N)},Y_{i})|\mathcal{W}^{0})=\Law((X_{j}^{(N)},Y_{j})|\mathcal{W}^{0}),
\end{align*}
we obtain
\begin{align*}
    \E\big[&\varphi_{1}(X_{1}^{(N)}(t))\big(\varphi_{2}(X_{2}^{(N)}(t))-\varphi_{2}(Y_{2}(t))\big)+\varphi_{2}(Y_{2}(t))\big(\varphi_{1}(X_{1}^{(N)}(t))-\varphi_{1}(Y_{1}(t))\big)\big|\mathcal{W}^{0}\big]\\
    &=\E\bigg[\frac{1}{N(N-1)}\sum_{i\neq j=1}^{N}\varphi_{1}(X_{i}^{(N)}(t))\big(\varphi_{2}(X_{j}^{(N)}(t))-\varphi_{2}(Y_{j}(t))\big)\bigg|\mathcal{W}^{0}\bigg]\\
    &\phantom{xx}{}+\E\bigg[\frac{1}{N(N-1)}\sum_{i\neq j=1}^{N}\varphi_{2}(X_{j}(t))\big(\varphi_{1}(X_{i}^{(N)}(t))-\varphi_{1}(Y_{i}(t))\big)\bigg|\mathcal{W}^{0}\bigg].
\end{align*}
Next, we take the absolute value and expectation to obtain
\begin{align*}
    \E\bigg| &\E[\varphi_{1}(X_{1}^{(N)}(t))\varphi_{2}(X_{2}^{(N)}(t))|\mathcal{W}^{0}]-\langle \varphi_{1},\rho_{t}\rangle\langle \varphi_{2},\rho_{t}\rangle\bigg|\\
    &\leq\E\bigg|\E\bigg[\frac{1}{N(N-1)}\sum_{i\neq j=1}^{N}\varphi_{1}(X_{i}^{(N)}(t))\big(\varphi_{2}(X_{j}^{(N)}(t))-\varphi_{2}(Y_{j}(t))\big)\bigg|\mathcal{W}^{0}\bigg]\bigg|\\
    &\phantom{xx}{}+\E\bigg|\E\bigg[\frac{1}{N(N-1)}\sum_{i\neq j=1}^{N}\varphi_{2}(Y_{j}(t))\big(\varphi_{1}(X_{i}^{(N)}(t))-\varphi_{1}(Y_{i}(t))\big)\bigg|\mathcal{W}^{0}\bigg]\bigg|\\
    &\leq \|\varphi_{1}\|_{\infty}\E\E\bigg[\bigg|\frac{1}{N-1}\sum_{j=1, j\neq i}^{N}\big(\varphi_{2}(X_{j}^{(N)}(t))-\varphi_{2}(Y_{j}(t))\big)\bigg|\bigg|\mathcal{W}^{0}\bigg]\\
    &\phantom{xx}{}+\|\varphi_{2}\|_{\infty}\E\E\bigg[\bigg|\frac{1}{N}\sum_{i=1}^{N}\big(\varphi_{1}(X_{i}^{(N)}(t))-\varphi_{1}(Y_{i}(t))\big)\bigg|\bigg|\mathcal{W}^{0}\bigg].
\end{align*}
We expand with $\pm\frac{1}{N-1}\sum_{j=1,j\neq i}^{N}\langle\rho_{t},\varphi_{2}\rangle$ and $\pm\frac{1}{N-1}\sum_{j=1,j\neq i}^{N}\langle\rho_{t},\varphi_{1}\rangle$, respectively. This yields
\begin{align*}
    \E\bigg| &\E[\varphi_{1}(X_{1}^{(N)}(t))\varphi_{2}(X_{2}^{(N)}(t))|\mathcal{W}^{0}]-\langle \varphi_{1},\rho_{t}\rangle\langle \varphi_{2},\rho_{t}\rangle\bigg|\\
    &\leq  \|\varphi_{1}\|_{\infty}\E\bigg|\frac{1}{N-1}\sum_{j=1, j\neq i}^{N}\big(\varphi_{2}(X_{j}^{(N)}(t))-\langle\rho_{t},\varphi_{2}\rangle\big)\bigg|\\
    &\phantom{xx}{}+\|\varphi_{1}\|_{\infty}\E\bigg(\E\bigg[\bigg|\frac{1}{N-1}\sum_{j=1, j\neq i}^{N}\big(\langle\rho_{t},\varphi_{2}\rangle-\varphi_{2}(Y_{j}(t))\big)\bigg|^{2}\bigg|\mathcal{W}^{0}\bigg]\bigg)^{1/2}\\
    &\phantom{xx}{}+\|\varphi_{2}\|_{\infty}\E\bigg|\frac{1}{N}\sum_{i=1}^{N}\big(\varphi_{1}(X_{i}^{(N)}(t))-\langle\rho_{t},\varphi_{1}\rangle\big)\bigg|\\
    &\phantom{xx}{}+\|\varphi_{2}\|_{\infty}\E\bigg(\E\bigg[\bigg|\frac{1}{N}\sum_{i=1}^{N}\big(\langle\rho_{t},\varphi_{1}\rangle-\varphi_{1}(Y_{i}(t))\big)\bigg|^{2}\bigg|\mathcal{W}^{0}\bigg]\bigg)^{1/2}.
\end{align*}
The second and fourth term vanish due to Lemma \ref{lem:conditional_LLN}.
In order to argue that the first and third term vanish, we can combine the arguments from Proposition~\ref{prop:convergence_to_weak_solution_SPDE}, Lemma \ref{lem:convergence_of_mean_rho_N} and Lemma \ref{lem:EsupZp_EZp_estimates}. Recall \eqref{eqn:convergence_rho_t}, implying that $\rho^{N}\rightarrow \rho$ in $L^{1}(\Omega,(C([0,T],M_{1}))$.

\end{proof}

\section{Numerical simulations}\label{sec:numerics}
 One of the goals of this section is to compare the proposed model to actual market dynamics. The empirical data consists of 400 companies of the S\&P 500 index, which remained in the index over the considered time period. Indeed, starting from February 14th 2005, we considered the following 4637 business days until September 9th 2023. 
 We start by studying the relative capitalizations, i.e. 
 \[
 \mu_i(t):= \frac{X_i(t)}{\sum_{j=1}^N X_j(t)}.
 \]
 The first plot (Figure \ref{fig:comparison_market_curves_simulated_curves}) compares several relative capitalization curves from different dates as well as the (time) average of the relative capitalization curves with the average of the simulated curves. For this, we calculated, normalized, and averaged the empirical curves for the chosen 400 assets over the past 4637 trading days (roughly 270 months). The simulations were initiated with 400 particles starting at the given market capitalizations of the 400 empirical assets on February 14th 2005. The evolution of the system \eqref{eqn:X_Ni_SDE}, with parameters $\alpha=2$, $\beta=1$, was modeled by a Milstein-scheme over 270 time steps and analogously to the empirical data, the capitalizations were ranked, normalized and averaged. 

\begin{figure}[H]
\centering
 \includegraphics[scale=0.6]{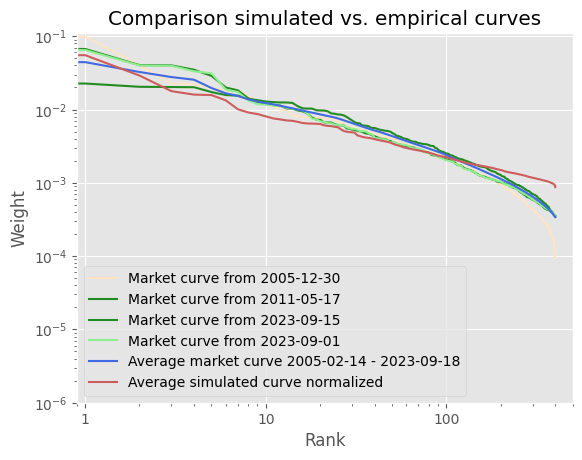}
    \caption{Comparison plot} \label{fig:comparison_market_curves_simulated_curves}
\end{figure}

To obtain the values for $\alpha$ and $\beta$ we calculated the market entropy (\cite[Section 2.3/Section 3.4, Example 3.4.3]{fernholz_02_stochastic_portfolio_theory})
and the market diversity (\cite[Section 3.4, Example 3.4.4]{fernholz_02_stochastic_portfolio_theory}) from our simulated system with different parameter combinations of $\alpha,\beta$. We then compared the results with the market entropy and diversity calculated from the empirical market data and it turned out that the parameter combination $\alpha=2$ and $\beta=1$ worked best. A more careful calibration is beyond the scope of this work.

Next, we want to study the impact of the common noise on the behavior of the particle system. For this purpose, we introduce a new parameter $\gamma\geq 0$ regulating the strength of the common noise in the system \eqref{eqn:X_Ni_SDE}, which yields the following system of equations
\begin{align}\label{eqn:X_Ni_SDE_gamma}
    \d X_{i}^{(\gamma)}(t)=\frac{\beta}{N}\sum_{j=1}^{N}X_{j}^{(\gamma)}(t)\dt+\sqrt{\frac{\alpha}{N}}\sqrt{X_{i}^{(\gamma)}(t)}\sqrt{\sum_{j=1}^{N}X_{j}^{(\gamma)}(t)}\dW_{t}^{i}+\gamma \sqrt{1-\frac{\alpha}{N}}X_{i}^{(\gamma)}(t)\dW^{0}_{t}.
\end{align}
For the choice $\alpha=1$ and $\gamma=0$, \eqref{eqn:X_Ni_SDE_gamma} corresponds to the system studied in \cite{Shkolnikov_13_large_volatility-stabilized}. 
We initialize \eqref{eqn:X_Ni_SDE_gamma} with $\gamma=1$ (orange) and $\gamma=0$ (blue) at the same initial condition which corresponds to the absolute market capitalization in our empirical data from the 14.02.2005. With the exception of $\gamma$, all other parameters are identical ($\alpha=2, \beta=1$) in both models. We let the two systems, each with 400 particles, evolve over 270 simulated time steps. For the system with common noise
we consider 1000 different simulations of the common noise 
and average each particle over them to approximate the empirical distribution. 
Then we compare the distribution of the market capitalization for both models at the final time step, see Figure \ref{fig:Densities_SDE_noise_comparison}.

  \begin{minipage}[H]{0.45\textwidth}
    \includegraphics[width=\textwidth]{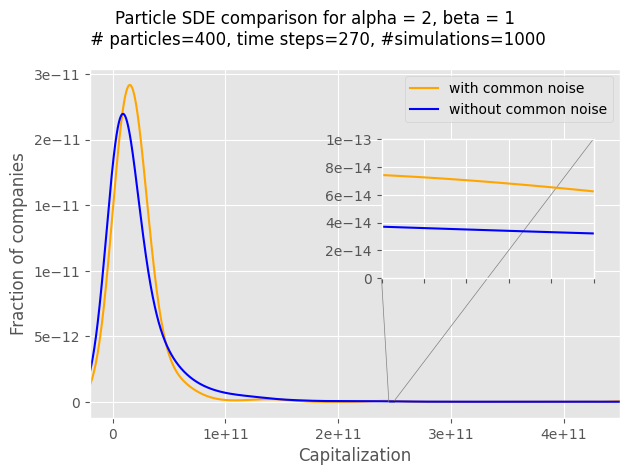}
    \captionof{figure}{Densities SDE system}\label{fig:Densities_SDE_noise_comparison}

  \end{minipage}
  \begin{minipage}[H]{0.45\textwidth}
    \includegraphics[width=\textwidth]{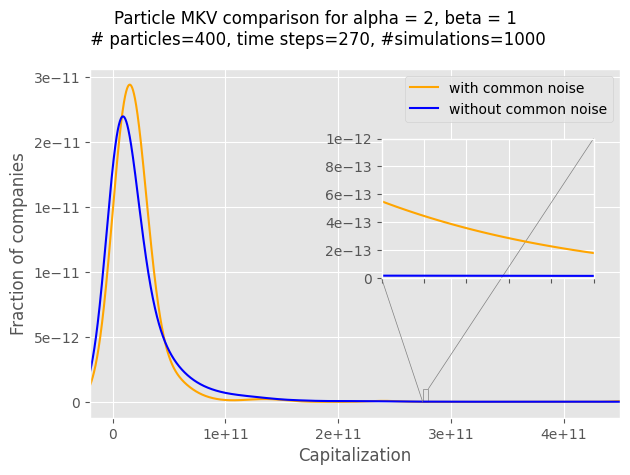}
    \captionof{figure}{Densities MKV system}\label{fig:Densities_MKV_noise_comparison}
  \end{minipage}

As expected, the distributions in the case $\gamma=1$ exhibit a heavier tail behavior due to the additional noise. Note that the mass being present at non-positive capitalizations is an artifact of the smooth approximation of the empirical densities.
The same behavior is observed when simulating the systems with the corresponding McKean-Vlasov dynamics (again 1000 different simulations for the common noise, each with 400 particles) for the systems with and without common noise (see Figure \ref{fig:Densities_MKV_noise_comparison}), which is given by
   \begin{align}
        &\d Y(t)=\beta m_{\lambda}\exp{\bigg(\bigg(\beta-\frac{1}{2}\bigg)t+W^{0}_{t}\bigg)}\dt+\sqrt{\alpha}\sqrt{m_{\lambda}\exp{\bigg(\bigg(\beta-\frac{1}{2}\bigg)t+W^{0}_{t}\bigg)}}\sqrt{Y(t)}\dB_{t}+Y(t)\dW^{0}_{t},
       \label{eqn:sim_MKV_common_noise}\\
          &\d Y_{\operatorname{nocn}}(t)=\beta m_{\lambda}\exp{\bigg(\beta t\bigg)}\dt+\sqrt{\alpha}\sqrt{m_{\lambda}\exp{\bigg(\beta t\bigg)}}\sqrt{Y_{\operatorname{nocn}}(t)}\dB_{t},\label{eqn:sim_MKV_no_common_noise}
    \end{align}
where $W^{0},B$ are independent Brownian motions.

In the following (Figures \ref{fig:time_evolution_SDE} and \ref{fig:time_evolution_market}), we look at the time evolution of the simulated capital densities of \eqref{eqn:X_Ni_SDE_gamma} with $\gamma=1$ and compare it to the empirical market densities. As mentioned above, for 1000 simulations of the common noise, a system of $400$ particles, with parameters $\alpha=2, \beta=1$, was initialized at the empirical market capitalizations from February 14th 2005 and we let it evolve over 270 time steps. At each time step, we averaged each particle over the 1000 simulations of the common noise term and approximated the associated density. For the plot we choose 5 out of the 270 time steps for the simulated system and 100 time steps (out of 4637) for the empirical asset distributions to plot.  We observe a similar behavior with slightly more vertical variation in the empirical system.

\begin{minipage}[H]{0.45\textwidth}
    \includegraphics[width=\textwidth]{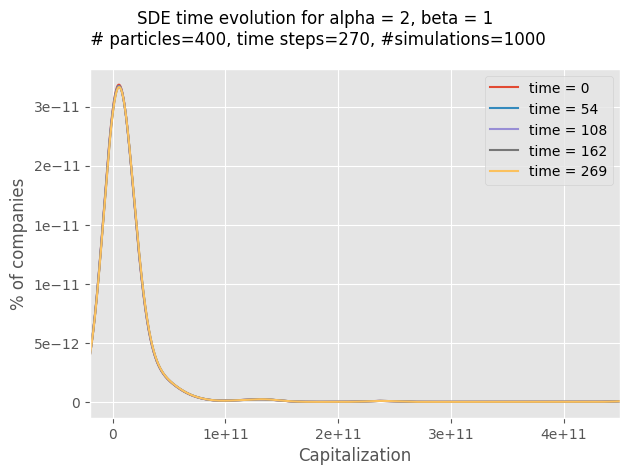}
    \captionof{figure}{Time evolution particle system} \label{fig:time_evolution_SDE}

  \end{minipage}
  \begin{minipage}[H]{0.45\textwidth}
    \includegraphics[width=\textwidth]{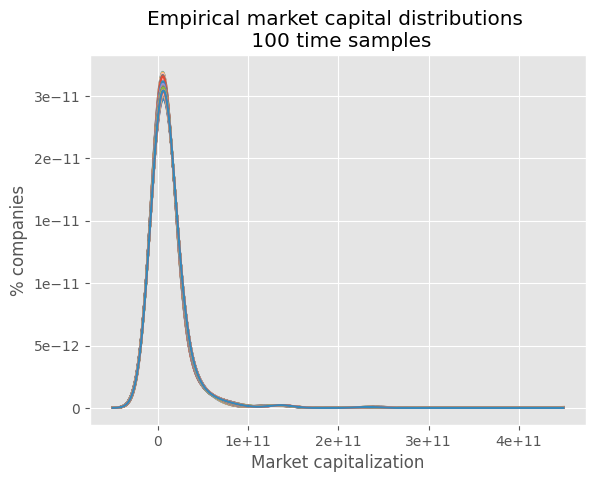}
    \captionof{figure}{Time evolution market}\label{fig:time_evolution_market}
  \end{minipage}

For ``large'' values of $N$, the empirical measure of the particle system should be well approximated by the solution of \eqref{eqn:SPDE_classical_formulation}, i.e.,
\begin{align*}
    \d \rho_{t}&= \frac{\alpha}{2}\langle \rho_{t},\idx\rangle\partial_{x}^{2}(x \rho_{t} )+\frac{1}{2}\partial_{x}^{2}(x^{2} \rho_{t} )-\beta\langle \rho_{t},\idx\rangle \partial_{x} \rho_{t} \dt+\partial_{x}(x \rho_{t} )\dW^{0}_{t},\quad \rho_{0}=\rho^{(0)}.
\end{align*}

To simulate \eqref{eqn:SPDE_classical_formulation}, we use a $P1$-Finite Element Method (FEM) approximation over a grid with $0$ Dirichlet boundary and $1000$ grid points. We let the equation evolve again over $270$ time steps with one simulation of the common noise. 
The initial condition was chosen as a $\chi^{2}$ density with parameter $k=3$, which appeared to be a good approximation of the empirical capital distribution. 

In the following plot, we compare the space profiles, at different points in time, of the solution of the SPDE \eqref{eqn:SPDE_classical_formulation}, simulated as indicated above, and the PDE
\begin{align*}
    \d u_{t}&= \frac{\alpha}{2}\langle u_{t},\idx\rangle\partial_{x}^{2}(x u_{t} )-\beta\langle u_{t},\idx\rangle \partial_{x} u_{t} \dt,\quad
    u_{0}=\rho^{(0)},
\end{align*}
corresponding to the limit of $u(t)=\lim_{N\rightarrow \infty}\frac{1}{N}\sum_{i=1}^{N}\delta_{X_{i}^{(\gamma=0)}(t)}$ of the empirical distributions of system \eqref{eqn:X_Ni_SDE_gamma} with $\gamma=0$. The PDE was simulated with the same specifications as the SPDE.
 As indicated by the following plots, the noise has a visible impact on the SPDE (compare Figure \ref{fig:9_SPDE} and Figure \ref{fig:10_PDE}). 

\begin{center}
 \begin{minipage}[b]{0.4\textwidth}

    \includegraphics[width=\textwidth]{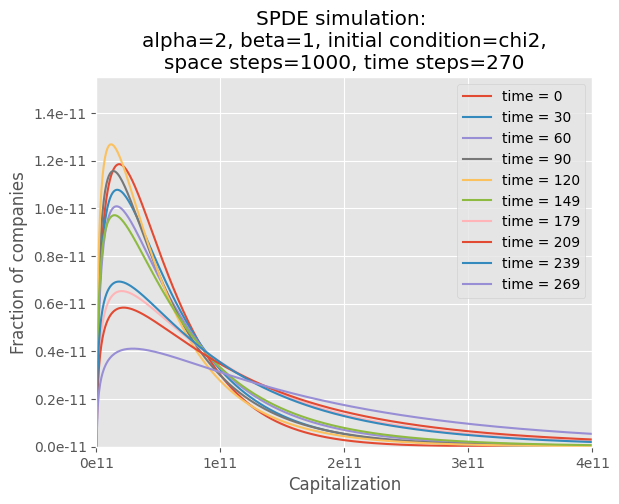}

    \captionof{figure}{SPDE $\chi^{2}$ i.c.}\label{fig:9_SPDE}
  \end{minipage}
  \begin{minipage}[b]{0.4\textwidth}
    \includegraphics[width=\textwidth]{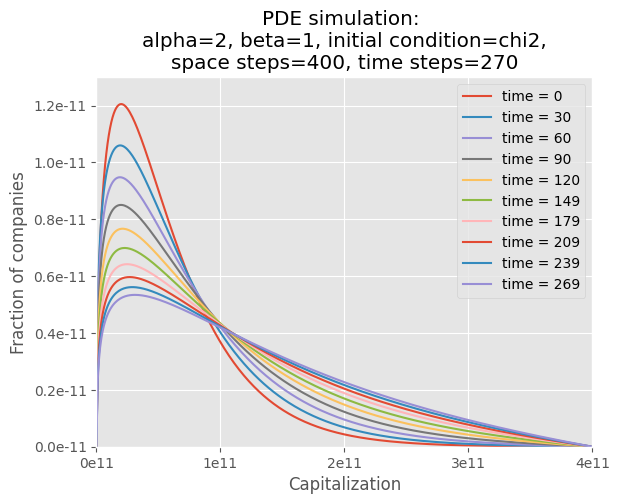}
        \captionof{figure}{PDE $\chi^{2}$ i.c.}\label{fig:10_PDE}
\end{minipage}
  \end{center}
Lastly, we compare the time evolution of the SPDE, simulated under the previous specifications but with a different realization of the common noise, depicted in Figure \ref{fig:SPDE_time_evolution_comparison_MKV} and the time evolution of the approximated density of the McKean-Vlasov system \eqref{eqn:sim_MKV_common_noise}, depicted in Figure \ref{fig:MKV_time_evolution_comparison_SPDE}. For the McKEan-Vlasov system 100000 independent Brownian motions and the same common noise $W^{0}$ used in this iteration of the SPDE simulation, were used. To be comparable to the SPDE, whose initial condition was chosen to be the density of a $\chi^{2}(3)$ distribution, the McKean-Vlasov system was initialized with i.i.d. initial conditions $Y_{0}\sim \chi^{2}(3)$. 
\begin{center}
 \begin{minipage}[b]{0.4\textwidth}
    \includegraphics[width=\textwidth]{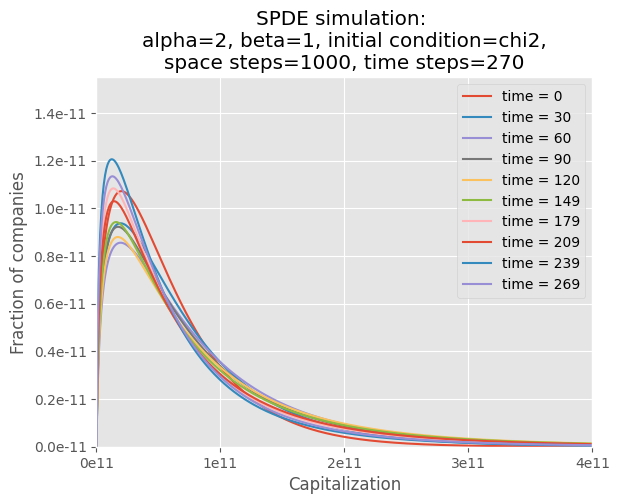}
    \captionof{figure}{SPDE $\chi^{2}$ i.c.}\label{fig:SPDE_time_evolution_comparison_MKV}
  \end{minipage}
 \begin{minipage}[b]{0.4\textwidth}
    \includegraphics[width=\textwidth]{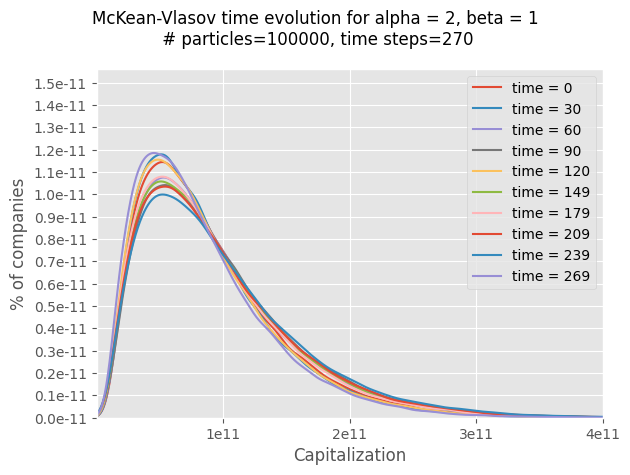}
    \captionof{figure}{McKean-Vlasov $\chi^{2}$ i.c.}\label{fig:MKV_time_evolution_comparison_SPDE}
  \end{minipage}

  \end{center}
  The behavior seems consistent, up to the differences caused by the numerical implementation and approximation of the initial condition.

\section{Appendix}
\subsection{Appendix A: Results regarding weighted Sobolev spaces}
For the sake of generality, we will state and prove the results here for $\R^{d}$ instead of $\R_{+}$. The case $\R_{+}$ can be recovered by setting $d=1$ and using an extension theorem for Sobolev functions. Since we are interested in the case $W^{m,p}_{0}(a,b)$ defined in Definition \ref{def:fractionalSobolev}, we can choose the trivial extension of the weight functions to $\R$. We recall the notation $\partial_{x}^{\gamma}f=\frac{\partial^{\gamma_{1}}}{\partial x_{1}^{\gamma_{1}}}\dots\frac{\partial^{\gamma_{d}}}{\partial x_{d}^{\gamma_{d}}}f$ for a multi-index $\gamma\in \N^{d}$.
\begin{lemma}\label{lem:weighted_spaces_embeddings}
Let $a \geq 0$, $m\in \N$, $a +m b \geq 0$ and $W_{0}^{m,p}(a , b )=W_{0}^{m,p}(a , b ,\R^{d})$ be the weighted Sobolev space of functions, for which 
\begin{align*}
    \sum_{\gamma \in \N^{d}\colon |\gamma|\leq m}\|(1+|x|^{2})^{a +|\gamma| b }\partial_{x}^{\gamma}f\|_{L^{p}}^{p}<\infty,
\end{align*}
then the following assertion holds:
For $0\leq a '<a $ and $0\leq  b '\leq  b $, the embedding 
    \begin{align*}
        W_{0}^{m,p}(a , b )\hookrightarrow W_{0}^{n,q}(a ', b '), \textnormal{ where }m>n,\,
    \end{align*}
is continuous in the case $\frac{1}{p}-\frac{m}{d}= \frac{1}{q}-\frac{n}{d}$ and compact in the case $ \frac{1}{p}-\frac{m}{d}< \frac{1}{q}-\frac{n}{d}$.
\end{lemma}
\begin{proof}
 For the continuity, we note that the norms
$\|f\|_{W^{n,q}(a,0 )}$ and $\|(1+|x|^{2})^{a}f\|_{W^{n,q}(0,0 )}$ are equivalent, which follows by similar arguments as in \cite[Section 4.2.2, Theorem 1]{edmunds_triebel1996function}. This also yields the continuous embedding $ W_{0}^{m,p}(a , 0)\hookrightarrow  W_{0}^{n,q}(a ,0)$, which we use to conclude that
\begin{align*}
    \|f\|_{W^{n,q}(a,b )}&=\left(\|f\|^{q}_{L^{q}(a)}+\sum_{\gamma\in \N^{d}\colon |\gamma|=1}^{n}\|\partial_{x}^{\gamma} f\|^{q}_{L^{q}(a+|\gamma|b)}\right)^{1/q}\\
    &\leq \left(\|f\|^{q}_{W^{n,q}(a,0 )}+\sum_{\gamma\in \N^{d}\colon |\gamma|=1}^{n}\|\partial_{x}^{\gamma} f\|^{q}_{W^{n-|\gamma|,q}(a+|\gamma|b,0 )}\right)^{1/q}\\
    &\leq \left(\|f\|^{q}_{W^{m,p}(a,0 )}+\sum_{\gamma\in \N^{d}\colon |\gamma|=1}^{n}\|\partial_{x}^{\gamma} f\|^{q}_{W^{m-|\gamma|,p}(a+|\gamma|b,0 )}\right)^{1/q}\\
    &\leq \left(\|f\|^{q}_{W^{m,p}(a,b )}+\sum_{\gamma\in \N^{d}\colon |\gamma|=1}^{n}\|\partial_{x}^{\gamma} f\|^{q}_{W^{m-|\gamma|,p}(a+|\gamma|b,b)}\right)^{1/q}\\
        &\leq C_{n,d} \|f\|_{W^{m,p}(a,b )}.
\end{align*}

 Let $B_{R}\subseteq \R^{d}$ be a ball with radius $R>0$, centered at the origin. 
 We introduce the notation $W^{k,p}(B_{R})$ for the Sobolev space of functions from $B_{R}\rightarrow \R$, equipped with the usual norm.
  For $f\in W^{k,p}(\R^{d})$ and any multi-index $\gamma\in\N^{d}$, such that $0\leq|\gamma|\leq k$, we have
\begin{align*}
    \left\|\frac{\partial^{\gamma_{1}}}{\partial x_{1}^{\gamma_{1}}}\dots\frac{\partial^{\gamma_{d}}}{\partial x_{d}^{\gamma_{d}}}f\right\|_{L^{p}(B_{R})}&=\|\partial_{x}^{\gamma}f\|_{L^{p}(B_{R})}\leq \max\{1,R^{-2a -2 b |\gamma|}\}\|(1+|x|^{2})^{a + b |\gamma|}\partial_{x}^{\gamma}f\|_{L^{p}(B_{R})}\\
    &\leq \max\{1,R^{-2a -2 b |\gamma|}\}\|(1+|x|^{2})^{a + b |\gamma|}\partial_{x}^{\gamma}f\|_{L^{p}(\R^{d})}. 
\end{align*} 
Hence,
\begin{align*}
    \|u\|_{W^{m,p}(B_{R})}\leq \max\{1,R^{-a}\}\|u\|_{W^{m,p}(a , b ,B_{R})}.
\end{align*}

We conclude that, for any $R\geq 0$, the restriction map 

\begin{align*}
   J_{R} & \colon W^{m,p}(a , b ,\R^{d})\rightarrow W^{m,p}(B_{R}),\\
         &u\mapsto u \big|_{B_{R}},
\end{align*}
is well-defined and bounded. Note that on $B_{R}$, weighted spaces with a positive weight function and unweighted spaces are norm equivalent. 

By \cite[Theorem 7.26]{gilbarg_trudinger_77_elliptic_PDE}, we know that the embedding $W_{0}^{m,p}(B_{R})\hookrightarrow W_{0}^{n,q}(B_{R})$ is compact if $\frac{1}{p}-\frac{m}{d}<\frac{1}{q}-\frac{n}{d}$. 
 Hence, the restriction map $J_{R}\colon  W_{0}^{m,p}(a , b ,\R^{d})\rightarrow W^{n,q}(B_{R})$ is compact. Let $\big(u_{n}\big)_{n\in\N}$ be a bounded sequence in $ W_{0}^{m,p}(a , b ,\R^{d})$ and $\frac{1}{p}-\frac{m}{d}<\frac{1}{q}-\frac{n}{d}$. Let $R=1$, we can find a subsequence $(u_{1n})_{n\in\N}$, which converges in $W_{0}^{n,q}(B_{1})$ and pointwise almost everywhere.
 \begin{align*}
    u_{11},u_{12},\dots,\quad \textnormal{ such that } \|u_{1n}-u_{1n'}\|_{W_{0}^{n,q}(B_{1})}\rightarrow 0,\quad \textnormal{ as } n, n'\rightarrow \infty.
\end{align*}
From this subsequence, we can extract a further subsequence
\begin{align*}
    u_{21},u_{22},\dots,\quad \textnormal{ such that } \|u_{2n}-u_{2n'}\|_{W_{0}^{n,q}(B_{2})}\rightarrow 0,\quad \textnormal{ as } n, n'\rightarrow \infty.
\end{align*}
Repeating this procedure for increasing radii $R$, we obtain a sequence of sequences
\begin{align}\label{eqn:subsequence_matrix}
\begin{matrix}
    u_{11},&u_{12},&\dots,&u_{1n},&\dots,&u_{1n'},&\dots\\
    u_{21},&u_{22},&\dots,&u_{2n},&\dots,&u_{2n'},&\dots\\
    \vdots&\vdots&\vdots&\vdots&\vdots&\vdots&\dots\\
    u_{k1},&u_{k2},&\dots,&u_{kn},&\dots,&u_{kn'},&\dots\\
    \vdots&\vdots&\vdots&\vdots&\vdots&\vdots&\dots
\end{matrix}
\end{align}
with the property that
\begin{align*}
     \|u_{kn}-u_{kn'}\|_{W_{0}^{n,q}(B_{k})}\rightarrow 0,\quad \textnormal{ as } n, n'\rightarrow \infty.
\end{align*}
 By the definition of the norm, we may assume that all partial weak derivatives also converge pointwise almost everywhere. We choose the diagonal sequence $\big(u_{ii}\big)_{i\in\N}$ and show that it has a well-defined limit. We can define the function $\overline{u}(x)=u_{k}(x)$, for $x\in B_{k}$, where $u_{k}$ is given by the limit $u_{k,n}\rightarrow u_{k}$ in $W_{0}^{n,q}(B_{k})$, for every $k$ and pointwise almost everywhere. By pointwise convergence and Fatou's lemma, it follows that this pointwise limit, denoted $\overline{u} \in W_{0}^{n,q}(a , b ,\R^{d})$. Let $a '< a $ and $ b '\leq  b $, then for any $\eps>0$
\begin{align*}
    &\left(\int_{\R^{d}\backslash B_{R}}\sum_{\gamma\in \N^{d}\colon |\gamma|\leq n}(1+|x|^{2})^{qa '+qk b '}|\partial_{x}^{\gamma}u_{ii}|^{q}\dx\right)^{1/q}\\
    &\leq C \left(\int_{\R^{d}\backslash B_{R}}\sum_{\gamma\in \N^{d}\colon |\gamma|\leq m}(1+|x|^{2})^{pa '+pk b '}|\partial_{x}^{\gamma}u_{ii}|^{p}\dx\right)^{1/p}\\
    &\leq C\sum_{\gamma\in \N^{d}\colon |\gamma|\leq m}(1+|R|^{2})^{a '-a +k( b '- b )}\left(\int_{\R^{d}\backslash B_{R}}(1+|x|^{2})^{pa +pk b }|\partial_{x}^{\gamma}u_{ii}|^{p}\dx\right)^{1/p}\leq \eps,
\end{align*}
provided $R$ is chosen sufficiently large. Since $u_{ii}\rightarrow \overline{u}$ in $W_{0}^{n,q}(B_{R})$, it follows that
\begin{align*}
    \left(\int_{B_{R}}\sum_{\gamma\in \N^{d}\colon |\gamma|\leq n}|\partial_{x}^{\gamma}u-\partial_{x}^{\gamma}u_{ii}|^{p}\dx\right)^{1/p}\leq \eps,
\end{align*}
for all $i\geq I$, where $I$ depends on $\eps$ and $R$. Combining these observations yields
\begin{align*}
    \|u_{ii}-\overline{u}\|_{W^{n,q}(a ', b ',\R^{d})}&\leq  CR^{a '+n b '}\|u_{ii}-\overline{u}\|_{W^{n,q}(B_{R})}+\|u_{ii}-\overline{u}\|_{W^{n,q}(a ', b ',\R^{d}\backslash B_{R})}\\
    &\leq \eps+2\eps.
\end{align*}

\end{proof}

\subsection{Appendix B: Mapping properties of the operators}
\begin{lemma}\label{lem:coercivity_L_sigma}
Let $L$ and $\sigma$ be defined by
\begin{align*}
    &L\varphi\coloneqq \partial_{x}(f(x)\nabla\varphi)+g(x)\partial_{x} \varphi\\
    &\sigma(\varphi)\coloneqq \varphi-\partial_{x} (x\varphi),
\end{align*}
where
\begin{align*}
     &f(x,t,\omega)=\left(m_{\lambda}\exp{((\beta-1/2)t+W^{0}_{t})}\frac{\alpha}{2}x+\frac{x^{2}}{2}\right),\\
     &g(x,t,\omega)=-\left(\left(\frac{\alpha}{2}- \beta\right) m_{\lambda}\exp{((\beta-1/2)t+W^{0}_{t})}+\idxo\right).
 \end{align*}
 Let $a\geq 0, b\geq \frac{1}{2}$.
  Then, for $u\in H_{1}=W^{m,2}_{0}(a,b)$, we have the following weak coercivity (weak parabolicity) estimate

\begin{align}\label{eqn:weak_monotonicity_L_sigma}
     \langle Lu,u\rangle_{H_{1}}&+\frac{1}{2} \sup_{\varphi\in H_{1}\colon \|\varphi\|_{H_{1}}=1}|\langle u,\sigma\varphi\rangle_{H_{1}}|^{2}\leq C_{\alpha,\beta,m}(1+m_{\lambda}\exp{((\beta-1/2)t+W^{0}_{t})})\|u\|^{2}_{H_{1}}.
\end{align}
\end{lemma}
Recall that by Lemma \ref{lem:rho_x}, $\langle \rho_{t},\idx\rangle=m_{\lambda}\exp{((\beta-1/2)t+W^{0}_{t})}$.

\begin{proof}
We will drop the explicit dependence on $\omega$ and $t$ and consider them to be fixed. Furthermore, we introduce the constants $a,b \geq 0$, whose values will be determined during the proof. 
 Before we begin, we introduce the following notation to keep the proof more concise:
\begin{align*}
    \V(t,\omega)\coloneqq   m_{\lambda}\exp{((\beta-1/2)t+W^{0}_{t})}.
\end{align*}
By density, it suffices to perform the following calculations for $u$ being smooth and compactly supported. Note that for any two such functions $v_{1},v_{2}$, the boundary term in the following integration by parts vanishes and we obtain \begin{align*}
    \langle \partial_{x}v_{1},v_{2}\rangle_{L^{2}}&=\int_{0}^{\infty}\partial_{x}v_{1}(x)v_{2}(x)\dx=v_{1}(x)v_{2}(x)\bigg|_{x=0}^{x=\infty}-\int_{0}^{\infty}v_{1}(x)\partial_{x}v_{2}(x)\dx\\
    &=-\int_{0}^{\infty}v_{1}(x)\partial_{x}v_{2}(x)\dx=-\langle v_{1},\partial_{x}v_{2}\rangle_{L^{2}}.
\end{align*}

Also recall that $\partial_{x}v_{1}(x)v_{1}(x)=\frac{1}{2}\partial_{x}\left(v_{1}(x)^{2}\right)$.
\begin{align*}
    \langle Lu,u\rangle_{H_{1}}&=\sum_{\gamma = 0}^{m}\langle \partial_{x}^{\gamma}\partial_{x} (f(x)\partial_{x} u)(1+x^{2})^{a+\gamma b},\partial_{x}^{\gamma}u(1+x^{2})^{a+\gamma b}\rangle_{L^{2}}\\
    &\phantom{xx}+\sum_{\gamma = 0}^{m}\langle \partial_{x}^{\gamma}(g(x)\partial_{x} u)(1+x^{2})^{a+\gamma b},\partial_{x}^{\gamma}u(1+x^{2})^{a+\gamma b}\big)\rangle_{L^{2}}\\
    &=I+II.
\end{align*}
We begin by studying the first term. By integration-by-parts and the product rule, we obtain
\begin{align*}
    I&=\sum_{\gamma = 0}^{m}\langle \partial_{x}^{\gamma+1} (f(x)\partial_{x} u)(1+x^{2})^{a+\gamma b},\partial_{x}^{\gamma}u(1+x^{2})^{a+\gamma b}\rangle_{L^{2}}\\
    &=-\sum_{\gamma = 0}^{m}\langle \partial_{x}^{\gamma}(f(x)\partial_{x} u)(1+x^{2})^{a+\gamma b},\partial_{x}^{\gamma+1}u (1+x^{2})^{a+\gamma b}\rangle_{L^{2}}\\
    &\phantom{xx}-\sum_{\gamma = 0}^{m}\langle \partial_{x}^{\gamma}(f(x)\partial_{x} u)(1+x^{2})^{a+\gamma b},\partial_{x}^{\gamma}u (2a+2\gamma b)(1+x^{2})^{a+\gamma b-1}2x\rangle_{L^{2}}\\
    &=-\sum_{\gamma = 0}^{m}\left\langle \left(\sum_{k=0}^{\gamma}C_{k}^{\gamma}\chi_{k \leq 2}\partial_{x}^{k }f(x)\partial_{x}^{\gamma -k}\partial_{x} u\right) (1+x^{2})^{a+\gamma b},\partial_{x}^{\gamma+1}u (1+x^{2})^{a+\gamma b}\right\rangle_{L^{2}}\\
    &\phantom{xx}-\sum_{\gamma = 0}^{m}\left\langle \left(\sum_{k=0}^{\gamma}C_{k}^{\gamma}\chi_{k \leq 2}\partial_{x}^{k }f(x)\partial_{x}^{\gamma -k}\partial_{x} u\right)(1+x^{2})^{a+\gamma b},\partial_{x}^{\gamma}u (2a+2\gamma b)(1+x^{2})^{a+\gamma b-1}2x\right\rangle_{L^{2}},
\end{align*}
where $C_{k}^{n}\coloneqq \frac{n!}{k!(n-k)!}$. Since $\partial_{x}^{k}f=0$ for $k\geq 3$, only 3 terms in each sum over $k$ are non-zero.
\begin{align*}
        &I=-\sum_{\gamma = 0}^{m}\left\langle C_{0}^{\gamma}f(x)\partial_{x}^{\gamma+1} u (1+x^{2})^{a+\gamma b},\partial_{x}^{\gamma+1}u(1+x^{2})^{a+\gamma b}\right\rangle_{L^{2}}\\
    &\phantom{xx}-\sum_{\gamma = 0}^{m}\left\langle C_{0}^{\gamma} f(x)\partial_{x}^{\gamma+1}\ u(1+x^{2})^{a+\gamma b},\partial_{x}^{\gamma}u(2a+2\gamma b)(1+x^{2})^{a+\gamma b-1}2x\right\rangle_{L^{2}}\\
    %%%% D
       &-\sum_{\gamma = 0}^{m}\left\langle C_{1}^{\gamma} \partial_{x}f(x)\partial_{x}^{\gamma} u (1+x^{2})^{a+\gamma b},\partial_{x}^{\gamma+1} u(1+x^{2})^{a+\gamma b}\right\rangle_{L^{2}}\\
    &\phantom{xx}-\sum_{\gamma = 0}^{m}\left\langle C_{1}^{\gamma} \partial_{x}f(x)\partial_{x}^{\gamma} u (1+x^{2})^{a+\gamma b},\partial_{x}^{\gamma}u (2a+2\gamma b)(1+x^{2})^{a+\gamma b-1}2x\right\rangle_{L^{2}}\\
    %%%% \partial_{x}^2
    &-\sum_{\gamma = 0}^{m}\left\langle C_{2}^{\gamma} \partial_{x}^{2}f(x)\partial_{x}^{\gamma-1}u (1+x^{2})^{a+\gamma b},\partial_{x}^{\gamma+1}u(1+x^{2})^{a+\gamma b}\right\rangle_{L^{2}}\\
    &\phantom{xx}-\sum_{\gamma = 0}^{m}\left\langle C_{2}^{\gamma} \partial_{x}^{2}f(x)\partial_{x}^{\gamma-1} u (1+x^{2})^{a+\gamma b},\partial_{x}^{\gamma}u (2a+2\gamma b)(1+x^{2})^{a+\gamma b-1}2x\right\rangle_{L^{2}}\\
     &=\underbrace{I^{1}_{1}+I^{1}_{2}}_{I^{1}}+\underbrace{I^{2}_{1}+I^{2}_{2}}_{I^{2}}+\underbrace{I^{3}_{1}+I^{3}_{2}}_{I^{3}}.
\end{align*}
We consider each term separately,
\begin{align*}
     &I^{1}_{1}=-\sum_{\gamma = 0}^{m}\left\langle C_{0}^{\gamma}f(x)\partial_{x}^{\gamma+1} u (1+x^{2})^{a+\gamma b},\partial_{x}^{\gamma+1}u(1+x^{2})^{a+\gamma b}\right\rangle_{L^{2}}\\
    &=-\sum_{\gamma = 0}^{m}\int_{\R_{+}}C_{0}^{\gamma}f(x)|\partial_{x}^{\gamma+1} u|^{2} (1+x^{2})^{2a+2\gamma b}\dx.
\end{align*}
\begin{align*}
     &I^{1}_{2}=-\sum_{\gamma = 0}^{m}\left\langle C_{0}^{\gamma} f(x)\partial_{x}^{\gamma+1} u(1+x^{2})^{a+\gamma b},\partial_{x}^{\gamma}u(2a+2\gamma b)(1+x^{2})^{a+\gamma b-1}2x\right\rangle_{L^{2}}\\
    &=-\sum_{\gamma = 0}^{m}C_{0}^{\gamma}\int_{\R_{+}}f(x)\partial_{x}^{\gamma+1}u\partial_{x}^{\gamma}u (2a+2\gamma b)(1+x^{2})^{a+\gamma b-1}2x\dx\\
    &=-\sum_{\gamma = 0}^{m}C_{0}^{\gamma}\frac{1}{2}\int_{\R_{+}}f(x)\partial_{x}\left(\left(\partial_{x}^{\gamma} u\right)^{2}\right) (2a+2\gamma b)(1+x^{2})^{2a+2\gamma b-1}2x\dx\\
    &=\sum_{\gamma = 0}^{m}C_{0}^{\gamma}\frac{1}{2}\int_{\R_{+}}\left(\partial_{x}^{\gamma} u\right)^{2} (2a+2\gamma b)\partial_{x}\left(f(x)2x(1+x^{2})^{2a+2\gamma b-1}\right)\dx\\
    &=\sum_{\gamma = 0}^{m}C_{0}^{\gamma}\frac{1}{2}\int_{\R_{+}}\left(\partial_{x}^{\gamma} u\right)^{2} (2a+2\gamma b) \left(2x(1+x^{2})^{2a+2\gamma b-1}\partial_{x} f(x)+f(x)(1+x^{2})^{2a+2\gamma b-2}2(1+(4a+4\gamma b-1)x^{2})\right)\dx\\
    &=\sum_{\gamma = 0}^{m}C_{0}^{\gamma}\frac{2}{2}(2a+2\gamma b)\int_{\R_{+}}\left(\partial_{x}^{\gamma} u\right)^{2}  x(1+x^{2})^{2a+2\gamma b-1}\partial_{x} f(x))\dx\\
    &\phantom{xx}+\sum_{\gamma = 0}^{m}C_{0}^{\gamma}\frac{2}{2}(2a+2\gamma b)\int_{\R_{+}}\left(\partial_{x}^{\gamma} u\right)^{2}  f(x)(1+x^{2})^{2a+2\gamma b-2}(1+(4a+4\gamma b-1)x^{2})\dx\\
\end{align*}
The other terms will be estimated in a similar fashion. 
\begin{align*}
       I^{2}_{1}=&-\sum_{\gamma = 0}^{m}\left\langle C_{1}^{\gamma} \partial_{x}f(x)\partial_{x}^{\gamma} u (1+x^{2})^{a+\gamma b},\partial_{x}^{\gamma+1}u(1+x^{2})^{a+\gamma b}\right\rangle_{L^{2}}\\
    &= -\sum_{\gamma = 0}^{m}\int_{\R_{+}}C_{1}^{\gamma}\partial_{x}f(x)\partial_{x}^{\gamma}u \partial_{x}^{\gamma+1} u  (1+x^{2})^{2a+2\gamma b}\dx\\
    &= \sum_{\gamma = 0}^{m}\frac{1}{2}\int_{\R_{+}}C_{1}^{\gamma}(\partial_{x}^{\gamma} u)^{2}\left((1+x^{2})^{2a+2\gamma b} \partial_{x}^{2}f(x)  +\partial_{x}f(x) (2a+2\gamma b)(1+x^{2})^{2a+2\gamma b-1}2x\right)\dx\\
    &= \sum_{\gamma = 0}^{m}\frac{1}{2}\int_{\R_{+}}C_{1}^{\gamma}(\partial_{x}^{\gamma} u)^{2}(1+x^{2})^{2a+2\gamma b}\partial_{x}^{2}f(x)\dx\\
&\phantom{xx}+ \sum_{\gamma = 0}^{m}\frac{2}{2}(2a+2\gamma b)\int_{\R_{+}}C_{1}^{\gamma}(\partial_{x}^{\gamma} u)^{2}\left(\partial_{x}f(x) (1+x^{2})^{2a+2\gamma b-1}x\right)\dx.
\end{align*}
\begin{align*}
       I^{2}_{2}=&-\sum_{\gamma = 0}^{m}\left\langle C_{1}^{\gamma} \partial_{x}f(x)\partial_{x}^{\gamma} u (1+x^{2})^{a+\gamma b},\partial_{x}^{\gamma} u (2a+2\gamma b)(1+x^{2})^{a+\gamma b-1}2x\right\rangle_{L^{2}}\\
    &=-\sum_{\gamma = 0}^{m}2(2a+2\gamma b)\int_{\R_{+}}C_{1}^{\gamma} \partial_{x}f(x)\left(\partial_{x}^{\gamma} u\right)^{2} (1+x^{2})^{2a+2\gamma b-1}x\dx.
\end{align*}
The third term, after integrating by parts and recalling that $\partial_{x}^{k}f(x)=0$ for $k\geq 3$, can be estimated in a similar way.
\begin{align*}
     I^{3}_{1}=&-\sum_{\gamma = 0}^{m}\left\langle C_{2}^{\gamma} \partial_{x}^{2}f(x) \partial_{x}^{\gamma-1} u (1+x^{2})^{a+\gamma b},\partial_{x}^{\gamma+1}u(1+x^{2})^{a+\gamma b}\right\rangle_{L^{2}}\\
    &= -\sum_{\gamma = 0}^{m}\int_{\R_{+}}C_{2}^{\gamma}\partial_{x}^{2}f(x)\partial_{x}^{\gamma-1}u \partial_{x}^{\gamma+1} u  (1+x^{2})^{2a+2\gamma b}\dx\\
    &= \sum_{\gamma = 0}^{m}C_{2}^{\gamma}\int_{\R_{+}}\partial_{x}^{2}f(x)\partial_{x}^{\gamma}u \partial_{x}^{\gamma} u  (1+x^{2})^{2a+2\gamma b}\dx+ \sum_{\gamma = 0}^{m}C_{2}^{\gamma}\int_{\R_{+}}\partial_{x}^{2}f(x)\partial_{x}^{\gamma-1}u \partial_{x}^{\gamma} u  (2a+2\gamma b)(1+x^{2})^{2a+2\gamma b-1}2x\dx\\
    &= \sum_{\gamma = 0}^{m}C_{2}^{\gamma}\int_{\R_{+}}\partial_{x}^{2}f(x)|\partial_{x}^{\gamma}u|^{2} (1+x^{2})^{2a+2\gamma b}\dx+ \sum_{\gamma = 0}^{m}C_{2}^{\gamma}\frac{1}{2}\int_{\R_{+}}\partial_{x}^{2}f(x)\partial_{x}\left(\left(\partial_{x}^{\gamma-1}u\right)^{2}\right)  (2a+2\gamma b)(1+x^{2})^{2a+2\gamma b-1}2x\dx\\
    &= \sum_{\gamma = 0}^{m}C_{2}^{\gamma}\int_{\R_{+}}\partial_{x}^{2}f(x)|\partial_{x}^{\gamma}u|^{2} (1+x^{2})^{2a+2\gamma b}\dx\\
    &\phantom{xx}- \sum_{\gamma = 0}^{m}C_{2}^{\gamma}\frac{2}{2}(2a+2\gamma b)\int_{\R_{+}}\partial_{x}^{2}f(x)\left(\partial_{x}^{\gamma-1}u\right)^{2}  (1+x^{2})^{2a+2\gamma b-2}(1+(4a+4\gamma b-1)x^{2})\dx\\
\end{align*}
\begin{align*}
     I^{3}_{2}=&-\sum_{\gamma = 0}^{m}\left\langle C_{2}^{\gamma} \partial_{x}^{2}f(x)\partial_{x}^{\gamma-1} u (1+x^{2})^{a+\gamma b},\partial_{x}^{\gamma} u(2a+2\gamma b)(1+x^{2})^{a+\gamma b-1}2x\right\rangle_{L^{2}}\\
    &=-\sum_{\gamma = 0}^{m}\int_{\R_{+}}C_{2}^{\gamma} \partial_{x}^{2}f(x)\frac{1}{2}\partial_{x}\left(\partial_{x}^{\gamma-1} u\right)^{2} (2a+2\gamma b) (1+x^{2})^{2a+2\gamma b-1}2x\dx\\
    &=-\sum_{\gamma = 0}^{m}C_{2}^{\gamma}\frac{2}{2}(2a+2\gamma b)\int_{\R_{+}} \partial_{x}^{2}f(x)\left(\partial_{x}^{\gamma-1} u\right)^{2}  (1+x^{2})^{2a+2\gamma b-2}(1+(4a+4\gamma b-1)x^{2})\dx
\end{align*}
For the second term,
\begin{align*}
    II&=\sum_{\gamma = 0}^{m}\langle \partial_{x}^{\gamma}(g(x)\partial_{x} u)(1+x^{2})^{a+\gamma b},\partial_{x}^{\gamma}u(1+x^{2})^{a+\gamma b}\big)\rangle_{L^{2}}\\
    &=\sum_{\gamma = 0}^{m}\left\langle \left(\sum_{k=0}^{\gamma}C_{k}^{\gamma}\chi_{k \leq 1}\partial_{x}^{k }g(x)\partial_{x}^{\gamma -k}\partial_{x} u\right) (1+x^{2})^{a+\gamma b},\partial_{x}^{\gamma}u(1+x^{2})^{a+\gamma b}\right\rangle_{L^{2}},
\end{align*}
where we used again the notation $C_{k}^{n}\coloneqq \frac{n!}{k!(n-k)!}$. As before, we only consider the terms for which$ \partial_{x}^{k}g\neq 0$.
\begin{align*}
    II&=\sum_{\gamma = 0}^{m}\int_{\R_{+}}C_{0}^{\gamma}g(x) \partial_{x}^{\gamma+1} u \partial_{x}^{\gamma}u(1+x^{2})^{2a+2\gamma b}\dx+\sum_{\gamma = 0}^{m}\int_{\R_{+}}C_{1}^{\gamma}\partial_{x}g(x) \partial_{x}^{\gamma}u \partial_{x}^{\gamma}u(1+x^{2})^{2a+2\gamma b}\dx\\
    &=\sum_{\gamma = 0}^{m}C_{0}^{\gamma}\frac{1}{2}\int_{\R_{+}}g(x) \partial_{x}\left(\left(\partial_{x}^{\gamma} u\right)^{2}\right) (1+x^{2})^{2a+2\gamma b}\dx+\sum_{\gamma = 0}^{m}C_{1}^{\gamma}\int_{\R_{+}}\partial_{x}g(x) \partial_{x}^{\gamma} u \partial_{x}^{\gamma}u(1+x^{2})^{2a+2\gamma b}\dx\\
    &=-\sum_{\gamma = 0}^{m}C_{0}^{\gamma} \frac{1}{2}\int_{\R_{+}}\left(\partial_{x}^{\gamma} u\right)^{2} \left((1+x^{2})^{2a+2\gamma b}\partial_{x}g(x)+g(x)(2a+2\gamma b)(1+x^{2})^{2a+2\gamma b-1}2x\right)\dx\\
    &\phantom{xx}+\sum_{\gamma = 0}^{m}C_{1}^{\gamma}\int_{\R_{+}}\partial_{x}g(x) \partial_{x}^{\gamma} u \partial_{x}^{\gamma}u(1+x^{2})^{2a+2\gamma b}\dx\\
    &=-\sum_{\gamma = 0}^{m}\int_{\R_{+}}C_{0}^{\gamma}\frac{1}{2} \partial_{x}g(x)\left(\partial_{x}^{\gamma}u\right)^{2} (1+x^{2})^{2a+2\gamma b}\dx-\sum_{\gamma = 0}^{m}C_{0}^{\gamma}\frac{2}{2}(2a+2\gamma b)\int_{\R_{+}} \left(\partial_{x}^{\gamma}u\right)^{2} g(x)(1+x^{2})^{2a+2\gamma b-1}x\dx\\
    &\phantom{xx}+\sum_{\gamma = 0}^{m}C_{1}^{\gamma}\int_{\R_{+}} \partial_{x}g(x)|\partial_{x}^{\gamma}u|^{2} (1+x^{2})^{2a+2\gamma b}\dx\\
\end{align*}
The last term is already of a desired form, since it corresponds to $\gamma \|u\|^{2}_{H_{1}}$.
Collecting all terms with the same powers of $D$ and $x$, observing that $C_{0}^{\gamma}=1$, $C_{1}^{\gamma}=\gamma\chi_{\gamma\geq 1}$, $C_{2}^{\gamma}=\frac{\gamma(\gamma-1)}{2}\chi_{\gamma\geq 2}$ and recalling the notation $ \V(t,\omega)=  m_{\lambda}\exp{((\beta-1/2)t+W^{0}_{t})}$, yields
\begin{align*}
    \langle Lu,u\rangle_{H_{1}}&=-\sum_{\gamma = 0}^{m}\int_{\R_{+}}f(x)|\partial_{x}^{\gamma+1} u|^{2} (1+x^{2})^{2a+2\gamma b}\dx\\
    &\phantom{xx}+\sum_{\gamma = 0}^{m}\int_{\R_{+}}|\partial_{x}^{\gamma} u|^{2} (1+x^{2})^{2a+2\gamma b}h_{1}(\gamma,x)\dx\\
    &\phantom{xx}+\sum_{\gamma = 0}^{m}\int_{\R_{+}}|\partial_{x}^{\gamma} u|^{2} x(1+x^{2})^{2a+2\gamma b-1}h_{2}(\gamma,x)\dx\\
    &\phantom{xx}+\sum_{\gamma = 0}^{m}\int_{\R_{+}}|\partial_{x}^{\gamma} u|^{2} (1+x^{2})^{2a+2\gamma b-2}(1+(4a+4\gamma b-1)x^{2})h_{3}(\gamma,x)\dx\\
    &\phantom{xx}+\sum_{\gamma = 0}^{m}\int_{\R_{+}}|\partial_{x}^{\gamma-1} u|^{2} (1+x^{2})^{2a+2\gamma b-2}(1+(4a+4\gamma b-1)x^{2})h_{4}(\gamma,x)\dx,
\end{align*}
where 
\begin{align*}
    h_{1}(\gamma,x)&=\left(\frac{1}{2}\gamma \chi_{\gamma\geq 1} \partial_{x}^{2}f(x)+\gamma \chi_{\gamma\geq 1}\partial_{x}g(x)+\frac{\gamma(\gamma-1)}{2}\chi_{\gamma\geq 2}\partial_{x}^{2}f(x)-\frac{1}{2}\partial_{x}g(x)\right)\\
    &=\left(\frac{1}{2}\gamma \chi_{\gamma\geq 1} -\gamma \chi_{\gamma\geq 1}+\frac{\gamma(\gamma-1)}{2}\chi_{\gamma\geq 2}+\frac{1}{2}\right)\\
    h_{2}(\gamma,x)&=(2a+2\gamma b)\left(\partial_{x}f(x)-\gamma \chi_{\gamma\geq 1}\partial_{x}f(x)-g(x)\right)\\
    &=(2a+2\gamma b)\left(\V(t,\omega)\frac{\alpha}{2}+x-\gamma \chi_{\gamma\geq 1}\V(t,\omega)\frac{\alpha}{2}-\gamma \chi_{\gamma\geq 1}x+\left(\left(\frac{\alpha}{2}- \beta\right) \V(t,\omega)+x\right)\right)\\
    &=(2a+2\gamma b)\left(\V(t,\omega)\left(\frac{\alpha}{2}(2-\gamma \chi_{\gamma\geq 1})-\beta\right)+x(2-\gamma \chi_{\gamma\geq 1})\right)\\
    h_{3}(\gamma,x)&=(2a+2\gamma b)f(x)=(2a+2\gamma b)\left(\V(t,\omega)\frac{\alpha}{2}x+\frac{x^{2}}{2}\right)\\
    h_{4}(\gamma,x)&=-2\frac{\gamma(\gamma-1)}{2}\chi_{\gamma \geq 2}(2a+2\gamma b)\partial_{x}^{2}f(x)=-2\frac{\gamma(\gamma-1)}{2}\chi_{\gamma \geq 2}(2a+2\gamma b)\partial_{x}^{2}.
\end{align*}
Using Young's inequality when necessary,
\begin{align*}
    &h_{1}(\gamma,x)\leq m\\
    &x(1+x^{2})^{2a+2\gamma b-1}h_{2}(\gamma,x)=\frac{x}{(1+x^{2})}(1+x^{2})^{2a+2\gamma b}h_{2}(\gamma,x)\leq C_{\alpha,\beta,a,b,m}\V(t,\omega)(1+x^{2})^{2a+2\gamma b}\\
    &(1+x^{2})^{2a+2\gamma b-2}(1+(4a+4\gamma b-1)x^{2})h_{3}(\gamma,x)\leq C_{\alpha,\beta,a,b,m}\V(t,\omega)(1+x^{2})^{2a+2\gamma b}\\
    &(1+x^{2})^{2a+2\gamma b-2}(1+(4a+4\gamma b-1)x^{2})h_{4}(\gamma,x)= (1+x^{2})^{2a+2\gamma b}\frac{(1+(4a+4\gamma b-1)x^{2})}{(1+x^{2})^{2}}h_{4}(\gamma,x)\leq 0.
\end{align*}
Hence,
\begin{align*}
    \langle Lu,u\rangle_{H_{1}}&\leq -\sum_{\gamma = 0}^{m}\int_{\R_{+}}\V(t,\omega)\frac{\alpha}{2}x|\partial_{x}^{\gamma+1} u|^{2} (1+x^{2})^{2a+2\gamma b}\dx\\
&\phantom{xx} -\sum_{\gamma = 0}^{m}\int_{\R_{+}}\frac{x^{2}}{2}|\partial_{x}^{\gamma+1} u|^{2} (1+x^{2})^{2a+2\gamma b}\dx\\
&\phantom{xx} +C_{\alpha,\beta,a,b,m}\V(t,\omega)\|u\|^{2}_{H_{1}}.
\end{align*}
It remains to bound the remaining term in \eqref{eqn:weak_monotonicity_L_sigma}.
We will begin with the term $\langle u,\sigma \varphi\rangle_{H_{1}}$ without the square and supremum, but assume that $\|\varphi\|_{H_{1}}=1$.
\begin{align*}
    &\left|\langle u,\sigma \varphi\rangle_{H_{1}}\right|=\left|\sum_{\gamma = 0}^{m} \int_{\R_{+}}\partial_{x}^{\gamma}u (\partial_{x}^{\gamma}\varphi - \partial_{x}^{\gamma}\partial_{x}(x \varphi))(1+x^{2})^{2a+2\gamma b}\dx\right|\\
         &=\left|\sum_{\gamma = 0}^{m} \int_{\R_{+}}\partial_{x}^{\gamma}u \partial_{x}^{\gamma}\varphi(1+x^{2})^{2a+2\gamma b}\dx- \sum_{\gamma = 0}^{m} \int_{\R_{+}}\partial_{x}^{\gamma}u \partial_{x}^{\gamma+1}(x \varphi))(1+x^{2})^{2a+2\gamma b}\dx \right|\\
        &=\left|\sum_{\gamma = 0}^{m} \int_{\R_{+}}\partial_{x}^{\gamma}u \partial_{x}^{\gamma}\varphi(1+x^{2})^{2a+2\gamma b}\dx- \sum_{\gamma = 0}^{m} \int_{\R_{+}}\sum_{k=0}^{|\gamma|+1}C_{k}^{|\gamma|+1} \partial_{x}^{\gamma}u \partial_{x}^{k}x\partial_{x}^{\gamma+1-k} \varphi(1+x^{2})^{2a+2\gamma b}\dx\right|\\
                &=\left|\sum_{\gamma = 0}^{m} \int_{\R_{+}}\partial_{x}^{\gamma}u \partial_{x}^{\gamma}\varphi(1+x^{2})^{2a+2\gamma b}\dx- \sum_{\gamma = 0}^{m} \int_{\R_{+}}C_{0}^{|\gamma|+1} \partial_{x}^{\gamma}u x\partial_{x}^{\gamma+1} \varphi(1+x^{2})^{2a+2\gamma b}\dx\right.\\
                &\phantom{xx}\left.- \sum_{\gamma = 0}^{m} \int_{\R_{+}}C_{1}^{|\gamma|+1} \partial_{x}^{\gamma}u \partial_{x}^{\gamma} \varphi(1+x^{2})^{2a+2\gamma b}\dx\right|\\
                        &=\left|\sum_{\gamma = 0}^{m} \int_{\R_{+}}\partial_{x}^{\gamma}u \partial_{x}^{\gamma}\varphi(1+x^{2})^{2a+2\gamma b}\dx+\sum_{\gamma = 0}^{m} \int_{\R_{+}}C_{0}^{|\gamma|+1} \partial_{x}^{\gamma+1}u x\partial_{x}^{\gamma} \varphi(1+x^{2})^{2a+2\gamma b}\dx\right.\\
                        &\phantom{xx}+\sum_{\gamma = 0}^{m} \int_{\R_{+}}C_{0}^{|\gamma|+1} \partial_{x}^{\gamma}u \partial_{x}^{\gamma} \varphi(1+x^{2})^{2a+2\gamma b}\dx+\sum_{\gamma = 0}^{m} \int_{\R_{+}}C_{0}^{|\gamma|+1} \partial_{x}^{\gamma}u \partial_{x}^{\gamma} \varphi 2x^{2}(2a+2\gamma b)(1+x^{2})^{2a+2\gamma b-1}\dx\\  
                        &\phantom{xx}\left. - \sum_{\gamma = 0}^{m} \int_{\R_{+}}C_{1}^{|\gamma|+1} \partial_{x}^{\gamma}u \partial_{x}^{\gamma} \varphi(1+x^{2})^{2a+2\gamma b}\dx\right|\\
         &=\left|\sum_{\gamma = 0}^{m} \int_{\R_{+}}\left(\partial_{x}^{\gamma}u -C_{1}^{|\gamma|+1} \partial_{x}^{\gamma}u+C_{0}^{|\gamma|+1}\left( x \partial_{x}^{\gamma+1}u+  \partial_{x}^{\gamma}u+\frac{2x^{2}(2a+2\gamma b)}{(1+x^{2})}\partial_{x}^{\gamma}u\right)\right)\partial_{x}^{\gamma}\varphi(1+x^{2})^{2a+2\gamma b}\dx\right|\\
     &\leq  \sqrt{\sum_{\gamma = 0}^{m}\int_{\R_{+}}\left(\left(2-C_{1}^{|\gamma|+1} +\frac{2x^{2}(2a+2\gamma b)}{(1+x^{2})}\right)\partial_{x}^{\gamma}u + x \partial_{x}^{\gamma+1}u\right)^{2} (1+x^{2})^{2a+2\gamma b}\dx}\sqrt{\sum_{\gamma = 0}^{m}\int_{\R_{+}}\left(\partial_{x}^{\gamma}\varphi\right)^{2} (1+x^{2})^{2a+2\gamma b}}\dx\\
     &= \sqrt{ \int_{\R_{+}}\sum_{\gamma = 0}^{m}\left(\left(2-C_{1}^{|\gamma|+1} +\frac{2x^{2}(2a+2\gamma b)}{(1+x^{2})}\right)\partial_{x}^{\gamma}u + x \partial_{x}^{\gamma+1}u\right)^{2}(1+x^{2})^{2a+2\gamma b}\dx}\|\varphi\|_{H_{1}}\\
     &\leq\sqrt{ \int_{\R_{+}}\sum_{\gamma = 0}^{m}\left(\left(2-C_{1}^{|\gamma|+1} +\frac{2x^{2}(2a+2\gamma b)}{(1+x^{2})}\right)\partial_{x}^{\gamma}u + x \partial_{x}^{\gamma+1}u\right)^{2} (1+x^{2})^{2a+2\gamma b}\dx}.
\end{align*}
We expand the square inside the integral,
\begin{align*}
    &\left(\left(2-C_{1}^{|\gamma|+1} +\frac{2x^{2}(2a+2\gamma b)}{(1+x^{2})}\right)\partial_{x}^{\gamma}u + x \partial_{x}^{\gamma+1}u\right)^{2} \\
    &=\left(\left(2-C_{1}^{|\gamma|+1} +\frac{2x^{2}(2a+2\gamma b)}{(1+x^{2})}\right)^{2}(\partial_{x}^{\gamma}u)^{2} + x^2 (\partial_{x}^{\gamma+1}u)^{2} +2\left(2-C_{1}^{|\gamma|+1} +\frac{2x^{2}(2a+2\gamma b)}{(1+x^{2})}\right)x \partial_{x}^{\gamma}u\partial_{x}^{\gamma+1}u\right).
\end{align*}
Squaring the whole integral implies
\begin{align*}
    \sup_{\varphi\colon \|\varphi\|_{H_{1}}=1}\langle u,\sigma \varphi\rangle_{H_{1}}^{2}&\leq
    \int_{\R_{+}}\sum_{\gamma = 0}^{m}\left(\left(2-C_{1}^{|\gamma|+1} +\frac{2x^{2}(2a+2\gamma b)}{(1+x^{2})}\right)\partial_{x}^{\gamma}u + x \partial_{x}^{\gamma+1}u\right)^{2} (1+x^{2})^{2a+2\gamma b}\dx\\
    &=\int_{\R_{+}}\sum_{\gamma = 0}^{m}\left(2-C_{1}^{|\gamma|+1} +\frac{2x^{2}(2a+2\gamma b)}{(1+x^{2})}\right)^{2}(\partial_{x}^{\gamma}u)^{2} (1+x^{2})^{2a+2\gamma b}\dx\\
    &\phantom{xx}+\int_{\R_{+}}\sum_{\gamma = 0}^{m} x^2 (\partial_{x}^{\gamma+1}u)^{2} (1+x^{2})^{2a+2\gamma b}\dx\\
      &\phantom{xx}+\int_{\R_{+}}\sum_{\gamma = 0}^{m}2\frac{2x^{2}(2a+2\gamma b)}{(1+x^{2})}x \partial_{x}^{\gamma}u\partial_{x}^{\gamma+1}u (1+x^{2})^{2a+2\gamma b}\dx.
\end{align*}
The last term is the most troublesome, which is why we analyse it separately,
\begin{align*}
   &\int_{\R_{+}}\sum_{\gamma = 0}^{m}2\frac{2x^{2}(2a+2\gamma b)}{(1+x^{2})}x \partial_{x}^{\gamma}u\partial_{x}^{\gamma+1}u (1+x^{2})^{2a+2\gamma b}\dx=\int_{\R_{+}}\sum_{\gamma = 0}^{m}2\frac{2x^{2}(2a+2\gamma b)}{(1+x^{2})}x \frac{\partial_{x} \left(\partial_{x}^{\gamma}u\right)^{2}}{2} (1+x^{2})^{2a+2\gamma b}\dx\\
   &=-\int_{\R_{+}}\sum_{\gamma = 0}^{m} \left(\partial_{x}^{\gamma}u\right)^{2}2(2a+2\gamma b)\partial_{x}\left(\frac{x^{3}}{(1+x^{2})}x(1+x^{2})^{2a+2\gamma b}\right)\dx\\
   &=-\int_{\R_{+}}\sum_{\gamma = 0}^{m} \left(\partial_{x}^{\gamma}u\right)^{2}2(2a+2\gamma b)\frac{x^{2}\left(3+(4a+4\gamma b+1)x^{2}\right)}{(1+x^2)^{2}}(1+x^2)^{2a+2\gamma b}\dx\leq 0.
\end{align*}
Hence,
\begin{align*}
    \sup_{\varphi\in H_{1}\colon \|\varphi\|_{H_{1}}=1}|\langle u,\sigma\varphi\rangle_{H_{1}}|^{2}&\leq \sum_{\gamma = 0}^{m} \int_{\R_{+}} x^{2}\left(\partial_{x}^{\gamma+1}u\right)^{2}(1+x^{2})^{2a+2\gamma b}\dx\\
    &\phantom{xx}+ \sum_{\gamma = 0}^{m} \int_{\R_{+}} C_{a,b,\gamma}(x)\left(\partial_{x}^{\gamma}u\right)^{2}(1+x^{2})^{2a+2\gamma b}\dx\\
    &\leq \sum_{\gamma = 0}^{m} \int_{\R_{+}} x^{2}\left(\partial_{x}^{\gamma+1}u\right)^{2}(1+x^{2})^{2a+2\gamma b}\dx\\
    &\phantom{xx}+ C_{a,b,m}\|u\|^{2}_{H_{1}},
\end{align*}
where
\begin{align*}
    C_{\alpha,\beta,\gamma}(x)=\left(2-C_{1}^{|\gamma|+1} +\frac{2x^{2}(2a+2\gamma b)}{(1+x^{2})}\right)^{2}-2(2a+2\gamma b)\frac{x^{2}\left(3+(4a+4\gamma b+1)x^{2}\right)}{(1+x^2)^{2}}.
\end{align*}
Combining the estimates for $L$ and $\sigma$ yields
\begin{align*}
     \langle Lu,u\rangle_{H_{1}}-\frac{1}{2} \|\sigma(u)\|^{2}_{H_{1}}&\leq -\sum_{\gamma = 0}^{m}\int_{\R_{+}}\V(t,\omega)\frac{\alpha}{2}x|D^{\gamma+1} u|^{2} (1+x^{2})^{2a+2\gamma b}\dx\\
&\phantom{xx} +C_{\alpha,\beta,a,b,m}(1+\V(t,\omega))\|u\|^{2}_{H_{1}}.
\end{align*}
Since the first term in non positive, we obtain
\begin{align*}
     \langle Lu,u\rangle_{H_{1}}-\frac{1}{2} \|\sigma(u)\|^{2}_{H_{1}}&\leq C_{\alpha,\beta,a,b,m}(1+\V(t,\omega))\|u\|^{2}_{H_{1}}.
\end{align*}
\end{proof}

\subsection{Appendix C: SDE estimates}
We here provide several estimates used throughout the proofs.

\begin{lemma}\label{lem:EsupZp_EZp_estimates}
Fix $N \in \mathbb{N}$ and consider
$Z_{(N)}\coloneqq \frac{1}{N}\sum_{i=1}^{N}X_{i}$, where $(X_{1},\dots,X_{N})$ is a (non-negative) solution of the system \eqref{eqn:X_Ni_SDE}. Take $p\geq 1$ and assume $\E Z_{(N)}(0)^{p}\leq C$. Then the following estimate hold,
\begin{align}
    \E Z_{(N)}(t)^{p}\leq \E Z_{(N)}(0)^{p}\exp{\bigg(\bigg(p\beta+\frac{p(p-1)}{2}\bigg)t\bigg)} .\label{eqn:EZp_estimate}
\end{align}
\begin{align}
    \E \sup_{t\in [0,T]} Z_{(N)}(t)^{p}\leq C_{\E Z_{(N)}(0)^{p},T,\beta,p}.\label{eqn:EsupZp_estimate}
\end{align}
Moreover, if \ref{A:A_3_X(0)_Z(0)} holds for some $\zeta >0$, then 
\[
\E Z_{(N)}(0)^{1+\zeta} \leq C
\]
for some constant $C$.

\end{lemma}
\begin{proof}
It{\^o}'s formula and taking expectation results in
\begin{align*}
    \E Z_{(N)}(t)^{p}&=\E Z_{(N)}(0)^{p}+p\int_{0}^{t} \beta \E Z_{(N)}(s)^{p}\ds+\frac{p(p-1)}{2}\frac{\alpha}{N}\int_{0}^{t}\E Z_{(N)}(s)^{p}\ds+\frac{p(p-1)}{2}\bigg(1-\frac{\alpha}{N}\bigg)\int_{0}^{t}\E Z_{(N)}(s)^{p}\ds.
\end{align*}
By Gronwall's inequality, we obtain
\begin{align*}
    \E Z_{(N)}(t)^{p}\leq \big(\E Z_{(N)}(0)^{p}\big)\exp{\bigg(\bigg(p\beta+\frac{p(p-1)}{2}\bigg)t\bigg)} \leq C_{\E Z_{(N)}(0)^{p},t,\alpha,\beta,p}.
\end{align*}
For the proof of \eqref{eqn:EsupZp_estimate} we note that, by L{\'e}vy's representation theorem, $B_{(N)}(t)\coloneqq \sum_{i=1}^{N}\frac{\sqrt{X_{i}(t)}}{\sqrt{Z_{(N)}(t)}}W^{i}(t)$ is a Brownian motion. Singling out the stochastic integrals, we apply the BDG inequality and Young's inequality to obtain
\begin{align*}
    p\E\sup_{t\in [0,T]}& \frac{1}{N}\sum_{i=1}^{N}\int_{0}^{t} \sqrt{\alpha} Z_{(N)}(s)^{p}\frac{\sqrt{X_{i}(s)}}{\sqrt{Z_{(N)}(s)}}\dW^{i}_{s}\leq C_{p}\E\bigg(\int_{0}^{T} \frac{\alpha}{N^{2}} Z_{(N)}(s)^{2p}\ds\bigg)^{1/2}\\
    &\leq C_{p}\sqrt{\frac{\alpha}{N^{2}}}\E\bigg(\int_{0}^{T} \sup_{t\in [0,T]}|Z_{(N)}(t)^{p}| Z_{(N)}(s)^{p}\ds\bigg)^{1/2}\leq C_{p}\sqrt{\frac{\alpha}{N^{2}}}\E\bigg(\sqrt{\sup_{t\in [0,T]}|Z_{(N)}(t)^{p}|}\bigg(\int_{0}^{T}  Z_{(N)}(s)^{p}\ds\bigg)^{1/2}\bigg)\\
    &= C_{p}\sqrt{\frac{\alpha}{N^{2}}}\E\bigg(\sqrt{\sup_{t\in [0,T]}|Z_{(N)}(t)^{p}|}\sqrt{\frac{\widetilde{\eps}}{\widetilde{\eps}}}\bigg(\int_{0}^{T}  Z_{(N)}(s)^{p}\ds\bigg)^{1/2}\bigg)\\
    &\leq C_{p}\sqrt{\frac{\alpha}{N^{2}}}\E\bigg(\widetilde{\eps}\sup_{t\in [0,T]}|Z_{(N)}(t)^{p}|+\frac{1}{\widetilde{\eps}}\int_{0}^{T}  Z_{(N)}(s)^{p}\ds\bigg),
\end{align*}
where we introduced a small constant $\widetilde{\eps}$, which will be chosen later, and used Young's inequality in the last step.
The estimation procedure for the second stochastic integral is similar. We again introduce the small constant $\widetilde{\eps}$ which will be used to absorb terms including the supremum over time into the left-hand side, given by $\E\sup_{t\in [0,T]}Z_{(N)}(t)^{p}$.
\begin{align*}
   p\E\sup_{t\in [0,T]}& \frac{1}{N}\sum_{i=1}^{N}\int_{0}^{t} Z_{(N)}(s)^{p}\sqrt{1-\frac{\alpha}{N}}\dW^{0}_{s}\leq \widetilde{C}_{p}\E\bigg(\int_{0}^{T} \bigg(1-\frac{\alpha}{N}\bigg) Z_{(N)}(s)^{2p}\ds\bigg)^{1/2}\\
   &\leq \widetilde{C}_{p}\sqrt{ 1-\frac{\alpha}{N}}\E\bigg(\widetilde{\eps}\sup_{t\in [0,T]}|Z_{(N)}(t)^{p}|+\frac{1}{\widetilde{\eps}}\int_{0}^{T}  Z_{(N)}(s)^{p}\ds\bigg).
\end{align*}
Combining these and the previous estimates, we obtain 
\begin{align*}
   (1-\widetilde{\eps}(\widetilde{C}_{p,\alpha}+C_{p,\alpha})) \E\sup_{t\in [0,T]} Z_{(N)}(t)^{p}&\leq \E Z_{(N)}(0)^{p}+\bigg(p\beta + \frac{p(p-1)}{2N}+ \frac{p(p-1)}{2}\bigg(1-\frac{\alpha}{N}\bigg) \bigg)\E \int_{0}^{T}  Z_{(N)}(s)^{p}\ds\\
    &\phantom{xx}{}+C_{p,\widetilde{\eps}}\bigg(\frac{\sqrt{\alpha}}{N}+\sqrt{ 1-\frac{\alpha}{N}} \bigg)\E\int_{0}^{T}  Z_{(N)}(s)^{p}\ds.
\end{align*}
We choose $\widetilde{\eps}<\frac{1}{\widetilde{C}_{p,\alpha}+C_{p,\alpha})}$ and multiply both sides of the inequality by $\frac{1}{(1-\widetilde{\eps}(\widetilde{C}_{p,\alpha}+C_{p,\alpha}))}$.  Using either \eqref{eqn:EZp_estimate}, or Gronwall's inequality we obtain
\begin{align*}
   \E\sup_{t\in [0,T]} Z_{(N)}(t)^{p}&\leq C_{p,\beta,\alpha}\E Z_{(N)}(0)^{p}\exp{(C_{p,\alpha,\beta}T)}.
\end{align*}
The constant on the right-hand side can indeed be chosen independent of $N$, since the terms involving $N$ are at most of order $\frac{1}{\sqrt{N}}(\geq \frac{1}{N}$, for $N\geq 1$).

For the second assertion we simply use Jensen's inequality and the non-negativity of the initial condition
\begin{align*}
    \E Z_{(N)}(0)^{1+\zeta}=\E\left(\frac{1}{N}\sum_{i=1}^{N}X_{i}(0)\right)^{1+\zeta}\leq \frac{1}{N}\sum_{i=1}^{N}\E X_{i}(0)^{1+\zeta}.
\end{align*}

By \ref{A:A_3_X(0)_Z(0)}, this then yields a uniform bound.

\end{proof}
\begin{lemma}\label{lem:EsupXp_EXp_estimates}
Let $(X_{1},\dots,X_{N})$ be a (non-negative) solution of the system \eqref{eqn:X_Ni_SDE} and define $Z_{(N)}\coloneqq \frac{1}{N}\sum_{i=1}^{N}X_{i}$. Let $p\geq 1$, $\E X_{i}^{p}(0)<\infty$ and $\E Z_{N}^{p}(0)<\infty$, then the following estimate is satisfied,
\begin{align}
     \E \sup_{t\in [0,T]} X_{i}^{p}(t)\leq C_{T,p,\alpha,\beta}\E \left(X_{i}^{p}(0)+Z_{(N)}(0)^{p}\right).\label{eqn:EsupXp_estimate}
\end{align}
Further, we have
    \begin{align*}
        &\E\frac{1}{N}\sum_{i=1}^{N}\sup_{t\in[0,T]}X_{i}^{p}(t)\leq C_{p,\alpha,\beta, Z_{(N)}(0)}\E\frac{1}{N}\sum_{i=1}^{N}X_{i}^{p}(0).
\end{align*}
\end{lemma}
\begin{proof}
Ito's formula implies
\begin{align*}
   \E \sup_{t\in [0,T]} X_{i}^{p}(t)&=\E X_{i}^{p}(0)+p\beta  \E \sup_{t\in [0,T]}\int_{0}^{t}X_{i}^{p-1}(s)Z_{(N)}(s)\dt+p\sqrt{\alpha} \E \sup_{t\in [0,T]}\int_{0}^{t}X_{i}^{p-1}(s)\sqrt{X_{i}(s)}\sqrt{Z_{(N)}(s)}\dW_{s}^{i}\\
    &\phantom{xx}{}+p\sqrt{1-\frac{\alpha}{N}} \E \sup_{t\in [0,T]}\int_{0}^{t}X_{i}^{p}(s)\dW^{0}_{s}+\alpha\frac{p(p-1)}{2}\E \sup_{t\in [0,T]}\int_{0}^{t} X_{i}^{p-1}(s)Z_{(N)}(s)\ds\\
    &\phantom{xx}+\frac{p(p-1)}{2}\bigg(1-\frac{\alpha}{N}\bigg)\E \sup_{t\in [0,T]}\int_{0}^{t} X_{i}^{p}(s)\ds.
\end{align*}
Applying the BDG inequality to the stochastic terms,
\begin{align*}
    p\sqrt{\alpha} \E \sup_{t\in [0,T]}&\int_{0}^{t}X_{i}^{p-1}(s)\sqrt{X_{i}(s)}\sqrt{Z_{(N)}(s)}\dW_{s}^{i}\leq p\sqrt{\alpha} C \E\bigg(\int_{0}^{T}X_{i}^{2p-1}(s)Z_{(N)}(s)\ds \bigg)^{1/2}\\
    &\leq C_{p,\alpha}  \E\bigg(\sqrt{\sup_{t\in [0,T]}X_{i}^{p}(t)}\bigg(\int_{0}^{T}X_{i}^{p-1}(s)Z_{(N)}(s)\ds \bigg)^{1/2}\bigg).
\end{align*}
Applying Young's inequality twice, where we introduced the term $\sqrt{\frac{\widetilde{\eps}}{\widetilde{\eps}}}$ for a small constant $\widetilde{\eps}$ which will be chosen later, we obtain,
\begin{align*}
    &\leq C_{p,\alpha}  \E\bigg(\sqrt{\sup_{t\in [0,T]}X_{i}^{p}(t)}\bigg(\int_{0}^{T}X_{i}^{p}(s)+Z_{(N)}(s)^{p}\ds \bigg)^{1/2}\bigg)\\
    &\leq C_{p,\alpha}  \E\bigg(\widetilde{\eps}\sup_{t\in [0,T]}X_{i}^{p}(t) +C_{\widetilde{\eps}}\int_{0}^{T}X_{i}^{p}(s)+Z_{(N)}(s)^{p}\ds \bigg)\\
    &= C_{p,\alpha}  \widetilde{\eps}\E\sup_{t\in [0,T]}X_{i}^{p}(t) +C_{\widetilde{\eps},p,\alpha}\E\int_{0}^{T}X_{i}^{p}(s)\ds +C_{\widetilde{\eps},p,\alpha}\E\int_{0}^{T}Z_{(N)}(s)^{p}\ds.
\end{align*}
Similarly,
\begin{align*}
    \E \sup_{t\in [0,T]}&\int_{0}^{t}X_{i}^{p}(s)\dW^{0}_{s}\leq C \E\bigg(\int_{0}^{T}X_{i}^{2p}(s)\ds \bigg)^{1/2}\\
    &\leq \widetilde{C}_{p,\alpha}  \E\bigg(\sqrt{\sup_{t\in [0,T]}X_{i}^{p}(t)}\bigg(\int_{0}^{T}X_{i}^{p}(s)\ds \bigg)^{1/2}\bigg)\\
    &\leq  \widetilde{C} \widetilde{\eps}\E\sup_{t\in [0,T]}X_{i}^{p}(t) +\widetilde{C}_{\widetilde{\eps}}\E\int_{0}^{T}X_{i}^{p}(s)\ds.
\end{align*}
Hence, choosing $\widetilde{\eps}<\frac{1}{\widetilde{C}_{p,\alpha}+C_{p,\alpha}}$, bringing the terms involving a supremum in time of $X_{i}$ from the right to the left-hand-side and multiplying both sides by $\frac{1}{(1-\widetilde{\eps}(\widetilde{C}_{p,\alpha}+C_{p,\alpha}))}$ yields
\begin{align*}
     \E \sup_{t\in [0,T]} X_{i}^{p}(t)&\leq C_{\widetilde{\eps},p,\alpha}\E X_{i}^{p}(0)+C_{\widetilde{\eps},p,\alpha,\beta}\E\int_{0}^{T}X_{i}^{p}(s)\ds +C_{\widetilde{\eps},p,\alpha,\beta}\E\int_{0}^{T}Z_{(N)}(s)^{p}\ds\\
     &\leq C_{\widetilde{\eps},T,p,\alpha,\beta}\E \left(X_{i}^{p}(0)+Z_{(N)}(0)^{p}\right).
\end{align*}
 The last claim follows directly from \eqref{eqn:EsupXp_estimate} and the triangle inequality.
\end{proof}
\subsection{Appendix D: Estimates for the McKean-Vlasov SDE \texorpdfstring{ \eqref{eqn:MKV_SDE_introduction} }{}}
\begin{lemma}\label{lem:estimates_Y}
Let $p\geq 2$ and let $Y$ denote the solution to \eqref{eqn:MKV_SDE_introduction} with initial condition $Y_{0}$, such that $\E[Y_{0}^{p}]<\infty$. Then
\begin{align*}
    \E\sup_{t\in[s,T]}\left|Y(t)\right|^{p}\leq C_{T,p,\alpha,\beta} \E\left|Y(s)\right|^{p},
\end{align*}
\begin{align*}
     \E\left|Y(t)-Y(s)\right|^{p}\leq C_{p,\alpha,\beta} |t-s|^{p/2}\left(|t-s|^{p/2}+1\right)\E\sup_{r\in[0,T]}Y(r)^{p}\leq \widetilde{C}_{T,p,\alpha,\beta} |t-s|^{p/2}\left(|t-s|^{p/2}+1\right)\E Y(0)^{p}.
\end{align*}
\end{lemma}
\begin{proof}
\begin{align*}
    \E\sup_{t\in[s,T]}\left|Y(t)\right|^{p}&=\E\left|Y(s)\right|^{p}+p\beta\E\sup_{t\in[s,T]}\int_{s}^{t} Y(r)^{p-1}\E\left[Y(r)\big|\mathcal{W}^{0}_{r}\right]\dr\\
    &\phantom{xx}{}+p\sqrt{\alpha}\E\sup_{t\in[s,T]}\int_{s}^{t}Y(r)^{p-1}\sqrt{Y(r)\E\left[Y(r)\big|\mathcal{W}^{0}_{r}\right]}\dB_{r}+p\E\sup_{t\in[s,T]}\int_{s}^{t}Y(r)^{p}\dW^{0}_{r}\\
    &\phantom{xx}{}+p(p-1)\E\sup_{t\in[s,T]}\int_{s}^{t}\alpha Y(r)^{p-1}\E\left[Y(r)\big|\mathcal{W}^{0}_{r}\right]+Y(r)^{p}\dr\\
    &\leq \E\left|Y(s)\right|^{p}+p\beta\int_{s}^{T} \left(\E Y(r)^{p}\right)^{\frac{p-1}{p}}\left(\E\E\left[Y(r)\big|\mathcal{W}^{0}_{r}\right]^{p}\right)^{1/p}\dr\\
    &\phantom{xx}{}+p\sqrt{\alpha}\E\left(\int_{s}^{T} \left|Y(r)^{2p-1}\E\left[Y(r)\big|\mathcal{W}^{0}_{r}\right]\right|\dr\right)^{1/2}+p\E\left(\int_{s}^{T} Y(r)^{2p}\dr\right)^{1/2}\\
    &\phantom{xx}{}+p(p-1)\int_{s}^{T}\alpha\left(\E Y(r)^{p}\right)^{\frac{p-1}{p}}\left(\E\E\left[Y(r)\big|\mathcal{W}^{0}_{r}\right]^{p}\right)^{1/p}+\E Y(r)^{p}\dr\\
    &\leq \E\left|Y(s)\right|^{p}+p\beta\int_{s}^{T} \left(\E Y(r)^{p}\right)^{\frac{p-1}{p}}\left(\E\E\left[Y(r)\big|\mathcal{W}^{0}_{r}\right]^{p}\right)^{1/p}\dr\\
    &\phantom{xx}{}+p\sqrt{\alpha}\E\left(\sup_{r\in[s,T]}|Y(s)|^{\frac{p}{2}}\left(\int_{s}^{T} Y(r)^{p-1}\E\left[Y(r)\big|\mathcal{W}^{0}_{r}\right]\dr\right)^{1/2}\right)+p\E\left(\sup_{r\in[s,T]}|Y(r)|^{\frac{p}{2}}\left(\int_{s}^{T} Y(r)^{p}\dr\right)^{1/2}\right)\\
    &\phantom{xx}{}+p(p-1)\int_{s}^{T}\alpha\left(\E Y(r)^{p}\right)^{\frac{p-1}{p}}\left(\E\E\left[Y(r)\big|\mathcal{W}^{0}_{r}\right]^{p}\right)^{1/p}+\E Y(r)^{p}\dr\\
    &\leq \E\left|Y(s)\right|^{p}+p\beta\int_{s}^{T} \left(\E Y(r)^{p}\right)^{\frac{p-1}{p}}\left(\E\E\left[Y(r)\big|\mathcal{W}^{0}_{r}\right]^{p}\right)^{1/p}\dr\\
    &\phantom{xx}{}+p\sqrt{\alpha}\left(\E\sup_{r\in[s,T]}Y(r)^{p}\right)^{1/2}\left(\E\int_{s}^{T} Y(r)^{p-1}\E\left[Y(r)\big|\mathcal{W}^{0}_{r}\right]\dr\right)^{1/2}\\
    &\phantom{xx}+p\left(\E\sup_{r\in[s,T]}Y(r)^{p}\right)^{1/2}\left(\E\int_{s}^{T} Y(r)^{p}\dr\right)^{1/2}\\
    &\phantom{xx}{}+\int_{s}^{T}\alpha\left(\E Y(r)^{p}\right)^{\frac{p-1}{p}}\left(\E\E\left[Y(r)\big|\mathcal{W}^{0}_{r}\right]^{p}\right)^{1/p}+\E Y(r)^{p}\dr.
    \end{align*}
    We introduce the term $\frac{\widetilde{\eps}}{\widetilde{\eps}}$, for a small constant $\widetilde{\eps}$, which will be chosen later. This constant will help us to absorb the terms involving $\sup_{r\in [s,T]}Y(r)^{p}$ into the left-hand side. Young's inequality now yields
    \begin{align*}
    \E\sup_{t\in[s,T]}\left|Y(t)\right|^{p}&\leq \E\left|Y(s)\right|^{p}+C_{\beta}\int_{s}^{T} \E Y(r)^{p}+\E\E\left[Y(r)\big|\mathcal{W}^{0}_{r}\right]^{p}\dr\\
    &\phantom{xx}{}+C_{p,\alpha}\widetilde{\eps}\E\sup_{r\in[s,T]}Y(s)^{p}+C_{\widetilde{\eps}}\E\int_{s}^{T} \E Y(r)^{p}+\E\E\left[Y(r)\big|\mathcal{W}^{0}_{r}\right]^{p}\dr\\
    &\phantom{xx}{}+C_{\widetilde{\eps},p,\alpha}\int_{s}^{T} \E Y(r)^{p}\dr\\
    &\phantom{xx}{}+C_{p,\alpha}\int_{s}^{T}\E Y(r)^{p}+\E\E\left[Y(r)\big|\mathcal{W}^{0}_{r}\right]^{p}+\E Y(r)^{p}\dr\\
    &\leq \E\left|Y(s)\right|^{p}+C_{\widetilde{\eps},p,\alpha,\beta}\int_{s}^{T} \E Y(r)^{p}\dr+\widetilde{\eps} C_{p,\alpha,\beta}\E\sup_{r\in[s,T]}Y(r)^{p}\\
    &\leq \E\left|Y(s)\right|^{p}+C_{\widetilde{\eps},p,\alpha,\beta}\int_{s}^{T} \E \sup_{u\in[s,r]}Y(u)^{p}\dr+\widetilde{\eps} C_{p,\alpha,\beta}\E\sup_{r\in[s,T]}Y(r)^{p}.
\end{align*}
Choosing $\widetilde{\eps}<\frac{1}{C_{p,\alpha,\beta}}$, rearranging terms, multiplying both sides with $\frac{1}{(1-\widetilde{\eps}C_{p,\alpha,\beta})}$ and Gronwall's inequality yield that 
\begin{align*}
     \E\sup_{t\in[s,T]}\left|Y(t)\right|^{p}\leq \widetilde{C}_{\widetilde{\eps},p,\alpha,\beta} e^{\widetilde{C}_{\widetilde{\eps},p,\alpha,\beta}\left|T-s\right|}\E\left|Y(s)\right|^{p}.
\end{align*}
By similar estimates, we obtain
\begin{align*}
    \E\left|Y(t)-Y(s)\right|^{p}&\leq C \E\left(\int_{s}^{t} \E\left[Y(r)\big|\mathcal{W}^{0}_{r}\right]\dr\right)^{p}+C\E\left(\int_{s}^{t}\sqrt{Y(r)\E\left[Y(r)\big|\mathcal{W}^{0}_{r}\right]}\dW_{r}\right)^{p}\\
    &\phantom{xx}{}+C\E\left(\int_{s}^{t}Y(r)\dW^{0}_{r}\right)^{p}\\
    &\leq C|t-s|^{p-1}\int_{s}^{t} \E\E\left[Y(r)\big|\mathcal{W}^{0}_{r}\right]^{p}\dr\\
    &\phantom{xx}{}+C\E\left(\int_{s}^{t}\left(Y(r)\E\left[Y(r)\big|\mathcal{W}^{0}_{r}\right]\right)\dr\right)^{p/2}\\
    &\phantom{xx}{}+C\E\left(\int_{s}^{t} Y(r)^{2}\dr\right)^{p/2}\\
    &\leq C|t-s|^{p}\E\sup_{r\in[s,t]}Y(r)^{p}+C|t-s|^{p/2}\E\sup_{r\in[s,t]}Y(r)^{p}\\
    &\leq C|t-s|^{p}\E\sup_{r\in[s,t]}Y(r)^{p}+C|t-s|^{p/2}\E\sup_{r\in[s,t]}Y(r)^{p}\\
    &\leq C|t-s|^{p/2}\left(1+|t-s|^{p/2}\right)\E\sup_{r\in[s,t]}Y(r)^{p}.
\end{align*}
The previously obtained bound on $\E\sup_{t\in[s,T]}\left|Y(t)\right|^{p}$ now yields the result.
\end{proof}
\begin{lemma}\label{lem:Y_density_error_estimate}
Let $Y$ be the solution of the McKean-Vlasov SDE  \eqref{eqn:MKV_SDE_introduction} with initial condition $Y_{0}$ satisfying $\E[Y_{0}^{p}]\leq C$ for a finite constant $C>0$ and $p\geq 2$. Further let $Y^{\eps}$ be given by \eqref{eqn:Y_approximation}. Then we have the following error estimate,
    \begin{align*}
        \E\left[\left|Y(t)-Y^{\eps}(t)\right|^{p}\right] \leq C\eps^{\frac{3}{4}p}.
    \end{align*}
\end{lemma}
\begin{proof}
Recall that $\E\left[Y(t-\eps)\big|\mathcal{W}^{0}_{t-\eps}\right]=\E\left[Y(t-\eps)\big|\mathcal{W}^{0}_{r}\right]$ for any $r\geq t-\eps$.
    \begin{align*}
    \E\left[\left|Y(t)-Y^{\eps}(t)\right|^{p}\right]&\leq C_{p} \E\left[\left|\int_{t-\eps}^{t}\E\left[Y(r)\big|\mathcal{W}^{0}_{r}\right]-\E\left[Y(t-\eps)\big|\mathcal{W}^{0}_{t-\eps}\right]\dr\right|^{p}\right]\\
    &\phantom{xx}{}+C_{p} \E\left[\left|\int_{t-\eps}^{t}\sqrt{Y(r)\E\left[Y(r)\big|\mathcal{W}^{0}_{r}\right]}-\sqrt{Y(t-\eps)\E\left[Y(t-\eps)\big|\mathcal{W}^{0}_{t-\eps}\right]}\dB_{r}\right|^{p}\right]\\
    &\phantom{xx}{}+C_{p} \E\left[\left|\int_{t-\eps}^{t}Y(r)-Y(t-\eps)\dW^{0}_{r}\right|^{p}\right]\\
    &\leq C_{p} |\eps|^{p-1}\E\left[\int_{t-\eps}^{t}\E\left[\left|Y(r)-Y(t-\eps)\right|^{p}\big|\mathcal{W}^{0}_{r}\right]\dr\right]\\
    &\phantom{xx}{}+C_{p} \E\left[\left|\int_{t-\eps}^{t}\left|Y(r)\E\left[Y(r)\big|\mathcal{W}^{0}_{r}\right]-Y(t-\eps)\E\left[Y(t-\eps)\big|\mathcal{W}^{0}_{t-\eps}\right]\right|\dr\right|^{p/2}\right]\\
    &\phantom{xx}{}+C_{p} \E\left[\left|\int_{t-\eps}^{t}\left|Y(r)-Y(t-\eps)\right|^{2}\dr\right|^{p/2}\right]\\
    &\leq |\eps|^{p-1}\E\left[\int_{t-\eps}^{t}\E\left[\left|Y(r)-Y(t-\eps)\right|^{p}\big|\mathcal{W}^{0}_{r}\right]\dr\right]\\
    &\phantom{xx}{}+C_{p} |\eps|^{\frac{p}{2}-1}\E\left[\int_{t-\eps}^{t}\left|Y(r)\E\left[Y(r)\big|\mathcal{W}^{0}_{r}\right]-Y(t-\eps)\E\left[Y(t-\eps)\big|\mathcal{W}^{0}_{t-\eps}\right]\right|^{p/2}\dr\right]\\
    &\phantom{xx}{}+C_{p}|\eps|^{\frac{p}{2}-1} \E\left[\int_{t-\eps}^{t}\left|Y(r)-Y(t-\eps)\right|^{p}\dr\right]\\
        &\leq |\eps|^{p-1}\E\left[\int_{t-\eps}^{t}\E\left[\left|Y(r)-Y(t-\eps)\right|^{p}\big|\mathcal{W}^{0}_{r}\right]\dr\right]\\
    &\phantom{xx}{}+C_{p}|\eps|^{\frac{p}{2}-1} \int_{t-\eps}^{t}\left(\E\left|Y(r)\right|^{p}\right)^{1/2}\left(\E\E\left[\left|Y(r)-Y(t-\eps)\right|^{p}\big|\mathcal{W}^{0}_{r}\right]\right)^{1/2}\dr\\
    &\phantom{xx}{}+C_{p}|\eps|^{\frac{p}{2}-1} \int_{t-\eps}^{t}\left(\E\left|\E\left[Y(t-\eps)^{p}\big|\mathcal{W}^{0}_{t-\eps}\right]\right|\right)^{1/2}\left(\E\left|Y(r)-Y(t-\eps)\right|^{p}\right)^{1/2}\dr\\
    &\phantom{xx}{}+C_{p}|\eps|^{\frac{p}{2}-1} \E\left[\int_{t-\eps}^{t}\left|Y(r)-Y(t-\eps)\right|^{p}\dr\right].
\end{align*}
The result now follows from Lemma \ref{lem:estimates_Y}.
\end{proof}
\subsection{Appendix E: Tightness criteria}
Let $H$ be a Polish space.
We recall the requirements for a subset of $(M_{1}(H), \Wd)$, the space of probability  measures with finite first moment equipped with the Wasserstein distance,
   to be relatively-compact (see for example \cite[Proposition 2.2.3]{panaretos_2020_statistics_in_Wasserstein_space}). 
\begin{lemma}\label{lem:compactness_in_M1}
   A set $S$ is relatively compact in $(M_{1}(H),\Wd)$, if
\begin{enumerate}[label=(\Roman*)]
    \item \label{compactness_M1_C1} it is tight, i.e.~for every $\eps>0$ there exists a compact set $K_{\eps}$ such that
    \begin{align*}
         \mu(K_{\eps})>1-\eps,
    \end{align*}
    for every $\mu\in S$,
   \item \label{compactness_M1_C2} and 1-uniformly integrable: For every $\eps>0$, there exists an $R >0$ and $x_{0}$, such that
    \begin{align}\label{eqn:uniform_1_integrability_condition}
        \sup_{\mu\in S}\int_{\{x\in H \backslash B_{R}(x_{0})\}}d(x,x_{0}) \mu (\dx)<\eps,
    \end{align}
    where $B_{R}(x_0)$ denotes the metric ball with radius $R$ around $x_{0}$.
\end{enumerate}
\end{lemma}
Recall also that a sequence of measures converges in the Wasserstein distance, iff
\begin{align*}
\mu_{n}\rightharpoonup \mu \quad \text{and} \quad 
    \int d(x,x_{0}) \mu_{n}(\dx)\rightarrow \int d(x,x_{0}) \mu(\dx),
\end{align*}
for some $x_{0}\in H$, where $\mu_{n}\rightharpoonup \mu$ denotes  weak convergence of measures, see \cite[Definition 6.8]{villani_09_optimal_transport_old_and_new} for equivalent characterizations. 
Since in the paper, we work with $H=\R_{+}$ which possesses the Heine-Borel property, the first tightness condition follows directly from \eqref{eqn:uniform_1_integrability_condition}.
\begin{theorem}\label{thm:tightness_ethier_kurtz}
    Let $E$ be a Polish space. Suppose $\left\{u_n(\cdot)\right\}_{n \geq 1}$ is a sequence of continuous E-valued processes, defined on the same stochastic basis $(\Omega,\mathcal{F},(\mathcal{F}_{t})_{t},\Prob)$. Then the laws of the family $\left\{u_n(\cdot)\right\}_{n \geq 1}$ are tight on $C([0, \infty), E)$ if and only if the following two conditions hold:
    \begin{enumerate}[label=(\Roman*)]
        \item\label{item:tightness_ek_tightness} For every $ t >0$ and every $\varepsilon>0$ there exists a compact set $K_{t, \varepsilon}^0 \subseteq E$ such 
        that
        \begin{align*}
            \sup _n \Prob\left(u_n(t) \notin K_{t, \varepsilon}^0 \right) \leq \varepsilon.
        \end{align*}
        
        \item \label{item:tightness_Aldous_condition} For every $\eta, \eps>0$, there exists a $\delta>0$ such that for every $n$ and every $(\mathcal{F}_{t})$--stopping time 
        \begin{align*}
          \sup_{0\leq \theta \leq \delta}\Prob\left(d\left(u_n(\tau+\theta), u_n(\tau)\right) >\eta\right) \leq \eps.
        \end{align*}
    \end{enumerate}
\end{theorem}
This theorem follows from of \cite[Corollary 3.7.4 and Theorem 3.10.2]{ethier_kurtz_09_markov} and \cite[Theorem 16.9]{billingsley_99_convergence_of_probability_measures}. Condition \ref{item:tightness_Aldous_condition} is also referred to as Aldous's condition (see \cite{aldous_78_stopping_and_tightness_1,aldous_89_stopping_and_tightness_2} as well as \cite{jacod_83_tightness_and_stopping_times}). 

\subsection{Appendix F: Construction of the functions used in Lemma \ref{lem:rho_x}}\label{sec:contruction_of_test_function}
    We give an explicit construction of the sequence of smooth functions used in Lemma \ref{lem:rho_x}. Although the construction is standard, it might not be obvious that such a function exists.
    Let $A\subseteq \R$ and
    \begin{align*}
        \varphi(t)= \begin{cases}e^{-\frac{1}{t}} & \text { if } t>0 \\ 0 & \text { otherwise }\end{cases}
    \end{align*}
This function is $C^{\infty}$, and $0<\varphi(t)$ iff $0<t$.
Define $\psi$ as
\begin{align*}
    \psi(x):=k \varphi\left(1-\|x\|^2\right),
\end{align*}
where $k$ can be chosen so that $\int \psi(x) \dx=1$, if so desired. $\psi$ is $C^{\infty}$, further $0 \leq \psi$, and in particular $0<\psi(x)$ iff $\|x\|^2<1$.
Let $\delta>0$ and set 
\begin{align*}
    \psi_\delta(x)=\frac{1}{\delta} \psi\left(\frac{x}{\delta}\right)
\end{align*}
and note that $0<\psi_\delta(x)$ iff $\|x\|<\delta$ and $\int \rho_\delta=1$.
Let $\dom=A+B_{\delta}(0)$, where $B_{\delta}(0)$ denotes the ball with radius $\delta$ around $0$ and consider the function
\begin{align*}
    \psi_{A,\delta}=\indicator_{\dom} * \psi_\delta .
\end{align*}
Clearly $0 \leq \psi_{A,\delta} \leq 1$ and $\psi_{A,\delta}$ is $C^{\infty}$.
If $y \in A$, then $B_{\delta}(y) \subseteq A+B_{\delta}(0)=\dom$, hence $\left(\indicator_{\dom} * \psi_\delta\right)(y)=1$.
If $y \notin A+B_{2\delta}(0)$ then $B_{\delta}(y) \cap\left(A+B_{\delta}(0)\right)=\emptyset$, so $\left(\indicator_{\dom} * \psi_\delta\right)(y)=0$.
For derivatives of any order we obtain
\begin{align*}
  D^\nu \psi_{A,\delta}(x)=D^\nu \int_{\dom} \frac{1}{\delta} \psi\left(\frac{x-y}{\delta}\right) \dy=\frac{1}{\delta^{|\nu|+1}} \int_{\dom} \left(D^\nu \psi\right)\left(\frac{x-y}{\delta}\right) \dy,  
\end{align*}
which yields that
\begin{align*}
    \left|D^\nu \psi_{A,\delta}(x)\right| \leq\left(\int\left|\left(D^\nu \psi\right)(y)\right| \dy\right) \frac{1}{\delta^{|\nu|}} .
\end{align*}
In Lemma \ref{lem:rho_x} for $R\in\R,\,R>0$, $A=[-R,R]$ and $\delta=4$ will be chosen. It is important to note that we do not want to approximate $\indicator_{A}$, but rather the constant $1$ function pointwise. Hence we are not interested in letting $\delta\rightarrow 0$ or maintaining the $L^{1}$ norm of our test functions.

 \bibliographystyle{abbrv}
 \bibliography{references}
\end{document}